\documentclass[a4paper]{amsart}
\usepackage[english]{babel}
\usepackage{amssymb}
\usepackage[initials]{amsrefs}
\usepackage{tikz}

\usepackage[arrow,matrix,curve]{xy}

\usepackage{ifthen}
\newboolean{drawtikz}
\setboolean{drawtikz}{true}

\allowdisplaybreaks

\let\oldtocsection=\tocsection
\let\oldtocsubsection=\tocsubsection
\let\oldtocsubsubsection=\tocsubsubsection
\renewcommand{\tocsection}[2]{\hspace{0em}\oldtocsection{#1}{#2}}
\renewcommand{\tocsubsection}[2]{\hspace{1em}\oldtocsubsection{#1}{#2}}
\renewcommand{\tocsubsubsection}[2]{\hspace{2em}\oldtocsubsubsection{#1}{#2}}

\newcommand{\into}{{\triangleright}}
\newcommand{\downa}{{\downarrow}}
\newcommand{\op}{\operatorname}
\newcommand{\co}{\colon}
\newcommand{\id}{\mathrm{id}}
\newcommand{\NN}{\mathbb{N}}

\newcommand{\RR}{\mathbb{R}}
\newcommand{\ZZ}{\mathbb{Z}}

\newcommand{\calA}{\mathcal{A}}
\newcommand{\calB}{\mathcal{B}}
\newcommand{\calC}{\mathcal{C}}
\newcommand{\calD}{\mathcal{D}}
\newcommand{\calE}{\mathcal{E}}
\newcommand{\calF}{\mathcal{F}}
\newcommand{\calG}{\mathcal{G}}

\newcommand{\calI}{\mathcal{I}}
\newcommand{\calJ}{\mathcal{J}}
\newcommand{\calK}{\mathcal{K}}
\newcommand{\calL}{\mathcal{L}}
\newcommand{\calM}{\mathcal{M}}

\newcommand{\calO}{\mathcal{O}}
\newcommand{\calP}{\mathcal{P}}
\newcommand{\calQ}{\mathcal{Q}}

\newcommand{\calS}{\mathcal{S}}
\newcommand{\calT}{\mathcal{T}}
\newcommand{\calU}{\mathcal{U}}
\newcommand{\calV}{\mathcal{V}}

\newcommand{\calX}{\mathcal{X}}
\newcommand{\calY}{\mathcal{Y}}
\newcommand{\calZ}{\mathcal{Z}}

\newtheorem{thm}{Theorem}[section]
\newtheorem*{thm*}{Theorem}
\newtheorem{lem}[thm]{Lemma}
\newtheorem{cor}[thm]{Corollary}
\newtheorem{prop}[thm]{Proposition}
\theoremstyle{definition}
\newtheorem{defi}[thm]{Definition}

\newtheorem{cons}[thm]{Construction}
\theoremstyle{remark}
\newtheorem{rem}[thm]{Remark}
\newtheorem{obs}[thm]{Observation}
\newtheorem{que}[thm]{Question}

\title{Operad groups and their finiteness properties}
\author{Werner Thumann}
\address{Karlsruhe Institute of Technology, Karlsruhe, Germany}

\begin{document}

\begin{abstract}
	We propose a new unifying framework for Thompson-like groups using a well-known device called operads and category
	theory as language. We discuss examples of operad groups which have appeared in the literature before. As a first
	application, we proof a theorem which implies that planar or symmetric or braided operads with transformations 
	satisfying some finiteness conditions yield operad groups of type $F_\infty$. This unifies and extends existing 
	proofs that certain Thompson-like groups are of type $F_\infty$.
\end{abstract}
\maketitle
\tableofcontents

\section{Introduction}

In unpublished notes of 1965, Richard Thompson defined three interesting groups $F,T,V$. For example, $F$ is the group of all
orientation preserving piecewise linear homeomorphisms of the unit interval with breakpoints lying in the dyadic rationals and with slopes being
powers of $2$. It has the presentation
\[F=\big\langle x_0,x_1,x_2,\ldots\mid x_k^{-1}x_n x_k=x_{n+1}\text{ for }k<n \big\rangle\]
In the subsequent years until the present days,
hundreds of papers have been devoted to these and to related groups. The reason for this is that they have the
ability to unite seemingly incompatible properties. For example, Thompson showed that $V$ is an infinite finitely-presented
simple group which contains every finite group as a subgroup. Even more is true:
Brown showed in \cite{bro:fpo} that $V$ is of type $F_\infty$ which means that there is a classifying space for $V$ with
finitely many cells in every dimension. For $F$, this was proven by
Brown and Geoghegan in \cite{br-ge:ait}.
They also showed that $H^k(F,\ZZ F)=0$ for every $k\geq 0$. This implies in particular that all homotopy groups of $F$ at infinity
vanish and that $F$ has infinite cohomological dimension. Thus, they found the first example of an infinite dimensional torsion-free 
group of type $F_\infty$. In \cite{br-sq:gop}, Brin and Squier showed that $F$ is a free group free group, i.e.~contains no non-abelian free
subgroups. Geoghegan conjectured in 1979 that $F$ is non-amenable. If this is true, $F$ would be an elegant counterexample to the von Neumann
conjecture. Ol'shanskii disproved the von Neumann conjecture around 1980 by giving a different counterexample (see \cite{mon:gop} and the references 
therein). Despite several attempts of various authors, the amenability question for $F$ still seems to be open at the time of writing.
During the 1970s, Thompson's group $F$ was rediscovered twice: In the context of homotopy theory by Freyd and Heller \cite{fr-he:shi} and
in connection with a problem in shape theory by Dydak \cite{dyd:asp}.

Since the introduction of the classical Thompson groups $F,T$ and $V$, a lot of generalizations have appeared in the literature 
which have a ``Thompson-esque'' feeling to them. Among them are the so-called diagram or picture groups \cite{gu-sa:dg},
various groups of piecewise linear homeomorphisms of the unit interval \cite{ste:gop}, groups acting on ultrametric spaces
via local similarities \cite{hug:lsa}, higher dimensional Thompson groups $nV$ \cite{bri:hdt} and the braided Thompson group
$BV$ \cite{bri:tao}. A recurrent theme in the study of these groups are topological finiteness properties, most notably
property $F_\infty$. The proof of this property
is very similar in each case, going back to a method of Brown, the Brown criterion \cite{bro:fpo}, and a technique of Bestvina
and Brady, the discrete Morse Lemma for affine complexes \cite{be-br:mta}. This program has been conducted in all the above
mentioned classes of groups: For diagram or picture groups in \cites{far:fac,far:haf}, for the piecewise linear
homeomorphisms in \cite{ste:gop}, for local similarity groups in \cite{fa-hu:fpo}, for the higher dimensional Thompson groups 
in \cite{f-m-w-z:tbg} and for the braided Thompson group in \cite{b-f-m-w-z:tbt}.

The main motivation to define the class of operad groups, which are the central objects in this article, was to find a framework 
in which a lot of the Thompson-like groups could be recovered and in which the established techniques 
could be performed to show property $F_\infty$, thus unifying and extending existing proofs in the literature.
The main device to define these groups are discrete operads. Operads are well established objects whose importance 
in mathematics and physics has steadily increased during the last decades. Representations of operads constitute algebras of various 
types and consequently find applications in such diverse areas as Lie-Theory, Noncommutative Geometry, Algebraic Topology, Differential 
Geometry, Field Theories and many more. To apply our $F_\infty$ theorem to a given Thompson-like group, one has to find the operadic 
structure underlying the group. Then one has to check whether this operad satisfies certain finiteness conditions. In a lot of cases, the
proofs of these conditions are either trivial or straightforward.

\subsection{Structure of the article}

Our language will be strongly category theory flavoured. Although we assume the basics of category theory, we collect and recall
in Section \ref{31364} all the tools we will need for the definition of operad groups and for our main result. We lay
a particular emphasis on topological aspects of categories by considering categories as topological objects via the nerve functor.
This can be made precise by endowing the category of (small) categories with a model structure Quillen equivalent to the usual
homotopy category of spaces, but we won't use this fact. In Subsection \ref{00484}, we will discuss a tool which
is probably not so well-known as the others. There, we introduce the discrete Morse method for categories in analogy to the
one for simplicial complexes: With the help of a Morse function, a category can be filtered by a nested sequence of full 
subcategories. The relative connectivity of such a filtration is controlled by the connectivity of certain categories associated to each
filtration step, the so-called descending links. This can be used to compute lower bounds for the connectivity of categories.

In Section \ref{97985}, we will introduce the main objects of this article, the so-called operad groups. Before we do this, we recall
the notion of operads (internal to the category of sets). This is an abstract algebraic structure generalizing that of a monoid.
It comes with an associative multiplication and with identity elements. However, elements in an operad, which are called operations,
can be of higher arity (or degree): An operation posseses several inputs and one output. If we have an operation with $n$ inputs, then we can
plug the outputs of $n$ other operations into the inputs of the first one, yielding composition maps for the operad. This concept can 
be generalized even more: Just as one proceeds from monoids to categories by introducing further objects, we can introduce colors to operads 
and label the inputs and outputs of operations with these colors. Then we require that the composition maps respect this coloring. Furthermore, we
can introduce actions of the symmetric or braid groups on the inputs of the operations and obtain symmetric or braided operads.

We then attach, in a very natural way, a category to each operad, called the category of operators. When taking fundamental groups
of these categories, we arrive at the concept of operad groups. In Subsection \ref{13887}, we then discuss
some examples of operads and corresponding operad groups. We will see that all of the Thompson-like groups mentioned in the first part of the
introduction can be realized as operad groups. Furthermore, we give new examples and even a procedure how to generate a lot of
these Thompson-like groups as operad groups associated to suboperads of endomorphism operads.

We will also discuss so-called operads with transformations which are operads with invertible degree $1$ operations. In this context,
we will introduce very elementary and elementary operations. These model in some sense the generators and relations in such an operad
with transformations. In particular, we can define what it means for such an operad to be finitely generated or of finite type.
This will be important in Section \ref{17057} where we prove the following
\begin{thm*}
	Let $\calO$ be a finite type (symmetric/braided) operad with transformations which is color-tame and such that there are only 
	finitely many colors and degree $1$ operations. Assume further that $\calO$ satisfies the cancellative calculus of 
	fractions. Then the operad groups associated to $\calO$ are of type $F_\infty$. 
\end{thm*}
The conditions are explained in the text and are usually not hard to verify in practice. The proof proceeds roughly as follows and the
ideas are mainly inspired by \cites{bro:fpo,be-br:mta,f-m-w-z:tbg,b-f-m-w-z:tbt,ste:gop}. Denote by $\calS$ the category of operators
of $\calO$. We then can look at the universal covering category $\calU$ of $\calS$ which is contractible due to the conditions in the
theorem. We mod out the isomorphisms in $\calU$ and obtain the quotient category $\calU/\calG$ which is still contractible. The operad group 
$\Gamma$, which is the fundamental group of $\calS$, acts on $\calU$ by deck transformations. This induces an action on $\calU/\calG$.
Brown's criterion applied to this action yields that $\Gamma$ is of type $F_\infty$ if we show that the isotropy groups of the action 
are of type $F_\infty$ and if we find a filtration by invariant finity type subcategories with relative connectivity tending to infinity. The latter
is shown by appealing to the discrete Morse method for categories mentioned earlier. Thus, we have to inspect the connectivity of certain descending
links. This is the hardest part of the proof. We filter each descending link by two subcategories, the core and the corona.
The core is related to certain arc complexes in $\RR^d$ with $d=1,2,3$. A lower bound for the connectivity of these complexes is given in
Theorem \ref{85504}. The connectivity for the corona and for the whole descending link is then deduced from the connectivity of
the core by using again the discrete Morse method.

\subsection{Notation and conventions}

When $f\co A\rightarrow B$ and $g\co B\rightarrow C$ are two composable arrows, we write $f*g$ or $fg$ for the composite $A\rightarrow C$
instead of the usual notation $g\circ f$. Consequently, it is often better to plug in arguments from the left. When we do this, 
we use the notation $x\into f$ for the evaluation of $f$ at $x$. However, we won't entirely drop the usual notation $f(x)$ and use 
both notations side by side. Objects of type $\operatorname{Aut}(X)$ will be made into a group by the definition $f\cdot g:=f*g$.
Conversely, a group $G$ is considered as a groupoid with one object and arrows the elements in $G$ together with the composition
$f*g:=f\cdot g$.

\subsection{Acknowledgements}

I want to thank my adviser Roman Sauer for the opportunity to pursue mathematics, for his guidance, encouragement and
support over the last few years. I also gratefully acknowledge financial support by the DFG grants 1661/3-1 and 1661/3-2.

\section{Preliminaries on categories}\label{31364}

In this section, we review some aspects of category theory which we need for later considerations. In particular, we want to
emphasize the concept of seeing categories as topological objects. Note that everything, except the Morse method for categories
explained in Subsection \ref{00484}, should be mathematical folklore and we make no claim of originality.

\subsection{Comma categories}

Let $\calA\xrightarrow{f}\calC\xleftarrow{g}\calB$ be two functors. Then the comma category $f\downa g$ has as objects all the 
triples $(A,B,\gamma)$ where $A$ resp.~$B$ is an object in $\calA$ resp.~$\calB$ and $\gamma\co f(A)\rightarrow g(B)$ is an arrow in 
$\calC$. An arrow from $(A,B,\gamma)$ to $(A',B',\gamma')$ is a pair $(\alpha,\beta)$ of arrows $\alpha\co A\rightarrow A'$ in $\calA$ 
and $\beta\co B\rightarrow B'$ in $\calB$ such that the diagram
\begin{displaymath}\xymatrix{
	f(A)\ar[r]^{\gamma}\ar[d]_{f(\alpha)} & g(B)\ar[d]^{g(\beta)}\\
	f(A')\ar[r]_{\gamma'} & g(B')
}\end{displaymath}
commutes. Composition is given by composing the components.

If $f$ is the inclusion of a subcategory, we write $\calA\downa g$ for the comma category $f\downa g$. Furthermore, if $\calA$ 
is just a subcategory with one object $A$ and its identity arrow, we write $A\downa g$. In this case, the objects of the comma 
category are pairs $(B,\gamma)$ where $B$ is an object in $\calB$ and $\gamma\co A\rightarrow g(B)$ is an arrow. An arrow from 
$(B,\gamma)$ to $(B',\gamma')$ is an arrow $\beta\co B\rightarrow B'$ such that the triangle 
\begin{displaymath}\xymatrix{
	& g(B)\ar[dd]^{g(\beta)}\\
	A\ar[ur]^{\gamma}\ar[dr]_{\gamma'}&\\
	& g(B')
}\end{displaymath}
commutes. Of course, there are analogous abbreviations for the right factor. 

\subsection{The classifying space of a category}\label{53419}

We assume that the reader is familiar with the basics of simplicial sets (see e.g.~\cite{go-ja:sht}).
The nerve $N(\calC)$ of a category $\calC$ is a
simplicial set defined as follows: A $k$-simplex is a sequence
\[A_0\xrightarrow{\alpha_0}A_1\xrightarrow{\alpha_1}\ldots\xrightarrow{\alpha_{k-1}}A_k\]
of $k$ composable arrows. The $i$'th face map $d_i\co N(\calC)_k\rightarrow N(\calC)_{k-1}$ is given by composing the arrows
at the object $A_i$. When $i$ is $0$ or $k$, then the object $A_i$ is removed from the sequence instead. The $i$'th
degeneracy map $s_i\co N(\calC)_k\rightarrow N(\calC)_{k+1}$ is given by inserting the identity at the object $A_i$.

The geometric realization $|N(\calC)|$ of $N(\calC)$ is a CW-complex which we call the classifying space $B(\calC)$ of 
$\calC$. See \cite{wei:wdt} for the reason why this is called a classifying space. If the category $\calC$ is a group, then 
$B(\calC)$ is the usual classifying space of the group which is defined as the unique space (up to homotopy equivalence) with 
fundamental group the given group and with higher homotopy groups vanishing.

Since we can view any category as a space via the above construction, any topological notion or concept can be
transported to the world of categories. For example, if we say that the category $\calC$ is connected, then
we mean that $B(\calC)$ is connected. Of course, one can easily think of an intrinsic definition of connectedness
for categories and we will give some for other topological concepts below. But there are also concepts for which
a combinatorial description is at least unknown, for example higher homotopy groups.

Transporting topological concepts to the category $\mathtt{CAT}$ of (small) categories via the nerve functor can be made 
precise: The Thomason model structure on $\mathtt{CAT}$ \cites{tho:caa,cis:lcd} is a model structure Quillen equivalent
to the usual model structure on $\mathtt{SSET}$, the category of simplicial sets.

Every simplicial complex is homeomorphic to the classifying space of some category:
A simplicial complex can be seen as a partially ordered set of simplices with the order relation given by the face relation.
Moreover, a partially ordered set (poset) is just a category with at most one arrow between any two objects. The
classifying space of a poset coming from a simplicial complex is exactly the barycentric subdivision of the simplicial complex.

Even more is true: McDuff showed in \cite{mcd:otc} that for each connected simplicial complex there is a monoid
(i.e.~a category with only one object) with classifying space homotopy equivalent to the given complex.
Thus, every path-connected space has the weak homotopy type of some monoid. For example, observe the monoid 
consisting of the identity element and elements $x_{ij}$ with multiplication rules $x_{ij}x_{kl}=x_{il}$. In \cite{fie:act}
it is shown that its classifying space is homotopy equivalent to the $2$-sphere.

\subsection{The fundamental groupoid of a category}

Following the philosophy of transporting topological concepts to categories via the nerve functor, we define the 
fundamental groupoid $\pi_1(\calC)$ of a category $\calC$ to be the fundamental groupoid of its classifying space.
There is also an intrinsic description of the fundamental groupoid of $\calC$ in terms of the category itself 
which we will describe now (see e.g.~\cite{go-ja:sht}*{Chapter III, Corollary 1.2} that these two notions are
indeed the same up to equivalence). The objects of $\pi_1(\calC)$ are the objects of $\calC$ and the arrows of $\pi_1(\calC)$ are paths
modulo homotopy. Here, a path in $\calC$ from an object $A$ to an object $B$ is a zig-zag of morphisms from $A$ to $B$,
i.e.~starting from $A$, one travels from object to object over the arrows of $\calC$, regardless of the direction of the 
arrows. For example, the following zig-zag is a path in $\calC$
\[A\leftarrow C_1\rightarrow C_2\leftarrow C_3\leftarrow C_4\rightarrow C_5\rightarrow B\]
Paths can be concatenated in the obvious way. The homotopy relation on paths is the smallest equivalence 
relation respecting the operation of concatenation of paths generated by the following elementary relations:
\begin{eqnarray*}
	A\xrightarrow{\alpha}B\xrightarrow{\beta}C & \sim & A\xrightarrow{\alpha\beta}C \\
	A\xleftarrow{\alpha}B\xleftarrow{\beta}C & \sim & A\xleftarrow{\beta\alpha}C \\
	A\xrightarrow{\alpha}B\xleftarrow{\alpha}A & \sim & A \\
	A\xleftarrow{\alpha}B\xrightarrow{\alpha}A & \sim & A \\
	A\xrightarrow{\id}A & \sim & A\\
	A\xleftarrow{\id}A & \sim & A
\end{eqnarray*}
where the $A$'s on the right represent the empty path at $A$. Composition in $\pi_1(\calC)$ is given by 
concatenating representatives. The identities are represented by the empty paths. If $A$ is an object of $\calC$ 
then we denote by $\pi_1(\calC,A)$ the automorphism group of $\pi_1(\calC)$ at $A$ and call it the fundamental 
group of $\calC$ at $A$.

The fundamental groupoid of $\calC$ has two further descriptions: First, denote by $G$ the left adjoint functor to the
inclusion functor from groupoids to categories. Then we have $\pi_1(\calC)=G(\calC)$. Second, it is the localization 
$\calC[\calC^{-1}]$ of $\calC$ (at all its morphisms) since it comes with a canonical functor
$\varphi\co \calC\rightarrow\pi_1(\calC)$ satisfying the following universal property:
Having any other functor $\eta\co \calC\rightarrow\calA$ with the property that $\eta(f)$ is an isomorphism in 
$\calA$ for every arrow $f$ in $\calC$, then there is a unique functor $\epsilon\co \pi_1(\calC)\rightarrow\calA$ 
such that $\varphi\epsilon=\eta$.
\begin{displaymath}\xymatrix{
	\calC\ar[r]^\varphi\ar[dr]_\eta & \pi_1(\calC)\ar@{.>}[d]^\epsilon\\
	&\calA
}\end{displaymath}

\subsection{Coverings of categories}\label{78615}

Let $P\co \calD\rightarrow\calC$ be a functor. We say that $P$ is a covering if for every arrow $a$ in $\calC$
and every object $X$ in $\calD$ which projects via $P$ onto the domain or the codomain of $a$, there exists
exactly one arrow $b$ in $\calD$ with domain resp.~codomain $X$ and projecting onto $a$ via $P$. In other
words, arrows can be lifted uniquely provided that the lift of the domain or codomain is given. Of course,
$P$ yields a map on the classifying spaces. To justify the definition of covering functor, we have the following:

\begin{prop}
	Let $P\co \calD\rightarrow\calC$ be a functor. Then $P$ is a covering functor if and only if
	$BP\co B\calD\rightarrow B\calC$ is a covering map of spaces.
\end{prop}
\begin{proof}
	By \cite{ga-zi:cof}*{Appendix I, 3.2}, $BP=|NP|\co |N\calD|\rightarrow|N\calC|$ is a covering map if
	and only if $NP\co N\calD\rightarrow N\calC$ is a covering of simplicial sets as defined in
	\cite{ga-zi:cof}*{Appendix I, 2.1}. This means that every $n$-simplex in $N\calC$ uniquely lifts to
	$N\calD$ provided that the lift of a vertex of the simplex is given. The lifting property for $P$ as defined
	above says that this is true for $1$-simplices. So it is clear that $P$ is a covering functor provided that
	$BP$ is a covering map of spaces. For the converse implication, one exploits special properties of nerves of categories. 
	Not every simplicial set arises as the nerve of a category. The Segal condition gives a necessary and sufficient 
	condition for a simplicial set to come from a category: Every horn $\Lambda^i_n$ for $0<i<n$ can be uniquely filled by 
	an $n$-simplex. Using this, the lifting property for $1$-simplices implies the lifting property for $n$-simplices.
\end{proof}

Now let $\calC$ be a category and $X$ an object in $\calC$. Observe the canonical functor
$\varphi\co \calC\rightarrow\pi_1(\calC)$. Define $\calU_X(\calC)$ to be the category $X\downa\varphi$. The canonical
projection $\calU_X(\calC)\rightarrow\calC$ sending an object $(B,\gamma)$ to $B$ is a covering. Furthermore, $\calU_X(\calC)$
is simply connected, i.e.~connected and its fundamental groupoid is equivalent to the terminal category (see \cite{par:ucc}). 
So it deserves the name universal covering category. More precisely, it is the universal covering of the component of $\calC$
which contains the object $X$.

There is a canonical functor $\pi_1(\calC)\rightarrow\mathtt{CAT}$ taking objects $X$ to the category $\calU_X(\calC)$ 
and an arrow $f\co X\rightarrow Y$ to a functor $\calU_X(\calC)\rightarrow\calU_Y(\calC)$ which is given by precomposition
with $f^{-1}$. Fixing the object $X$, this functor restricts to a functor $\pi_1(\calC,X)\rightarrow\mathtt{CAT}$ 
sending the unique object of the group $\pi_1(\calC,X)$ to the universal covering $\calU_X(\calC)$. This is the same as a
representation of $\pi_1(\calC,X)$ in $\mathtt{CAT}$, i.e.~a group homomorphism $\rho\co \pi_1(\calC,X)\rightarrow
\op{Aut}\big(\calU_X(\calC)\big)$ into the group of invertible functors with multiplication given by $f\cdot g:=f*g$. 
Equivalently, this is a right action of the group $\pi_1(\calC,X)$ on $\calU_X(\calC)$ given by the formula 
$\alpha\cdot\gamma:=\alpha\into(\gamma\into\rho)$ for $\gamma\in\pi_1(\calC,X)$ and arrows $\alpha$ in $\calU_X(\calC)$. 
This gives the usual deck transformations on the universal covering.

\subsection{Contractibility and homotopy equivalences}\label{03887}

We say that a category is contractible if its classifying space is contractible and we say that a functor 
$F\co \calC\rightarrow\calD$ is a homotopy equivalence if $BF\co B\calC\rightarrow B\calD$ is one. There are some
standard conditions which assure that a category is contractible or a functor is a homotopy equivalence. These
will be recalled below.

A non-empty category $\calC$ is contractible if
\begin{itemize}
	\item[i)] $\calC$ has an initial object.
	\item[ii)] $\calC$ has binary products.
	\item[iii)] $\calC$ is a generalized poset (see Definition \ref{55528}) and there is an object $X_0$ together with a functor $F\co \calC\rightarrow\calC$ 
		such that for each object $X$ there exist arrows $X\rightarrow F(X)\leftarrow X_0$ (compare with 
		\cite{qui:hpo}*{Subsection 1.5}).
	\item[iv)] $\calC$ is filtered which means that for every two objects $X,Y$ there is an object $Z$ with arrows
		$X\rightarrow Z$, $Y\rightarrow Z$ and for every two arrows $f,g\co A\rightarrow B$
		there is an arrow $h\co B\rightarrow C$ such that $fh=gh$.
\end{itemize}
Of course, the dual statements are also true. It is instructive to sketch the arguments for these four claims:
\begin{itemize}
	\item[i)] Let $\calI$ be the category with two objects and one non-identity arrow from the first to the second
	object. The classifying space of $\calI$ is the unit interval $I$. A natural transformation of two functors
	$f,g\co \calC\rightarrow\calD$ can be interpreted as a functor $\calC\times\calI\rightarrow\calD$. On the level
	of spaces, this gives a homotopy $B\calC\times I\rightarrow B\calD$. If $\calC$ is a category with initial
	object $X_0$, then there is a unique natural transformation from the functor $\mathrm{const}_{X_0}$ (sending every 
	arrow of $\calC$ to $\id_{X_0}$) to the identity functor $\id_\calC$. On the level of spaces, this yields a homotopy 
	between $\id_{B\calC}$ and the constant map $B\calC\rightarrow B\calC$ with value the point $X_0$.
	\item[ii)] Choose an object $X_0$ in $\calC$. Let $F\co \calC\rightarrow\calC$ be the functor $Y\mapsto X_0\times Y$. 
	Projection onto the first factor yields a natural transformation $F\rightarrow \mathrm{const}_{X_0}$ and projection onto 
	the second factor yields a natural transformation $F\rightarrow\id_\calC$. This gives two homotopies which together give 
	the desired contraction of $B\calC$.
	\item[iii)] First note that, if $F,G\co \calC\rightarrow\calC$ are two functors with the property that
	there is an arrow $F(X)\rightarrow G(X)$ for each object $X$, then this already defines a natural transformation
	$F\rightarrow G$ by uniqueness of arrows in the generalized poset. Now the conditions on $X_0$ and $F$ yield that
	there are natural transformations $\id_\calC\rightarrow F$ and $\mathrm{const}_{X_0}\rightarrow F$. On the level of 
	spaces this gives the desired contraction of $B\calC$.
	\item[iv)] First, let $\calD$ be a finite subcategory of $\calC$. We claim that there exists a cocone over 
	$\calD$ in $\calC$, i.e.~there is an object $Z$ in $\calC$ and for each object $Y$ in $\calD$ an arrow 
	$Y\rightarrow Z$ which commute with the arrows in $\calD$. This cocone is contractible because $Z$ is a terminal
	object. A cocone can be constructed as follows: First pick two objects $Y_1,Y_2$ in $\calD$ and find an object $Z'$ with arrows 
	$Y_1\rightarrow Z'$ and $Y_2\rightarrow Z'$. Pick another object $Y_3$ and find an object $Z''$ with arrows
	$Y_3\rightarrow Z''$ and $Z'\rightarrow Z''$. Repeating this with all objects of $\calD$, we obtain an object
	$Q$ together with arrows $f_Y\co Y\rightarrow Q$ for every object $Y$ in $\calD$. The $f_Y$ probably
	won't commute with the arrows in $\calD$ yet, but we can repair this by repeatedly applying the second
	property of filteredness. Pick an arrow $d\co Y\rightarrow Y'$ in $\calD$ and observe the parallel arrows
	$df_{Y'}$ and $f_Y$. Apply the second property to find an arrow $\omega\co Q\rightarrow Q'$ with 
	$df_{Y'}\omega=f_Y\omega$. Replace $Q$ by $Q'$ and all the arrows $f_D$ for objects $D$ in $\calD$ by
	$f_D\omega$. Repeat this with all the other arrows in $\calD$.
	
	Now to finish the proof of this item, take a map $S^n\rightarrow B\calC$. Since $S^n$ is compact, it can be homotoped
	to a map such that the image is covered by the geometric realization of a finite subcategory. The cocone over this 
	subcategory then gives the desired null-homotopy.
\end{itemize}

We recall Quillen's famous Theorem A from \cite{qui:hak} which gives a sufficient but in general not necessary
condition for a functor to be a homotopy equivalence.
\begin{thm}\label{56361}
	Let $f\co \calC\rightarrow\calD$ be a functor. If for each object $Y$ in $\calD$ the category $Y\downa f$ is
	contractible, then the functor $f$ is a homotopy equivalence. Similarly, if the category $f\downa Y$ is
	contractible for each object $Y$ in $\calD$, then $f$ is a homotopy equivalence.
\end{thm}
\begin{rem}\label{86729}
	When applying this theorem to an inclusion $f\co \calA\rightarrow\calB$ of a {\it full} subcategory, it suffices 
	to check $Y\downa f=Y\downa\calA$ for objects $Y$ not in $\calA$. If $Y$ is an object in $\calA$, the comma 
	category $Y\downa\calA$ has the object $(Y,\id_Y)$ as initial object and thus is automatically
	contractible. Similar remarks apply to the comma categories $f\downa Y=\calA\downa Y$.
\end{rem}
\begin{rem}
	If $\calD$ is a groupoid, then for $Y,Y'\in\calD$ the comma categories $Y\downa f$ and $Y'\downa f$ are
	isomorphic. Thus one has to check contractibility only for one $Y$. The same remarks apply to the comma
	categories $f\downa Y$.
\end{rem}

\subsection{Smashing isomorphisms in categories}\label{16497}

Recall that a connected groupoid is equivalent, as a category, to any of its automorphism groups. Consequently, a connected
groupoid is contractible if and only if its automorphism groups are trivial. This is the case if and only if there is
exactly one isomorphism between any two objects.

Let $\calC$ be a category and $\calG\subset\calC$ a subcategory which is a disjoint union of contractible groupoids. We define
the quotient category $\calC/\calG$ as follows: The objects of $\calC/\calG$ are equivalence classes of objects
of $\calC$ where we say that $X\sim Y$ are equivalent if there is an isomorphism $X\rightarrow Y$ in $\calG$.
Note that such an isomorphism is unique since each component of $\calG$ is contractible. We define
\[\op{Hom}_{\calC/\calG}\big([X],[Y]\big):=\big\{A\rightarrow B\text{ in }\calC\ \big|\ A\in[X],
	B\in[Y]\big\}\big/{\sim}\]
where two elements $(A\rightarrow B)\sim(A'\rightarrow B')$ in the set are defined to be equivalent if the
diagram
\begin{displaymath}\xymatrix{
	A\ar[r]\ar@{-->}[d]_{\calG\ni}&B\ar@{-->}[d]^{\in\calG}\\
	A'\ar[r]&B'
}\end{displaymath}
commutes. Let $[\alpha\co A\rightarrow B]$ and $[\beta\co C\rightarrow D]$ be two composable arrows, i.e.~$[B]=[C]$, 
then there is a unique isomorphism $\gamma\co B\rightarrow C$ in $\calG$ and one defines
\[[\alpha\co A\rightarrow B]*[\beta\co C\rightarrow D]:=[\alpha\gamma\beta\co A\rightarrow D]\]
Set $\id_{[X]}=[\id_X]$. One easily checks that $\calC/\calG$ is a well-defined category.

\begin{rem}\label{42310}
	Observe that if $\calX\rightarrow\calY$ is an arrow in $\calC/\calG$ and representatives $X$ and $Y$ have been chosen
	for $\calX$ and $\calY$, then there is a unique arrow $X\rightarrow Y$ representing $\calX\rightarrow\calY$.
\end{rem}

\begin{rem}\label{64237}
	Let $\calX$ and $\calY$ be two objects in $\calC/\calG$. Fix some object $X_0$ representing $\calX$. Then the
	arrows $\calX\rightarrow\calY$ in $\calC/\calG$ are in one to one correspondence with arrows $X_0\rightarrow Y$
	in $\calC$ modulo isomorphisms in $\calG$ on the right. Likewise, if we fix some object $Y_0$ representing
	$\calY$, then arrows $\calX\rightarrow\calY$ in $\calC/\calG$ are in one to one correspondence with arrows
	$X\rightarrow Y_0$ in $\calC$ modulo isomorphisms in $\calG$ on the left.
\end{rem}

\begin{prop}\label{66872}
	The canonical projection $p\co \calC\rightarrow\calC/\calG$ is a homotopy equivalence.
\end{prop}
\begin{proof}
	We want to apply Quillen's Theorem A (Theorem \ref{56361}) to the projection $p\co \calC\rightarrow\calC/\calG$.
	Hence, we have to show that for each object $[X]$ in $\calC/\calG$ the comma category
	$[X]\downa p$ is contracible. Indeed, it follows with Remark \ref{42310} that the object $(X,\id_{[X]})$ in
	$[X]\downa p$ is an initial object.
%
\end{proof}

In the following, this technique will be applied primarily to \emph{generalized posets}:

\begin{defi}\label{55528}
	A generalized poset is a category such that $\alpha=\beta$ whenever $\alpha,\beta\colon A\rightarrow B$.
\end{defi}

Recall that a (honest) poset is a category with at most one arrow between any two objects (regardless of the direction of the arow).
In a generalized poset $\calC$, however, we allow objects to be uniquely isomorphic. Every subgroupoid $\calG$ of a generalized poset 
is a disjoint union of contractible ones and $\calC/\calG$ is a generalized poset again. 
If we collapse each connected component of the subgroupoid
consisting of all the isomorphisms, we even get a homotopy equivalent (honest) poset which we call the \emph{underlying poset} of
the generalized poset.

\subsection{Calculus of fractions and cancellation properties}

The next definition is very classical and due to Gabriel and Zisman \cite{ga-zi:cof}.
\begin{defi}\label{12985}
	Let $\calC$ be a category. It satisfies the calculus of fractions if the following two conditions
	are satisfied:
	\begin{itemize}
		\item ({\it Square filling}) For every pair of arrows $f\co B\rightarrow A$ and 
		$g\co C\rightarrow A$ there are 
		arrows $a\co D\rightarrow B$ and $b\co D\rightarrow C$ such that $af=bg$.
		\begin{displaymath}\xymatrix{
			D\ar@{.>}[r]^a\ar@{.>}[d]_b&B\ar[d]^f\\
			C\ar[r]_g&A
		}\end{displaymath}
		\item ({\it Equalization}) Whenever we have arrows $f,g\co A\rightarrow B$ and $a\co B\rightarrow C$ 
		such that $fa=ga$, then there exists an arrow $b\co D\rightarrow A$ with $bf=bg$.
		\begin{displaymath}\xymatrix{
  			D\ar@{.>}^b[r]&A\ar@<-2pt>[r]_g\ar@<2pt>[r]^f&B\ar[r]^a&C
		}\end{displaymath}
	\end{itemize}
\end{defi}
More precisely, this is called the {\it right} calculus of fractions. There is also a dual {\it left} calculus
of fractions. Since we are mainly interested in the right calculus of fractions, we omit the word ``right''.

\begin{rem}
	The existence of binary pullbacks in $\calC$ trivially implies the square filling property but it 
	also implies the equalization property \cite{ben:sro}*{Lemma 1.2}. So a category with binary pullbacks
	satisfies the calculus of fractions.
\end{rem}

The calculus of fractions has positive effects on the complexity of the fundamental groupoid $\pi_1(\calC)$: 
One can show (see e.g.~\cite{ga-zi:cof} or \cite{bor:hoc}) that each class in $\pi_1(\calC)$ can be represented 
by a \emph{span} which is a zig-zag of the form
\begin{displaymath}\xymatrix{
	\bullet & \bullet\ar[l]\ar[r] & \bullet
}\end{displaymath}
Furthermore, two spans
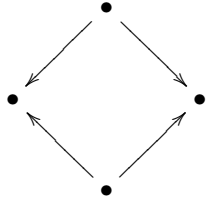
\begin{displaymath}\xymatrix{
	& \bullet\ar[dl]\ar[dr] &\\
	\bullet & &\bullet\\
	& \bullet\ar[ul]\ar[ur] &
}\end{displaymath}
are homotopic if and only if the diagram can be filled in the following way:
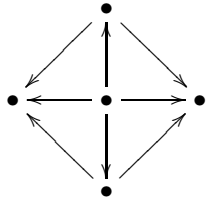
\begin{displaymath}\xymatrix{
	& \bullet\ar[dl]\ar[dr] &\\
	\bullet & \bullet\ar[l]\ar[r]\ar[u]\ar[d] &\bullet\\
	& \bullet\ar[ul]\ar[ur] &
}\end{displaymath}
In other words, the elements in the localization can be described as fractions and this explains the name
of the calculus of fractions. We will frequently write $(\alpha,\beta)$ for a span consisting of arrows $\alpha$
and $\beta$ where the first arrow $\alpha$ points to the left (i.e.~is the denominator) and the second arrow
$\beta$ points to the right (i.e.~is the nominator).
Two spans are composed by concatenating representatives to a zig-zag and then
transforming the zig-zag into a span by choosing a square filling of the middle cospan.
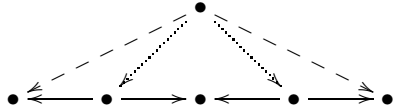
\begin{displaymath}\xymatrix{
	&&\bullet\ar@{-->}[lld]\ar@{.>}[ld]\ar@{.>}[rd]\ar@{-->}[rrd]&&\\
	\bullet&\bullet\ar[l]\ar[r]&\bullet&\bullet\ar[l]\ar[r]&\bullet
}\end{displaymath}
The canonical functor $\varphi\co\calC\rightarrow\pi_1(\calC)$ is given by sending an arrow $\alpha$ to the class 
represented by the span
\begin{displaymath}\xymatrix{
	\bullet&\bullet\ar[l]_{\id}\ar[r]^{\alpha}&\bullet
}\end{displaymath}
Using the special form of the homotopy relation from above, we see that two arrows $\alpha,\beta\co X\rightarrow Y$ are
homotopic if and only if there is an arrow
$\omega\co A\rightarrow X$ such that $\omega\alpha=\omega\beta$. 

%

\vspace{2mm}
We now turn to cancellation properties in categories.

\begin{defi}
	Let $\calC$ be a category. It is called right cancellative if $fa=ga$ for arrows $f,g,a$ implies $f=g$. It
	is called left cancellative if $af=ag$ implies $f=g$. It is called cancellative if it is left and right
	cancellative.
\end{defi}

\begin{rem}
	Note that we have the following implications:
	\begin{align*}
		\text{right cancellation}&\Longrightarrow\text{equalization}\\
		\text{equalization}+\text{left cancellation}&\Longrightarrow\text{right cancellation}
	\end{align*}
\end{rem}

\begin{prop}\label{20889}
	Let $\calC$ be a category satisfying the cancellative calculus of fractions. Then the canonical functor
	$\varphi\co \calC\rightarrow\pi_1(\calC)$ is faithful and a homotopy equivalence.
\end{prop}
\begin{proof}
	Injectivity is easy: Let $f,g$ be arrows in $\calC$ which are mapped to the same arrow in $\pi_1(\calC)$.
	This means that $f,g$ are homotopic. Since $\calC$ satisfies the calculus of fractions, this implies that there is an arrow
	$\omega$ with $\omega f=\omega g$. From the cancellation property it follows that $f=g$.
	
	For showing that the functor is a homotopy equivalence, we apply Quillen's Theorem A (Theorem \ref{56361}) to
	the functor $\varphi\co \calC\rightarrow\pi_1(\calC)$. Let $X$ be an object in $\pi_1(\calC)$, i.e.~an object 
	in $\calC$. We have to check that the comma category $X\downa\varphi$ is contractible. 
	Note that this is the universal covering category $\calU_X(\calC)$. 
	First we claim that this category is a generalized poset: Let $(Z,a)$ and $(Z',a')$ be objects in
	$X\downa\varphi$ and $\gamma_1,\gamma_2$ be two arrows from $(Z,a)$ to $(Z',a')$. This means that $Z,Z'$
	are objects in $\calC$, $a\co X\rightarrow Z\into\varphi$ and $a'\co X\rightarrow Z'\into\varphi$ are arrows in 
	$\pi_1(\calC)$ and $\gamma_1,\gamma_2\co Z\rightarrow Z'$ are arrows in $\calC$ such that 
	$a*(\gamma_1\into\varphi)=a'=a*(\gamma_2\into\varphi)$ in $\pi_1(\calC)$. 
	It follows $\gamma_1\into\varphi=\gamma_2\into\varphi$ and therefore $\gamma_1=\gamma_2$ by injectivity.
	
	Now we want to show that this generalized poset is cofiltered. Then we can apply item iv) of Subsection \ref{03887}.
	We have to show that for each two objects $A,A'$ in $X\downa\varphi$ there is another object
	$B$ and arrows $B\rightarrow A$ and $B\rightarrow A'$. Let $A=(Z,a)$ and $A'=(Z',a')$ with arrows
	$a\co X\rightarrow Z\into\varphi$ and $a'\co X\rightarrow Z'\into\varphi$ which can be represented by spans 
	$(\alpha,\beta)$ and $(\alpha',\beta')$ respectively. Choose a square filling $(\gamma,\gamma')$ of the
	cospan $(\alpha,\alpha')$.
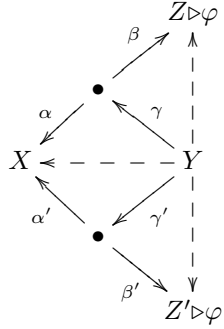
	\begin{displaymath}\xymatrix@-8pt{
		&&Z\into \varphi\\
		&\bullet\ar[dl]_{\alpha}\ar[ur]^{\beta}&\\
		X&&Y\ar@{-->}[uu]\ar@{-->}[dd]\ar@{-->}[ll]\ar[lu]_{\gamma}\ar[ld]^{\gamma'}\\
		&\bullet\ar[ul]^{\alpha'}\ar[dr]_{\beta'}&\\
		&&Z'\into \varphi\\
	}\end{displaymath}
	Then the arrow $\omega:=\gamma\alpha=\gamma'\alpha'$ can be interpreted as the denominator of a span representing an
	arrow in $\pi_1(\calC)$ which we denote by $\omega^{-1}$. Furthermore, since $\varphi\co\calC\rightarrow\pi_1(\calC)$
	is the identity on objects, we can write $Y=Y\into\varphi$. Thus, we can define the object $B:=(Y,\omega^{-1})$ in
	$X\downa\varphi$. Finally, the arrows $\gamma\beta$ and $\gamma'\beta'$ give arrows $B\rightarrow A$ and $B\rightarrow A'$ 
	respectively.
\end{proof}


\begin{rem}
	The functor $\varphi\co \calC\rightarrow\pi_1(\calC)$ in Proposition \ref{20889} is still a homotopy equivalence if we drop the cancellation property 
	from the hypothesis. This is proved in \cite{dw-ka:csl}*{Section 7}.
\end{rem}

\subsection{Monoidal categories}\label{50967}

We assume that the reader is acquainted with the definition of monoidal categories, symmetric monoidal categories and
braided monoidal categories (see e.g.~\cite{mac:cft}). In the following, we will always assume the strict versions, 
i.e.~the associator, right and left unitor are identities. We frequently use the symbol $I$ to denote the unit object.
Moreover, for objects $X$ and $Y$, the symbol $\gamma_{X,Y}$ denotes the natural braiding isomorphism
$X\otimes Y\rightarrow Y\otimes X$. We will sometimes call a monoidal category planar in order to stress that it's
neither symmetric nor braided.

Joyal and Street introduced the notion of braided monoidal categories in \cite{jo-st:btc}. It is designed
such that the braided monoidal category freely generated by a single object is the groupoid with components
the braid groups $B_n$. More precisely, we have an object for each natural number $n$, there are no morphisms
$n\rightarrow m$ with $n\neq m$ and $\op{Hom}(n,n)=B_n$. More generally, they indroduced the braided
monoidal category $\mathfrak{Braid}(\calC)$ freely generated by another category $\calC$ \cite{jo-st:btc}*{page 37}:
The objects are free words in the objects of $\calC$, i.e.~finite sequences of objects of $\calC$. A morphism
consists of a braid $\beta\in B_n$ where the strands are labelled with morphisms $\alpha_i\co A_i\rightarrow B_i$ 
of $\calC$, yielding an arrow 
\[(\beta,\alpha_1,\ldots,\alpha_n)\co (A_{1\into\beta},\ldots,A_{n\into\beta})\rightarrow(B_1,\ldots,B_n)\]
in $\mathfrak{Braid}(\calC)$. Composition is performed by composing the braids and applying composition in $\calC$
to every strand. The tensor product is given by juxtaposition, i.e.~by
\[(\beta,\alpha_1,\ldots,\alpha_n)\otimes(\beta',\alpha'_1,\ldots,\alpha'_n)=\big(\beta\otimes\beta',\alpha_1,\ldots,\alpha_n,\alpha'_1,\ldots,\alpha'_n\big)\]
where $\beta\otimes\beta'$ means juxtaposition of braids. A set $C$ can be viewed as a discrete category, so
we also obtain the notion of a braided monoidal category $\mathfrak{Braid}(C)$ freely generated by a set. The arrows
are just braids with strands labelled by the elements of $C$, i.e.~are colored.

The same remarks apply to the symmetric version. In particular, a category $\calC$ freely generates a symmetric
monoidal category $\mathfrak{Sym}(\calC)$.

Even simpler, we can form the free monoidal category $\mathfrak{Mon}(\calC)$ generated by a category $\calC$.
The strands are decorated by arrows in $\calC$ but they are not allowed to braid or cross each other.

\vspace{2mm}
If $\calC$ is a (symmetric/braided) monoidal category, then there is exactly one tensor structure on
$\pi_1(\calC)$ making it into a (symmetric/braided) monoidal category and such that the canonical
functor $\varphi\co \calC\rightarrow\pi_1(\calC)$ respects that structure, i.e.
\begin{eqnarray*}
	\varphi(X\otimes Y)&=&\varphi(X)\otimes\varphi(Y)\\
	\varphi(\alpha\otimes\beta)&=&\varphi(\alpha)\otimes\varphi(\beta)\\
	\varphi(I)&=&I\\
	\varphi(\gamma_{X,Y})&=&\gamma_{\varphi(X),\varphi(Y)}
\end{eqnarray*}
for objects $X,Y$ and arrows $\alpha,\beta$. 
The tensor product on the level of arrows can be constructed as follows: Let one arrow be represented by the zig-zag 
\begin{displaymath}\begin{array}{ccccccc}
		A_0&\xleftarrow{\alpha_1}&A_1&\xrightarrow{\alpha_2}&\cdots &\xrightarrow{\alpha_k}&A_k
\end{array}\end{displaymath}
and the other arrow by the zig-zag
\begin{displaymath}\begin{array}{ccccccccc}
		B_0&\xrightarrow{\beta_1}&B_1&\xleftarrow{\beta_2}&\cdots &\xrightarrow{\beta_l}&B_l
\end{array}\end{displaymath}
Then the tensor product may be represented by the zig-zag
\begin{displaymath}\begin{array}{ccccccccccccc}
		A_0&\xleftarrow{\alpha_1}&A_1&\xrightarrow{\alpha_2}&\cdots &\xrightarrow{\alpha_k}&A_k&
		\xrightarrow{\id}&A_k&\xleftarrow{\id}&\cdots &
		\xrightarrow{\id}&A_k\\
		\otimes &\otimes &\otimes &\otimes &\cdots &\otimes &\otimes &\otimes &\otimes &
		\otimes &\cdots &\otimes &\otimes\\
		B_0&\xleftarrow{\id}&B_0&\xrightarrow{\id}&\cdots &
		\xrightarrow{\id}&
		B_0&\xrightarrow{\beta_1}&B_1&\xleftarrow{\beta_2}&\cdots &\xrightarrow{\beta_l}&B_l
\end{array}\end{displaymath}

\subsection{Cones and joins}\label{06663}

Let $\calC,\calD$ be two categories. We define the join $\calC*\calD$. The set of objects of $\calC*\calD$ is the 
disjoint union of the objects of $\calC$ and $\calD$. The set of arrows is the disjoint union
of the arrows of $\calC$ and $\calD$ together with exactly one arrow $C\rightarrow D$ for each pair $(C,D)$ of 
objects $C$ of $\calC$ and $D$ of $\calD$. The composition rules are the unique ones extending the compositions
in $\calC$ and $\calD$. The classifying space of the join is homotopy equivalent to
the join of the classifying spaces $B(\calC*\calD)\simeq B\calC*B\calD$.

Now we define the cone over a category. The objects of $\op{Cone}(\calC)$
are the objects of $\calC$ plus another object called $\mathtt{tip}$. The arrows are the arrows of $\calC$ together
with exactly one arrow from $\mathtt{tip}$ to every object in $\calC$. Dually, there is a 
$\op{Cocone}(\calC)$ over $\calC$. In the cocone, the extra arrows go from the objects of $\calC$
to the extra object $\mathtt{tip}$. Last but not least, when we have a join $\calC*\calD$ of two categories,
there is a mixed version $\op{Coone}(\calC*\calD)$ which we call the coone over the join.
Again, there is one extra object $\mathtt{tip}$ and for every object in $\calC*\calD$ we have an extra arrow.
When we have an object in $\calC$, the extra arrow {\it goes to} $\mathtt{tip}$. When we have an arrow in
$\calD$, the extra arrow {\it comes from} $\mathtt{tip}$. The composition of an arrow $C\rightarrow\mathtt{tip}$
with an arrow $\mathtt{tip}\rightarrow D$ is the unique arrow $C\rightarrow D$ from the definition of the join.
All three coning versions give the usual conings on the topological level:
\begin{eqnarray*}
	B(\op{Cone}(\calC))&\cong &\op{Cone}(B(\calC))\\
	B(\op{Cocone}(\calC))&\cong &\op{Cone}(B(\calC))\\
	B(\op{Coone}(\calC*\calD))&\cong &\op{Cone}(B(\calC*\calD))
\end{eqnarray*}

\vspace{2mm}
The join of two spaces $X$ and $Y$ is defined to be the homotopy pushout of the two projections
$X\leftarrow X\times Y\rightarrow Y$. Thus, it is defined only up to homotopy and there is some freedom to choose models
of a join. Indeed, there is another construction giving the join of two categories. For this, we need to recall the Grothendieck
construction: Let $\calJ$ be some indexing category and $F\co \calJ\rightarrow\mathtt{CAT}$ a diagram in $\mathtt{CAT}$. 
The objects of the Grothendieck construction $\int F$ are pairs $(J,X)$ of objects $J$ in $\calJ$ and $X$ in $J\into F$. 
An arrow from $(J,X)$ to $(J',X')$ is a pair $(f,\alpha)$ consisting of an arrow $f\co J\rightarrow J'$ and an arrow 
$\alpha\co X\into(f\into F)\rightarrow X'$.
Composition is given by 
\[(f,\alpha)*(f',\alpha'):=\big(f*f',\alpha\into(f'\into F)*\alpha'\big)\]
In \cite{tho:caa}
it is shown that there is a model structure on $\mathtt{CAT}$ Quillen equivalent to $\mathtt{SSET}$, nowadays called
the Thomason model structure, and in \cite{tho:hci} it is shown that the nerve of the Grothendieck construction $\int F$ 
is homotopy equivalent to the homotopy pushout of the diagram $F*N$ which is obtained from the diagram $F$ by applying the nerve 
functor. In fact, $\int F$ realizes the homotopy pushout of $F$ with respect to the Thomason model structure on $\mathtt{CAT}$
\cite{f-l-s:eco}*{Section 3}.

Now let $\calC,\calD$ be categories. We call the Grothendieck construction of the diagram
\[\calC\xleftarrow{\op{pr}_\calC}\calC\times\calD\xrightarrow{\op{pr}_\calD}\calD\]
the Grothendieck join of $\calC$ and $\calD$ and denote it by $\calC\circ\calD$. From \cite{tho:hci} we know
that $B(\calC\circ\calD)$ is homotopy equivalent to the homotopy pushout of the diagram
\[B\calC\xleftarrow{\op{pr}_{B\calC}}B\calC\times B\calD\xrightarrow{\op{pr}_{B\calD}}B\calD\]
But the latter is the join $B\calC * B\calD$ by definition. So we have $B(\calC\circ\calD)\simeq B\calC * B\calD$.

One can show that the Grothendieck join is associative and thus we can write $\calC_1\circ\ldots\circ\calC_k$
for a finite collection $\calC_i$ of categories. The objects of such an iterated Grothendieck join are elements
of the set
\[\op{obj}(\calC_1\circ\ldots\circ\calC_k)=\coprod_{S\subset\{1,\ldots,k\}}\prod_{s\in S}\op{obj}(\calC_s)\]
Whenever we have $S\subset T\subset\{1,\ldots,k\}$, objects $(Y_t)_{t\in T}$ and $(X_s)_{s\in S}$ and arrows 
$\alpha_s\co Y_s\rightarrow X_s$ in $\calC_s$ for each $s\in S$, then there is an arrow
\[(\alpha_s)_{s\in S}\co (Y_t)_{t\in T}\rightarrow (X_s)_{s\in S}\]
For $R\subset S\subset T$ the composition is given by
\[(\alpha_s)_{s\in S}*(\beta_r)_{r\in R}=(\alpha_r*\beta_r)_{r\in R}\]

There is also a dual notion of the Grothendieck join which we define as
\[\calC\bullet\calD:=(\calC^{op}\circ\calD^{op})^{op}\]
Since $B(\calA^{op})=B(\calA)$ for any category, we still have $B(\calC\bullet\calD)\simeq B\calC * B\calD$.
Furthermore, it is still associative, so that we can write $\calC_1\bullet\ldots\bullet\calC_k$ for a finite collection 
$\calC_i$ of categories. The objects of such an iterated dual Grothendieck join are elements of the set
\[\op{obj}(\calC_1\bullet\ldots\bullet\calC_k)=\coprod_{S\subset\{1,\ldots,k\}}\prod_{s\in S}\op{obj}(\calC_s)\]
Whenever we have $S\subset T\subset\{1,\ldots,k\}$, objects $(X_s)_{s\in S}$ and $(Y_t)_{t\in T}$ and arrows 
$\alpha_s\co X_s\rightarrow Y_s$ in $\calC_s$ for each $s\in S$, then there is an arrow
\[(\alpha_s)_{s\in S}\co (X_s)_{s\in S}\rightarrow(Y_t)_{t\in T}\]
For $R\subset S\subset T$ the composition is given by
\[(\beta_r)_{r\in R}*(\alpha_s)_{s\in S}=(\beta_r*\alpha_r)_{r\in R}\]

\subsection{The Morse method for categories}\label{00484}

We first recall the Morse method in the case of simplicial complexes which has been used in \cite{be-br:mta} to prove 
finiteness properties of certain groups. We then explain the same method in the context of categories.

Let $\calC$ be a simplicial complex. Let $v$ be a vertex in $\calC$. Denote by $\calC^{-v}$ the full subcomplex spanned by the 
vertices of $\calC$ except $v$. Observe the link $lk(v)$ of $v$ in $\calC$ which is contained in $\calC^{-v}$. 
We then have a canonical pushout diagram
\begin{displaymath}\xymatrix{
	lk(v)\ar[d]\ar[rr]&& \op{Cone}\big(lk(v)\big)\ar[d]\\
	\calC^{-v}\ar[rr]&&\calC
}\end{displaymath}
where $\op{Cone}\big(lk(v)\big)$ denotes the simplicial cone over $lk(v)$. The following lemma expresses the 
connectivity of the pair $(\calC,\calC^{-v})$ in terms of the connectivity of $lk(v)$.

\begin{lem}\label{76169}
	Let $X$ and $L$ be two spaces and $L\rightarrow C$ a cofibration into a contractible space $C$. Let
	\begin{displaymath}\xymatrix{
		L\ar[d]\ar[rr]&& C\ar[d]\\
		X\ar[rr]&&Z
	}\end{displaymath}
	be a pushout of spaces. If $L$ is $(n-1)$-connected, then the pair $(Z,X)$ is $n$-connected.
\end{lem}
The proof is a standard application of the Seifert--van Kampen theorem, the Hurewicz theorem and the Mayer--Vietoris sequence for pushouts.

More generally, let $\calX_0$ be the full subcomplex of the simplicial complex $\calX$ spanned by a subset of vertices. Then $\calX$
can be built up from $\calX_0$ by successively adding vertices. We thus get a filtration $\calX_0\subset\calX_1\subset\ldots\subset \calX$
of $\calX$ by full subcomplexes. If $v$ is the vertex in $\calX_i$ which is not contained in $\calX_{i-1}$, then we define
\[lk_\downa(v):=lk_{\calX_i}(v)=lk_\calX(v)\cap\calX_{i-1}\]
to be the \emph{descending link} of $v$. If all the descending links appearing this way are highly connected, 
then also the pair $(\calX,\calX_0)$ will be highly connected and, using the long exact homotopy sequence, we obtain the following:

\begin{prop}
	Let $x_0\in \calX_0$ be a point. Assume that each descending link is $n$-connected. Then, we have
	\[\pi_k(\calX_0,x_0)\cong\pi_k(\calX,x_0)\]
	for $k=0,\ldots,n$.
\end{prop}

Note that, in general, the descending links depend on the order in which the vertices are added. We call two vertices 
$v_1,v_2$ in $\calX\setminus\calX_0$ \emph{independent} if they are not joined by an edge in $\calX$. Assume now that we want to
add $v_1$ and then $v_2$ at some step of the process.
The independence condition ensures that the descending links of $v_1$ and $v_2$ do \emph{not} depend on the order in which $v_1$ and 
$v_2$ are added.

The adding order is often encoded in a Morse function. This is a function $f$ assigning to each vertex in $\calX\setminus\calX_0$ 
an element in a totally ordered set, e.g.~$\NN$. We require that vertices with the same $f$-value are pairwise independent. We 
then add vertices in order of ascending $f$-values. Because of the independence property, the adding order of vertices with the same 
$f$-value can be chosen arbitrarily. Alternatively, we can add vertices with the same $f$-value all at once.

\vspace{2mm}
We now give a version of this concept for categories. Let $\calC$ be a category and $X$ an object in $\calC$ with 
$\op{Hom}_\calC(X,X)=\{\id_X\}$. Define $\calC^{-X}$ to be the full subcategory of $\calC$ spanned by the objects of $\calC$
except $X$. We define
\[\overline{lk}_\downa(X):=\calC^{-X}\downa X\]
to be the {\it descending up link} of $X$ and 
\[\underline{lk}_\downa(X):=X\downa\calC^{-X}\]
to be the {\it descending down link} of $X$. Furthermore, define
\[lk_\downa(X):=\overline{lk}_\downa(X)*\underline{lk}_\downa(X)\]
to be the {\it descending link} of $X$.

We have a commutative diagram $\mathfrak{D}$ as follows:
\begin{displaymath}\xymatrix{
	lk_\downa(X)\ar[d]\ar[rr]&&
	\op{Coone}\big(\overline{lk}_\downa(X)*\underline{lk}_\downa(X)\big)\ar[d]\\
	\calC^{-X}\ar[rr]&&\calC
}\end{displaymath}
The horizontal arrows are the obvious inclusions. We explain the vertical arrows, starting with
\[\overline{lk}_\downa(X)*\underline{lk}_\downa(X)\rightarrow\calC^{-X}\]
An object either comes from $\overline{lk}_\downa(X)$ and thus is a pair $(Y,Y\rightarrow X)$ with $Y$ an object in 
$\calC^{-X}$ or comes from $\underline{lk}_\downa(X)$ and thus is a pair $(Y,X\rightarrow Y)$ with $Y$ an object in 
$\calC^{-X}$. In both cases, the object will be sent to $Y$. Similarly, on the level of arrows, it is also the canonical
projection from $\overline{lk}_\downa(X)$ or $\underline{lk}_\downa(X)$ to $\calC^{-X}$. However, for each object 
$(Y,Y\rightarrow X)$ in $\overline{lk}_\downa(X)$ and each object $(Y',X\rightarrow Y')$ in $\underline{lk}_\downa(X)$, 
there is another unique arrow in the join. Send this arrow to the composed arrow $Y\rightarrow X\rightarrow Y'$. Next, we will 
define the arrow
\[\op{Coone}\big(\overline{lk}_\downa(X)*\underline{lk}_\downa(X)\big)\rightarrow\calC\]
Of course, in order to make the diagram commutative, this functor restricted to the base $lk_\downa(X)$ of
the coone is the one already defined above. So we have to define the images of the extra object $\mathtt{tip}$
and the extra arrows. Send $\mathtt{tip}$ to $X$. Let $(Y,Y\rightarrow X)$ be an object of
$\overline{lk}_\downa(X)$. The arrow from this object to $\mathtt{tip}$ is sent to the arrow $Y\rightarrow X$.
Similarly, let $(Y,X\rightarrow Y)$ be an object of $\underline{lk}_\downa(X)$. Then the arrow from
$\mathtt{tip}$ to this object is sent to the arrow $X\rightarrow Y$.

Our goal is to show that the diagram $\mathfrak{D}$ becomes a pushout on the level of classifying spaces. Unfortunately, this
is not always the case. Consider for example the groupoid $\bullet\rightleftarrows\bullet$ with two objects and two non-identity
arrows which are inverse to each other. In all these cases, however, the situation is even better:

\begin{lem}\label{11228}
	Assume that there is an object $A\neq X$ and arrows $\alpha\co X\rightarrow A$ and $\beta\co A\rightarrow X$.
	Then the inclusion $\calC^{-X}\rightarrow\calC$ is a homotopy equivalence.
\end{lem}
\begin{proof}
	We show that $\calC^{-X}\downa X$ is filtered and thus contractible. The lemma then follows from 
	Theorem \ref{56361} and Remark \ref{86729}. Let $(Y,\gamma)$ be an object in $\calC^{-X}\downa X$, 
	i.e.~$\gamma\co Y\rightarrow X$ is an arrow in $\calC$ with $Y$ an object in $\calC^{-X}$. Set $\epsilon:=\gamma\alpha$. 
	Because of the assumption $\op{Hom}_\calC(X,X)=\{\id_X\}$, the arrow $\alpha\beta\co X\rightarrow X$ must be the
	identity. Then we calculate
	\[\gamma=\gamma(\alpha\beta)=(\gamma\alpha)\beta=\epsilon\beta\]
	This shows that $\epsilon$ represents an arrow $(Y,\gamma)\rightarrow(A,\beta)$ in $\calC^{-X}\downa X$. In particular, 
	for every two objects in the comma category, there are arrows to the object $(A,\beta)$. This shows the first property of a 
	filtered category.
	
	For the second property, we have to show that any two parallel arrows are coequalized by another arrow.
	So let $(Z,\nu)$ and $(Y,\gamma)$ be two objects and $\epsilon,\epsilon'\co (Z,\nu)\rightarrow(Y,\gamma)$ 
	be two arrows, i.e.~$\epsilon,\epsilon'\co Z\rightarrow Y$ are arrows in $\calC^{-X}$ and we have 
	$\epsilon\gamma=\nu=\epsilon'\gamma$. Set $\mu=\gamma\alpha$ which is an arrow $(Y,\gamma)\rightarrow
	(A,\beta)$ as already pointed out. Then we calculate
	\[\epsilon\mu=\epsilon\gamma\alpha=\nu\alpha=\epsilon'\gamma\alpha=\epsilon'\mu\]
	and we are done.
\end{proof}

In all other cases, diagram $\mathfrak{D}$ is indeed a pushout on the level of classifying spaces:

\begin{lem}\label{23847}
	Assume that for any object $A\neq X$ either there are only arrows from $X$ to $A$ or there are only arrows from $A$ to $X$,
	but never both. Then the diagram $B(\mathfrak{D})$
	\begin{displaymath}\xymatrix{
		B\big(lk_\downa X\big)\ar[d]\ar[rr]&&
		\op{Cone}\big(B\big(lk_\downa X\big)\big)\ar[d]\\
		B\big(\calC^{-X}\big)\ar[rr]&&B(\calC)
	}\end{displaymath}
	is a pushout of spaces.
\end{lem}
\begin{proof}
	We claim that the nerve functor applied to the diagram $\mathfrak{D}$
	\begin{displaymath}\xymatrix{
		N\big(lk_\downa X\big)\ar[d]\ar[rr]&&
		N\big(\op{Coone}\big(\overline{lk}_\downa X*\underline{lk}_\downa X\big)\big)\ar[d]\\
		N\big(\calC^{-X}\big)\ar[rr]&&N(\calC)
	}\end{displaymath}
	yields a pushout in $\mathtt{SSET}$. Since the geometric realization functor $|?|\co \mathtt{SSET}\rightarrow\mathtt{TOP}$ is 
	left adjoint to the singular simplex functor, it preserves all colimits and in particular all pushouts. Therefore, applying the 
	geometric realization functor $|?|$ to the diagram $N(\mathfrak{D})$, we obtain a pushout in $\mathtt{TOP}$, as claimed in the 
	lemma.
	
	A simplex in $N(\calC)$ is just a string of composable arrows $A_0\rightarrow\ldots\rightarrow A_k$. 
	One can easily deduce from the assumption that whenever there are two occurences of $X$
	in such a string of composable arrows, then there cannot be objects different from $X$ in between. In other
	words, if $X$ occurs at all, then all the $X$ in the string are contained in a maximal substring
	of the form $X\rightarrow X\rightarrow\ldots\rightarrow X$ where all the arrows are (necessarily) $\id_X$.
	
	Assume now that we have a commutative diagram as follows:
	\begin{displaymath}\xymatrix{
		N\big(lk_\downa X\big)\ar[d]\ar[rr]&&
		N\big(\op{Coone}\big(\overline{lk}_\downa X*\underline{lk}_\downa X\big)\big)
		\ar[d]\ar[ddr]^f&\\
		N\big(\calC^{-X}\big)\ar[rr]\ar[rrrd]_{g}&&N(\calC)\ar@{-->}[dr]^h&\\
		&&&Z
	}\end{displaymath}
	We will show that $h$ is uniquely determined by $f$ and $g$. Assume $\sigma$ is a simplex in $N(\calC)$ given
	by a string of composable arrows $A_0\rightarrow\ldots\rightarrow A_k$. If all the $A_i$ are contained
	in the full subcategory $\calC^{-X}$ then $\sigma$ is a simplex in the simplicial subset
	$N\big(\calC^{-X}\big)$ and then necessarily $h(\sigma)=g(\sigma)$. On the other hand, assume
	that not all the $A_i$ are objects of $\calC^{-X}$, i.e.~at least one $A_i=X$. As pointed
	out above, $\sigma$ must be of the form
	\[B_0\rightarrow\ldots\rightarrow B_r\rightarrow X\rightarrow\ldots\rightarrow X\rightarrow C_0\rightarrow\ldots\rightarrow C_s\]
	where $B_i\neq X$ for all $i=0,\ldots,r$ and $C_j\neq X$ for all $j=0,\ldots,s$. All the $B_i$ are in the image of 
	$\overline{lk}_\downa(X)$ because, after composing, we get an arrow $B_i\rightarrow X$. Analogously, all the $C_j$ are in 
	the image of $\underline{lk}_\downa(X)$. Such a simplex $\sigma$ always lifts to a unique simplex $\bar{\sigma}$ along the map
	\[N\big(\op{Coone}\big(\overline{lk}_\downa X*\underline{lk}_\downa X\big)\big)\rightarrow N(\calC)\]
	Thus we have $h(\sigma)=f(\bar{\sigma})$. This proves uniqueness of $h$. Showing that $h$ actually defines a map of simplicial 
	sets is left to the reader.
\end{proof}

We combine Lemmas \ref{76169}, \ref{11228} and \ref{23847} to get:
\begin{prop}
	If the descending link $lk_\downa(X)$ is $(n-1)$-connected, then the pair $(\calC,\calC^{-X})$ is $n$-connected.
\end{prop}

More generally, let $\calX_0$ be the full subcategory of the category $\calX$ spanned by a collection of objects in $\calX$. 
Assume that $\op{Hom}_\calX(X,X)=\{\id_X\}$ for all objects $X$ in $\calX\setminus\calX_0$. Then $\calX$
can be built up from $\calX_0$ by successively adding objects. If all the descending links appearing this way are highly connected, 
then also the pair $(\calX,\calX_0)$ will be highly connected and, using the long exact homotopy sequence, we obtain the following:

\begin{thm}\label{34332}
	Let $x_0\in \calX_0$ be an object. Assume that each descending link is $n$-connected. Then, we have
	\[\pi_k(\calX_0,x_0)\cong\pi_k(\calX,x_0)\]
	for $k=0,\ldots,n$.
\end{thm}

We say that two objects $x_1$ and $x_2$ in $\calX\setminus\calX_0$ are independent
if there are no arrows $x_1\rightarrow x_2$ or $x_2\rightarrow x_1$ in $\calX$. This guarantees independence of $lk_\downa(x_1)$
and $lk_\downa(x_2)$ from the adding order of $x_1$ and $x_2$. 

Again, we can encode the adding order with the help of a Morse function $f$ which assigns to each object in $\calX\setminus\calX_0$
an element in a totally ordered set, e.g.~$\NN$. We require that objects with the same $f$-value are pairwise independent and we add
objects in order of increasing $f$-values.

\section{Operad groups}\label{97985}

In this section, we want to introduce our main objects of study, the operad groups. We first define the types of operads
we will be working with. We will then define operad groups to be the fundamental groups of the category of operators
naturally associated to operads. In the last subsection, we will discuss examples of operads and their corresponding operad
groups. We will recover some already well-known Thompson-like groups this way.

\subsection{Basic definitions}

\begin{defi}
	An operad $\calO$ consists of a set of colors $C$ and sets of operations $\calO(a_1,\ldots,a_n;b)$ for each
	finite ordered sequence $a_1,\ldots,a_n,b$ of colors in $C$ (the $a_i$ are the input colors and $b$ is the 
	output color) with $n\geq 1$ (allowing operations with no inputs is possible, but we won't consider such
	operads). See Figure \ref{91910} for a visualization of operations. There are composition maps
	(Figure \ref{91910})
	\begin{displaymath}\xymatrix{
		\calO(c_{11},\ldots,c_{1k_1};a_1)\times\ldots\times\calO(c_{n1},\ldots,c_{nk_n};a_n)\times\calO(a_1,\ldots,a_n;b)\ar[d]\\
		\calO(c_{11},\ldots,c_{1k_1},c_{21},\ldots,c_{nk_n};b)
	}\end{displaymath}
	denoted by $(\phi_1,\ldots,\phi_n)*\theta$. Composition is associative (Figure \ref{68615}):
	\begin{gather*}
		\big((\psi_{11},\ldots,\psi_{1k_1})*\phi_1,\ldots,(\psi_{n1},\ldots,\psi_{nk_n})*\phi_n\big)*\theta\\*
		\parallel\\*
		(\psi_{11},\ldots,\psi_{1k_1},\psi_{21},\ldots,\psi_{nk_n})*\big((\phi_1,\ldots,\phi_n)*\theta\big)
	\end{gather*}
	For each color $a$ there are distinguished unit elements $1_a\in\calO(a;a)$ such that
	\[(1_{a_1},\ldots,1_{a_n})*\theta=\theta=\theta*1_b\]
	for each operation $\theta$. Sometimes we call such an operad planar in order to distinguish it from the symmetric or 
	braided versions below.
	
	A symmetric/braided operad comes with additional maps (Figure \ref{82260})
	\[x\cdot\underline{\hspace{2mm}}\co \calO(a_1,\ldots,a_n;b)\rightarrow \calO(a_{1\into x},\ldots,a_{n\into x};b)\]
	for each $x$ in the symmetric group $S_n$ or in the braid group $B_n$ respectively. Here, $i\into x$ for $x\in S_n$ means
	plugging the element $i$ into the permutation $x$ which is considered as a bijection of the set
	$\{1,\ldots,n\}$. There is a canonical projection $B_n\rightarrow S_n$, so this makes sense also in the braided
	case. These maps are assumed to be actions:
	\[x\cdot(y\cdot\theta)=(xy)\cdot\theta\hspace{8mm}1\cdot\theta=\theta\]
	They also have to be equivariant with respect to composition (Figure \ref{58700}):
	\begin{gather*}
		(\phi_{1\into x},\ldots,\phi_{n\into x})*(x\cdot\theta)=\bar{x}\cdot\big((\phi_1,\ldots,\phi_n)*\theta\big)\\
		(y_1\cdot\phi_1,\ldots,y_n\cdot\phi_n)*\theta=(y_1,\ldots,y_n)\cdot\big((\phi_1,\ldots,\phi_n)*\theta\big)
	\end{gather*}
	Here, $\bar{x}$ is obtained from $x$ by replacing the $i$'th strand of $x$ by $n_i$ strands and $n_i$ is the 
	number of inputs of $\phi_{i\into x}$. Furthermore, $(y_1,\ldots,y_n)$ is the juxtaposition of the permutations
	resp.~braidings $y_i$.
\end{defi}

\ifthenelse{\boolean{drawtikz}}{
\begin{figure}
\begin{center}\begin{tikzpicture}[scale=0.35]
	\draw (0,0) -- node[right=15pt]{$\theta$} +(0,6) -- +(6,3) node[right]{$b$} -- +(0,0);
	\draw (0,1) -- +(-2,0) node[left]{$a_3$};
	\draw (0,3) -- +(-2,0) node[left]{$a_2$};
	\draw (0,5) -- +(-2,0) node[left]{$a_1$};
\end{tikzpicture}\hspace{10mm}\begin{tikzpicture}[scale=0.35]
	\draw (0,4.6) -- +(-1,0) node[left]{$c_{12}$};
	\draw (0,5.4) -- +(-1,0) node[left]{$c_{11}$};
	\draw (0,4.1) -- node[right=-3pt]{$\phi_1$} +(0,1.8) -- +(1.8,0.9) -- +(0,0);	
	\draw (0,3) -- +(-1,0) node[left]{$c_{21}$};
	\draw (0,2.1) -- node[right=-3pt]{$\phi_2$} +(0,1.8) -- +(1.8,0.9) -- +(0,0);	
	\draw (0,0.4) -- +(-1,0) node[left]{$c_{33}$};
	\draw (0,1) -- +(-1,0) node[left]{$c_{32}$};
	\draw (0,1.6) -- +(-1,0) node[left]{$c_{31}$};
	\draw (0,0.1) -- node[right=-3pt]{$\phi_3$} +(0,1.8) -- +(1.8,0.9) -- +(0,0);	
	\draw (3.8,0) -- node[right=15pt]{$\theta$} +(0,6) -- +(6,3) node[right]{$b$} -- +(0,0);
	\draw (3.8,1) -- node[above]{$a_3$} +(-2,0);
	\draw (3.8,3) -- node[above]{$a_2$} +(-2,0);
	\draw (3.8,5) -- node[above]{$a_1$} +(-2,0);
\end{tikzpicture}\end{center}
\caption{Visualization of an operation and composition of operations.}\label{91910}
\end{figure}
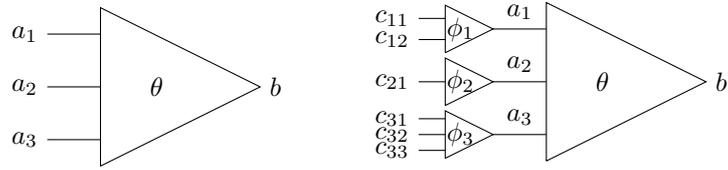

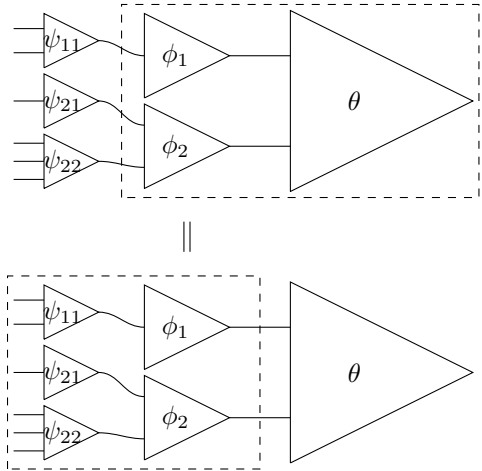
\begin{figure}
\begin{center}\begin{tikzpicture}[scale=0.4]
	\draw (0,13.6) -- +(-1,0);
	\draw (0,14.4) -- +(-1,0);
	\draw (0,13.1) -- node[right=-4pt]{$\psi_{11}$} +(0,1.8) -- +(1.8,0.9) -- +(0,0);
	\draw (0,12) -- +(-1,0);
	\draw (0,11.1) -- node[right=-4pt]{$\psi_{21}$} +(0,1.8) -- +(1.8,0.9) -- +(0,0);
	\draw (0,9.4) -- +(-1,0);
	\draw (0,10) -- +(-1,0);
	\draw (0,10.6) -- +(-1,0);
	\draw (0,9.1) -- node[right=-4pt]{$\psi_{22}$} +(0,1.8) -- +(1.8,0.9) -- +(0,0);
	\draw (3.3,9.1) -- node[right=3pt]{$\phi_2$} +(0,2.8) -- +(2.8,1.4) -- +(0,0);
	\draw (3.3,9.8) to[out=left,in=right] (1.8,10);
	\draw (3.3,11.2) to[out=left,in=right] (1.8,12);
	\draw (3.3,12.1) -- node[right=3pt]{$\phi_1$} +(0,2.8) -- +(2.8,1.4) -- +(0,0);
	\draw (3.3,13.5) to[out=left,in=right] (1.8,14);
	\draw (8.1,9) -- node[right=18pt]{$\theta$} +(0,6) -- +(6,3) -- +(0,0);
	\draw (8.1,10.5) -- +(-2,0);
	\draw (8.1,13.5) -- +(-2,0);
	\draw[dashed] (2.55,8.8) -- +(0,6.4) -- +(11.75,6.4) -- +(11.75,0) -- +(0,0);
	\draw (4.6,7) -- +(0,1);
	\draw (4.8,7) -- +(0,1);
	\draw[dashed] (-1.2,-0.2) -- +(0,6.4) -- +(8.3,6.4) -- +(8.3,0) -- +(0,0);	
	\draw (0,4.6) -- +(-1,0);
	\draw (0,5.4) -- +(-1,0);
	\draw (0,4.1) -- node[right=-4pt]{$\psi_{11}$} +(0,1.8) -- +(1.8,0.9) -- +(0,0);
	\draw (0,3) -- +(-1,0);
	\draw (0,2.1) -- node[right=-4pt]{$\psi_{21}$} +(0,1.8) -- +(1.8,0.9) -- +(0,0);
	\draw (0,0.4) -- +(-1,0);
	\draw (0,1) -- +(-1,0);
	\draw (0,1.6) -- +(-1,0);
	\draw (0,0.1) -- node[right=-4pt]{$\psi_{22}$} +(0,1.8) -- +(1.8,0.9) -- +(0,0);
	\draw (3.3,0.1) -- node[right=3pt]{$\phi_2$} +(0,2.8) -- +(2.8,1.4) -- +(0,0);
	\draw (3.3,0.8) to[out=left,in=right] (1.8,1);
	\draw (3.3,2.2) to[out=left,in=right] (1.8,3);
	\draw (3.3,3.1) -- node[right=3pt]{$\phi_1$} +(0,2.8) -- +(2.8,1.4) -- +(0,0);
	\draw (3.3,4.5) to[out=left,in=right] (1.8,5);
	\draw (8.1,0) -- node[right=18pt]{$\theta$} +(0,6) -- +(6,3) -- +(0,0);
	\draw (8.1,1.5) -- +(-2,0);
	\draw (8.1,4.5) -- +(-2,0);
\end{tikzpicture}\end{center}
\caption{Associativity.}\label{68615}
\end{figure}

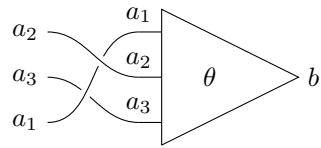
\begin{figure}
\begin{center}\begin{tikzpicture}[scale=0.3]
	\draw (0,0) -- node[right=12pt]{$\theta$} +(0,6) -- +(6,3) node[right]{$b$} -- +(0,0);
	\draw (0,1) -- +(-1,0) node[above]{$a_3$};
	\draw (0,3) -- +(-1,0) node[above]{$a_2$};
	\draw (0,5) -- +(-1,0) node[above]{$a_1$};
	\draw (-1,1) to[out=left,in=right] (-5,3) node[left]{$a_3$};
	\draw[white, line width=4pt] (-1,5) to[out=left,in=right] (-5,1) node[left]{$a_1$};
	\draw (-1,5) to[out=left,in=right] (-5,1) node[left]{$a_1$};
	\draw[white, line width=4pt] (-1,3) to[out=left,in=right] (-5,5) node[left]{$a_2$};
	\draw (-1,3) to[out=left,in=right] (-5,5) node[left]{$a_2$};
\end{tikzpicture}\end{center}
\caption{Action of the braid groups on the operations.}\label{82260}
\end{figure}

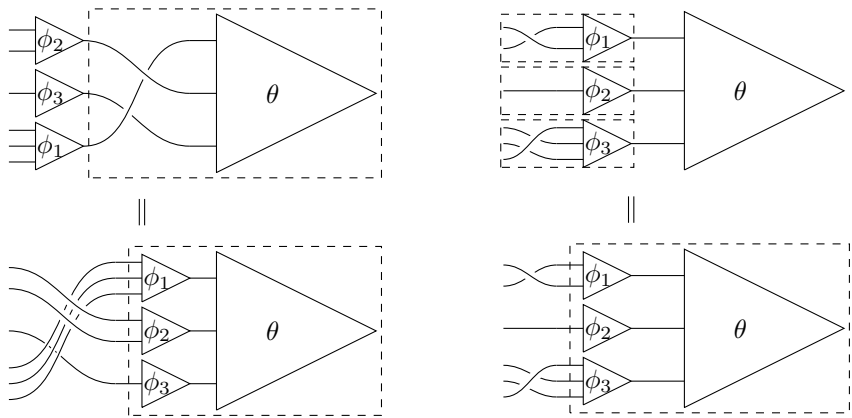
\begin{figure}
	\centering
	\begin{minipage}{0.5\textwidth}
	\centering
		\begin{tikzpicture}[scale=0.35]
			\draw (2.8,9) -- node[right=15pt]{$\theta$} +(0,6) -- +(6,3) -- +(0,0);
			\draw (2.8,10) -- +(-1,0);
			\draw (2.8,12) -- +(-1,0);
			\draw (2.8,14) -- +(-1,0);
			\draw[white, line width=4pt] (1.8,10) to[out=left,in=right] (-2.2,12);
			\draw (1.8,10) to[out=left,in=right] (-2.2,12);
			\draw[white, line width=4pt] (1.8,14) to[out=left,in=right] (-2.2,10);
			\draw (1.8,14) to[out=left,in=right] (-2.2,10);
			\draw[white, line width=4pt] (1.8,12) to[out=left,in=right] (-2.2,14);
			\draw (1.8,12) to[out=left,in=right] (-2.2,14);	
			\draw (-4,13.6) -- +(-1,0);
			\draw (-4,14.4) -- +(-1,0);
			\draw (-4,13.1) -- node[right=-3pt]{$\phi_2$} +(0,1.8) -- +(1.8,0.9) -- +(0,0);
			\draw (-4,12) -- +(-1,0);
			\draw (-4,11.1) -- node[right=-3pt]{$\phi_3$} +(0,1.8) -- +(1.8,0.9) -- +(0,0);
			\draw (-4,9.4) -- +(-1,0);
			\draw (-4,10) -- +(-1,0);
			\draw (-4,10.6) -- +(-1,0);
			\draw (-4,9.1) -- node[right=-3pt]{$\phi_1$} +(0,1.8) -- +(1.8,0.9) -- +(0,0);
			\draw[dashed] (-2,8.8) -- +(0,6.4) -- +(11,6.4) -- +(11,0) -- +(0,0);
			\draw (-0.1,7) -- +(0,1);
			\draw (0.1,7) -- +(0,1);
			\draw[dashed] (-0.5,-0.2) -- +(0,6.4) -- +(9.5,6.4) -- +(9.5,0) -- +(0,0);	
			\draw (0,4.4) -- +(-1,0);
			\draw (0,5) -- +(-1,0);
			\draw (0,5.6) -- +(-1,0);
			\draw (0,4.1) -- node[right=-3pt]{$\phi_1$} +(0,1.8) -- +(1.8,0.9) -- +(0,0);
			\draw (0,2.6) -- +(-1,0);
			\draw (0,3.4) -- +(-1,0);
			\draw (0,2.1) -- node[right=-3pt]{$\phi_2$} +(0,1.8) -- +(1.8,0.9) -- +(0,0);
			\draw (0,1) -- +(-1,0);
			\draw (0,0.1) -- node[right=-3pt]{$\phi_3$} +(0,1.8) -- +(1.8,0.9) -- +(0,0);
			\draw (2.8,0) -- node[right=15pt]{$\theta$} +(0,6) -- +(6,3) -- +(0,0);
			\draw (2.8,1) -- +(-1,0);
			\draw (2.8,3) -- +(-1,0);
			\draw (2.8,5) -- +(-1,0);
			\draw[white, line width=4pt] (-1,1) to[out=left,in=right] (-5,3);
			\draw (-1,1) to[out=left,in=right] (-5,3);
			\draw[white, line width=2pt] (-1,5.6) to[out=left,in=right] (-5,1.6);
			\draw (-1,5.6) to[out=left,in=right] (-5,1.6);
			\draw[white, line width=2pt] (-1,5) to[out=left,in=right] (-5,1);
			\draw (-1,5) to[out=left,in=right] (-5,1);
			\draw[white, line width=2pt] (-1,4.4) to[out=left,in=right] (-5,0.4);
			\draw (-1,4.4) to[out=left,in=right] (-5,0.4);	
			\draw[white, line width=4pt] (-1,3.4) to[out=left,in=right] (-5,5.4);
			\draw (-1,3.4) to[out=left,in=right] (-5,5.4);
			\draw[white, line width=4pt] (-1,2.6) to[out=left,in=right] (-5,4.6);
			\draw (-1,2.6) to[out=left,in=right] (-5,4.6);
		\end{tikzpicture}
	\end{minipage}\hfill
	\begin{minipage}{0.5\textwidth}
	\centering
		\begin{tikzpicture}[scale=0.35]
			\draw (0,13.6) -- +(-1,0);
			\draw (0,14.4) -- +(-1,0);
			\draw (0,13.1) -- node[right=-3pt]{$\phi_1$} +(0,1.8) -- +(1.8,0.9) -- +(0,0);
			\draw (0,12) -- +(-1,0);
			\draw (0,11.1) -- node[right=-3pt]{$\phi_2$} +(0,1.8) -- +(1.8,0.9) -- +(0,0);
			\draw (0,9.4) -- +(-1,0);
			\draw (0,10) -- +(-1,0);
			\draw (0,10.6) -- +(-1,0);
			\draw (0,9.1) -- node[right=-3pt]{$\phi_3$} +(0,1.8) -- +(1.8,0.9) -- +(0,0);
			\draw (3.8,9) -- node[right=15pt]{$\theta$} +(0,6) -- +(6,3) -- +(0,0);
			\draw (3.8,10) -- +(-2,0);
			\draw (3.8,12) -- +(-2,0);
			\draw (3.8,14) -- +(-2,0);
			\draw[white, line width=4pt] (-1,14.4) to[out=left,in=right] (-3,13.6);
			\draw (-1,14.4) to[out=left,in=right] (-3,13.6);
			\draw[white, line width=4pt] (-1,13.6) to[out=left,in=right] (-3,14.4);
			\draw (-1,13.6) to[out=left,in=right] (-3,14.4);
			\draw (-1,12) to[out=left,in=right] (-3,12);
			\draw[white, line width=4pt]  (-1,10) to[out=left,in=right] (-3,10.6);
			\draw (-1,10) to[out=left,in=right] (-3,10.6);
			\draw[white, line width=4pt]  (-1,9.4) to[out=left,in=right] (-3,10);
			\draw (-1,9.4) to[out=left,in=right] (-3,10);
			\draw[white, line width=4pt]  (-1,10.6) to[out=left,in=right] (-3,9.4);
			\draw (-1,10.6) to[out=left,in=right] (-3,9.4);
			\draw[dashed] (-3.1,9.1) -- +(0,1.8) -- +(5,1.8) -- +(5,0) -- +(0,0);
			\draw[dashed] (-3.1,11.1) -- +(0,1.8) -- +(5,1.8) -- +(5,0) -- +(0,0);
			\draw[dashed] (-3.1,13.1) -- +(0,1.8) -- +(5,1.8) -- +(5,0) -- +(0,0);
			\draw (1.7,7) -- +(0,1);
			\draw (1.9,7) -- +(0,1);
			\draw[dashed] (-0.5,-0.2) -- +(0,6.4) -- +(10.5,6.4) -- +(10.5,0) -- +(0,0);
			\draw (0,4.6) -- +(-1,0);
			\draw (0,5.4) -- +(-1,0);
			\draw (0,4.1) -- node[right=-3pt]{$\phi_1$} +(0,1.8) -- +(1.8,0.9) -- +(0,0);
			\draw (0,3) -- +(-1,0);
			\draw (0,2.1) -- node[right=-3pt]{$\phi_2$} +(0,1.8) -- +(1.8,0.9) -- +(0,0);
			\draw (0,0.4) -- +(-1,0);
			\draw (0,1) -- +(-1,0);
			\draw (0,1.6) -- +(-1,0);
			\draw (0,0.1) -- node[right=-3pt]{$\phi_3$} +(0,1.8) -- +(1.8,0.9) -- +(0,0);
			\draw (3.8,0) -- node[right=15pt]{$\theta$} +(0,6) -- +(6,3) -- +(0,0);
			\draw (3.8,1) -- +(-2,0);
			\draw (3.8,3) -- +(-2,0);
			\draw (3.8,5) -- +(-2,0);
			\draw[white, line width=4pt] (-1,5.4) to[out=left,in=right] (-3,4.6);
			\draw (-1,5.4) to[out=left,in=right] (-3,4.6);
			\draw[white, line width=4pt] (-1,4.6) to[out=left,in=right] (-3,5.4);
			\draw (-1,4.6) to[out=left,in=right] (-3,5.4);
			\draw (-1,3) to[out=left,in=right] (-3,3);
			\draw[white, line width=4pt] (-1,1) to[out=left,in=right] (-3,1.6);
			\draw (-1,1) to[out=left,in=right] (-3,1.6);
			\draw[white, line width=4pt] (-1,0.4) to[out=left,in=right] (-3,1);
			\draw (-1,0.4) to[out=left,in=right] (-3,1);
			\draw[white, line width=4pt] (-1,1.6) to[out=left,in=right] (-3,0.4);
			\draw (-1,1.6) to[out=left,in=right] (-3,0.4);
		\end{tikzpicture}
	\end{minipage}
	\caption{First and second equivariance property.}\label{58700}
\end{figure}
}{TIKZ}

\begin{rem}\label{73668}
	There is an equivalent way of writing the composition, namely with so-called partial compositions. The
	$i$'th partial compositions
	\[*_i\co\calO(c_1,\ldots,c_k;a_i)\times\calO(a_1,\ldots,a_n;b)\rightarrow
		\calO(a_1,\ldots,a_{i-1},c_1,\ldots,c_k,a_{i+1},\ldots,a_n;b)\]
	are defined as
	\[\phi*_i\theta:=\big(1_{a_1},\ldots,1_{a_{i-1}},\phi,1_{a_{i+1}},\ldots,1_{a_n}\big)*\theta\]
	Conversely, one could define operads via partial compositions and obtain the usual composition from
	successive partial compositions.
\end{rem}

The planar operads, symmetric operads and braided operads can be organized into categories $\mathtt{OP}$, $\mathtt{SYM.OP}$ and 
$\mathtt{BRA.OP}$ respectively. Denote by $\mathtt{MON}$, $\mathtt{SYM.MON}$ and $\mathtt{BRA.MON}$ the categories of monoidal 
categories, symmetric monoidal categories and braided monoidal categories respectively. There are functors
\begin{gather*}
	\op{End}\co\mathtt{MON}\longrightarrow\mathtt{OP}\\*
	\op{End}\co\mathtt{SYM.MON}\longrightarrow\mathtt{SYM.OP}\\*
	\op{End}\co\mathtt{BRA.MON}\longrightarrow\mathtt{BRA.OP}
\end{gather*}
assigning to each (symmetric/braided) monoidal category $\calC$ an operad $\op{End}(\calC)$, called the endomorphism operad.
The colors of $\op{End}(\calC)$ are the objects of $\calC$ and the sets of operations are given by
\[\op{End}(\calC)(a_1,\ldots,a_n;b)=\op{Hom}_\mathcal{C}(a_1\otimes\ldots\otimes a_n,b)\]
Composition in $\op{End}(\calC)$ is induced by the composition in $\mathcal{C}$ in the obvious way. The unit element
in $\op{End}(\calC)(a;a)$ is the identity $\id_a\co a\rightarrow a$ in $\mathcal{C}$. In the symmetric or braided case, $\calC$
comes with additional natural isomorphisms $\gamma_{X,Y}\co X\otimes Y\rightarrow Y\otimes X$. These can be used to define
the action of the symmetric resp.~braid groups on the sets of operations.
In the theory of operads, these endomorphism operads play an important role since morphisms of operads
\[\calO\rightarrow\op{End}(\calC)\]
are representations of or algebras over the operad $\calO$.

The functors $\op{End}$ have left adjoints
\begin{gather*}
	\calS\co\mathtt{OP}\longrightarrow\mathtt{MON}\\*
	\calS\co\mathtt{SYM.OP}\longrightarrow\mathtt{SYM.MON}\\*
	\calS\co\mathtt{BRA.OP}\longrightarrow\mathtt{BRA.MON}
\end{gather*}
The (symmetric/braided) monoidal category $\calS(\calO)$ is called the category of operators. We will define these categories
explicitly. We start with the planar case and then use it to define the braided case. The symmetric case is similar
to the braided case.

\ifthenelse{\boolean{drawtikz}}{
\begin{figure}
\begin{center}\begin{tikzpicture}[scale=0.4]
	\draw (0,4.6) -- +(-1,0);
	\draw (0,5.4) -- +(-1,0);
	\draw (0,4.1) -- node[right=-2pt]{$\phi_1$} +(0,1.8) -- +(1.8,0.9) -- +(0,0);
	\draw (0,3) -- +(-1,0);
	\draw (0,2.1) -- node[right=-2pt]{$\phi_2$} +(0,1.8) -- +(1.8,0.9) -- +(0,0);
	\draw (0,0.4) -- +(-1,0);
	\draw (0,1) -- +(-1,0);
	\draw (0,1.6) -- +(-1,0);
	\draw (0,0.1) -- node[right=-2pt]{$\phi_3$} +(0,1.8) -- +(1.8,0.9) -- +(0,0);
	\draw (-1,5.4) to[out=left,in=right] (-6,5.4);
	\draw[white, line width=4pt] (-3.8,2.4) to[out=left,in=right] (-6,3);
	\draw (-3.8,2.4) to[out=left,in=right] (-6,3);	
	\draw[white, line width=4pt] (-2,4.6) to[out=left,in=right] (-6,1.6);
	\draw (-2,4.6) to[out=left,in=right] (-6,1.6);
	\draw[white, line width=4pt] (-1,1.6) to[out=left,in=right] (-6,1);
	\draw (-1,1.6) to[out=left,in=right] (-6,1);
	\draw[white, line width=4pt] (-1,1) to[out=left,in=right] (-6,4.6);
	\draw (-1,1) to[out=left,in=right] (-6,4.6);
	\draw[white, line width=4pt] (-1.6,3) to[out=left,in=right] (-3.8,2.4);
	\draw (-1.6,3) to[out=left,in=right] (-3.8,2.4);
	\draw (-1,0.4) to[out=left,in=right] (-6,0.4);
	\draw (-2,4.6) -- (-1,4.6);
	\draw (-1.6,3) -- (-1,3);
	\draw[dotted] (-1,0) -- +(0,6);
\end{tikzpicture}\end{center}
\caption{Arrow in $\calS(\calO)$.}\label{05041}
\end{figure}

\begin{figure}
\begin{center}\begin{tikzpicture}[scale=0.5]
	\draw (2,4.6) -- +(-1,0);
	\draw (2,5.4) -- +(-1,0);
	\draw (2,4.1) -- node[right=0pt]{$\phi_1$} +(0,1.8) -- +(1.8,0.9) -- +(0,0);
	\draw (2,3) -- +(-1,0);
	\draw (2,2.1) -- node[right=0pt]{$\phi_2$} +(0,1.8) -- +(1.8,0.9) -- +(0,0);
	\draw (2,0.4) -- +(-1,0);
	\draw (2,1) -- +(-1,0);
	\draw (2,1.6) -- +(-1,0);
	\draw (2,0.1) -- node[right=0pt]{$\phi_3$} +(0,1.8) -- +(1.8,0.9) -- +(0,0);
	\draw[white, line width=4pt] (1,4.6) to[out=left,in=right] (-1,5.4);
	\draw (1,4.6) to[out=left,in=right] (-1,5.4);
	\draw[white, line width=4pt] (1,5.4) to[out=left,in=right] (-1,4.6);
	\draw (1,5.4) to[out=left,in=right] (-1,4.6);
	\draw (1,3) to[out=left,in=right] (-1,3);
	\draw[white, line width=4pt] (1,0.4) to[out=left,in=right] (-1,1);
	\draw (1,0.4) to[out=left,in=right] (-1,1);
	\draw[white, line width=4pt] (1,1.6) to[out=left,in=right] (-1,0.4);
	\draw (1,1.6) to[out=left,in=right] (-1,0.4);
	\draw[white, line width=4pt] (1,1) to[out=left,in=right] (-1,1.6);
	\draw (1,1) to[out=left,in=right] (-1,1.6);
	\draw[white, line width=4pt] (-2.7,2.7) to[out=left,in=right] (-4,3);
	\draw (-2.7,2.7) to[out=left,in=right] (-4,3);
	\draw[white, line width=4pt] (-1.5,5.4) to[out=left,in=right] (-4,1.6);
	\draw (-1.5,5.4) to[out=left,in=right] (-4,1.6);
	\draw[white, line width=4pt] (-1,4.6) to[out=left,in=right] (-4,5.4);
	\draw (-1,4.6) to[out=left,in=right] (-4,5.4);
	\draw[white, line width=4pt] (-1,1.6) to[out=left,in=right] (-4,4.6);
	\draw (-1,1.6) to[out=left,in=right] (-4,4.6);
	\draw[white, line width=4pt] (-1,1) to[out=left,in=right] (-4,0.4);
	\draw (-1,1) to[out=left,in=right] (-4,0.4);
	\draw[white, line width=4pt] (-1,0.4) to[out=left,in=right] (-4,1);
	\draw (-1,0.4) to[out=left,in=right] (-4,1);
	\draw[white, line width=4pt] (-1.4,3) to[out=left,in=right] (-2.7,2.7);
	\draw (-1.4,3) to[out=left,in=right] (-2.7,2.7);
	\draw (-1.5,5.4) -- (-1,5.4);
	\draw (-1.4,3) -- (-1,3);
	\draw[dashed] (-1,0.1) -- +(0,1.8) -- +(4.9,1.8) -- +(4.9,0) -- +(0,0);
	\draw[dashed] (-1,2.1) -- +(0,1.8) -- +(4.9,1.8) -- +(4.9,0) -- +(0,0);
	\draw[dashed] (-1,4.1) -- +(0,1.8) -- +(4.9,1.8) -- +(4.9,0) -- +(0,0);
	
	\draw (0.06,-0.6) -- +(0,-0.8);
	\draw (-0.06,-0.6) -- +(0,-0.8);
	
	\draw (2,-3.4) -- +(-1,0);
	\draw (2,-2.6) -- +(-1,0);
	\draw (2,-3.9) -- node[right=0pt]{$\phi_1$} +(0,1.8) -- +(1.8,0.9) -- +(0,0);
	\draw (2,-5) -- +(-1,0);
	\draw (2,-5.9) -- node[right=0pt]{$\phi_2$} +(0,1.8) -- +(1.8,0.9) -- +(0,0);
	\draw (2,-7.6) -- +(-1,0);
	\draw (2,-7) -- +(-1,0);
	\draw (2,-6.4) -- +(-1,0);
	\draw (2,-7.9) -- node[right=0pt]{$\phi_3$} +(0,1.8) -- +(1.8,0.9) -- +(0,0);	
	\draw[white, line width=4pt] (-2.7,-5.3) to[out=left,in=right] (-4,-5);
	\draw (-2.7,-5.3) to[out=left,in=right] (-4,-5);
	\draw[white, line width=4pt] (-1.5,-2.6) to[out=left,in=right] (-4,-6.4);
	\draw (-1.5,-2.6) to[out=left,in=right] (-4,-6.4);
	\draw[white, line width=4pt] (-1,-3.4) to[out=left,in=right] (-4,-2.6);
	\draw (-1,-3.4) to[out=left,in=right] (-4,-2.6);
	\draw[white, line width=4pt] (-1,-6.4) to[out=left,in=right] (-4,-3.4);
	\draw (-1,-6.4) to[out=left,in=right] (-4,-3.4);
	\draw[white, line width=4pt] (-1,-7) to[out=left,in=right] (-4,-7.6);
	\draw (-1,-7) to[out=left,in=right] (-4,-7.6);
	\draw[white, line width=4pt] (-1,-7.6) to[out=left,in=right] (-4,-7);
	\draw (-1,-7.6) to[out=left,in=right] (-4,-7);
	\draw[white, line width=4pt] (-1.4,-5) to[out=left,in=right] (-2.7,-5.3);
	\draw (-1.4,-5) to[out=left,in=right] (-2.7,-5.3);
	\draw (-1.5,-2.6) -- (-1,-2.6);
	\draw (-1.4,-5) -- (-1,-5);
	\draw[white, line width=4pt] (1,-3.4) to[out=left,in=right] (-1,-2.6);
	\draw (1,-3.4) to[out=left,in=right] (-1,-2.6);
	\draw[white, line width=4pt] (1,-2.6) to[out=left,in=right] (-1,-3.4);
	\draw (1,-2.6) to[out=left,in=right] (-1,-3.4);
	\draw (1,-5) to[out=left,in=right] (-1,-5);
	\draw[white, line width=4pt] (1,-7.6) to[out=left,in=right] (-1,-7);
	\draw (1,-7.6) to[out=left,in=right] (-1,-7);
	\draw[white, line width=4pt] (1,-6.4) to[out=left,in=right] (-1,-7.6);
	\draw (1,-6.4) to[out=left,in=right] (-1,-7.6);
	\draw[white, line width=4pt] (1,-7) to[out=left,in=right] (-1,-6.4);
	\draw (1,-7) to[out=left,in=right] (-1,-6.4);
	\draw[dotted] (-1,-8) -- +(0,6);
	\draw[dotted] (1,-8) -- +(0,6);
	
	\draw (0.06,-8.6) -- +(0,-0.8);
	\draw (-0.06,-8.6) -- +(0,-0.8);
	
	\draw (2,-11.4) -- +(-1,0);
	\draw (2,-10.6) -- +(-1,0);
	\draw (2,-11.9) -- node[right=0pt]{$\phi_1$} +(0,1.8) -- +(1.8,0.9) -- +(0,0);
	\draw (2,-13) -- +(-1,0);
	\draw (2,-13.9) -- node[right=0pt]{$\phi_2$} +(0,1.8) -- +(1.8,0.9) -- +(0,0);
	\draw (2,-15.6) -- +(-1,0);
	\draw (2,-15) -- +(-1,0);
	\draw (2,-14.4) -- +(-1,0);
	\draw (2,-15.9) -- node[right=0pt]{$\phi_3$} +(0,1.8) -- +(1.8,0.9) -- +(0,0);		
	\draw (1,-10.6) to[out=left,in=right] (-4,-10.6);
	\draw[white, line width=4pt] (-1.8,-13.6) to[out=left,in=right] (-4,-13);
	\draw (-1.8,-13.6) to[out=left,in=right] (-4,-13);	
	\draw[white, line width=4pt] (0,-11.4) to[out=left,in=right] (-4,-14.4);
	\draw (0,-11.4) to[out=left,in=right] (-4,-14.4);
	\draw[white, line width=4pt] (1,-14.4) to[out=left,in=right] (-4,-15);
	\draw (1,-14.4) to[out=left,in=right] (-4,-15);
	\draw[white, line width=4pt] (1,-15) to[out=left,in=right] (-4,-11.4);
	\draw (1,-15) to[out=left,in=right] (-4,-11.4);
	\draw[white, line width=4pt] (0.4,-13) to[out=left,in=right] (-1.8,-13.6);
	\draw (0.4,-13) to[out=left,in=right] (-1.8,-13.6);
	\draw (1,-15.6) to[out=left,in=right] (-4,-15.6);
	\draw (0,-11.4) -- (1,-11.4);
	\draw (0.4,-13) -- (1,-13);
	\draw[dotted] (1,-16) -- +(0,6);
\end{tikzpicture}\end{center}
\caption{Equivalence in $\calS(\calO)$.}\label{46525}
\end{figure}

\begin{figure}
\begin{center}\begin{tikzpicture}[scale=0.5]
	\draw (0,4.6) -- +(-1,0);
	\draw (0,5.4) -- +(-1,0);
	\draw (0,4.1) -- node[right=0pt]{$\phi_1$} +(0,1.8) -- +(1.8,0.9) -- +(0,0);
	\draw (0,3) -- +(-1,0);
	\draw (0,2.1) -- node[right=0pt]{$\phi_2$} +(0,1.8) -- +(1.8,0.9) -- +(0,0);
	\draw (0,0.4) -- +(-1,0);
	\draw (0,1) -- +(-1,0);
	\draw (0,1.6) -- +(-1,0);
	\draw (0,0.1) -- node[right=0pt]{$\phi_3$} +(0,1.8) -- +(1.8,0.9) -- +(0,0);
	\draw (-1,5.4) to[out=left,in=right] (-4,5.4);
	\draw[white, line width=4pt] (-2.8,2.8) to[out=left,in=right] (-4,3);
	\draw (-2.8,2.8) to[out=left,in=right] (-4,3);
	\draw[white, line width=4pt] (-2,4.6) to[out=left,in=right] (-4,1.6);
	\draw (-2,4.6) to[out=left,in=right] (-4,1.6);
	\draw[white, line width=4pt] (-1,1.6) to[out=left,in=right] (-4,1);
	\draw (-1,1.6) to[out=left,in=right] (-4,1);
	\draw[white, line width=4pt] (-1,1) to[out=left,in=right] (-4,4.6);
	\draw (-1,1) to[out=left,in=right] (-4,4.6);
	\draw[white, line width=4pt] (-1.6,3) to[out=left,in=right] (-2.8,2.8);
	\draw (-1.6,3) to[out=left,in=right] (-2.8,2.8);
	\draw (-1,0.4) to[out=left,in=right] (-4,0.4);
	\draw (-2,4.6) -- (-1,4.6);
	\draw (-1.6,3) -- (-1,3);
	\draw (5.3,0.1) -- node[right=6pt]{$\psi_2$} +(0,2.8) -- +(2.8,1.4) -- +(0,0);
	\draw (5.3,3.1) -- node[right=6pt]{$\psi_1$} +(0,2.8) -- +(2.8,1.4) -- +(0,0);
	\draw (5.3,4.5) -- +(-1.5,0);
	\draw (5.3,0.8) -- +(-1.5,0);
	\draw (5.3,2.2) -- +(-1.5,0);
	\draw (3.8,4.5) to[out=left,in=right] (1.8,5);
	\draw[white, line width=4pt] (3.8,2.2) to[out=left,in=right] (1.8,1);
	\draw (3.8,2.2) to[out=left,in=right] (1.8,1);
	\draw[white, line width=4pt] (3.8,0.8) to[out=left,in=right] (1.8,3);
	\draw (3.8,0.8) to[out=left,in=right] (1.8,3);
	\draw[dotted] (-1,0) -- +(0,6);
	\draw[dashed] (1.8,0) -- +(0,6);
	\draw[dotted] (3.8,0) -- +(0,6);
	
	\draw (1.74,-0.6) -- +(0,-0.8);
	\draw (1.86,-0.6) -- +(0,-0.8);
	
	\draw (2,-3.4) -- +(-1,0);
	\draw (2,-2.6) -- +(-1,0);
	\draw (2,-3.9) -- node[right=0pt]{$\phi_1$} +(0,1.8) -- +(1.8,0.9) -- +(0,0);
	\draw (2,-7) -- +(-1,0);
	\draw (2,-5.9) -- node[right=0pt]{$\phi_3$} +(0,1.8) -- +(1.8,0.9) -- +(0,0);
	\draw (2,-5.6) -- +(-1,0);
	\draw (2,-5) -- +(-1,0);
	\draw (2,-4.4) -- +(-1,0);
	\draw (2,-7.9) -- node[right=0pt]{$\phi_2$} +(0,1.8) -- +(1.8,0.9) -- +(0,0);
	\draw (-1,-2.6) to[out=left,in=right] (-4,-2.6);
	\draw[white, line width=4pt] (-2.8,-5.2) to[out=left,in=right] (-4,-5);
	\draw (-2.8,-5.2) to[out=left,in=right] (-4,-5);
	\draw[white, line width=4pt] (-2,-3.4) to[out=left,in=right] (-4,-6.4);
	\draw (-2,-3.4) to[out=left,in=right] (-4,-6.4);
	\draw[white, line width=4pt] (-1,-6.4) to[out=left,in=right] (-4,-7);
	\draw (-1,-6.4) to[out=left,in=right] (-4,-7);
	\draw[white, line width=4pt] (-1,-7) to[out=left,in=right] (-4,-3.4);
	\draw (-1,-7) to[out=left,in=right] (-4,-3.4);
	\draw[white, line width=4pt] (-1.6,-5) to[out=left,in=right] (-2.8,-5.2);
	\draw (-1.6,-5) to[out=left,in=right] (-2.8,-5.2);
	\draw (-1,-7.6) to[out=left,in=right] (-4,-7.6);
	\draw (-2,-3.4) -- (-1,-3.4);
	\draw (-1.6,-5) -- (-1,-5);
	\draw (1,-2.6) to[out=left,in=right] (-1,-2.6);
	\draw (1,-3.4) to[out=left,in=right] (-1,-3.4);
	\draw[white, line width=4pt] (1,-4.4) to[out=left,in=right] (-1,-6.4);
	\draw (1,-4.4) to[out=left,in=right] (-1,-6.4);
	\draw[white, line width=4pt] (1,-5) to[out=left,in=right] (-1,-7);
	\draw (1,-5) to[out=left,in=right] (-1,-7);
	\draw[white, line width=4pt] (1,-5.6) to[out=left,in=right] (-1,-7.6);
	\draw (1,-5.6) to[out=left,in=right] (-1,-7.6);
	\draw[white, line width=4pt] (1,-7) to[out=left,in=right] (-1,-5);
	\draw (1,-7) to[out=left,in=right] (-1,-5);	
	\draw (5.3,-7.9) -- node[right=6pt]{$\psi_2$} +(0,2.8) -- +(2.8,1.4) -- +(0,0);
	\draw (5.3,-4.9) -- node[right=6pt]{$\psi_1$} +(0,2.8) -- +(2.8,1.4) -- +(0,0);
	\draw (5.3,-7.2) to[out=left,in=right] (3.8,-7);
	\draw (5.3,-5.8) to[out=left,in=right] (3.8,-5);
	\draw (5.3,-3.5) to[out=left,in=right] (3.8,-3);
	\draw[dotted] (-1,-8) -- +(0,6);
	\draw[dotted] (1,-8) -- +(0,6);
	\draw[dotted] (3.8,-8) -- +(0,6);
	
	\draw (1.74,-8.6) -- +(0,-0.8);
	\draw (1.86,-8.6) -- +(0,-0.8);
	
	\draw (2,-11.4) -- +(-1,0);
	\draw (2,-10.6) -- +(-1,0);
	\draw (2,-11.9) -- node[right=0pt]{$\phi_1$} +(0,1.8) -- +(1.8,0.9) -- +(0,0);
	\draw (2,-15) -- +(-1,0);
	\draw (2,-13.9) -- node[right=0pt]{$\phi_3$} +(0,1.8) -- +(1.8,0.9) -- +(0,0);
	\draw (2,-13.6) -- +(-1,0);
	\draw (2,-13) -- +(-1,0);
	\draw (2,-12.4) -- +(-1,0);
	\draw (2,-15.9) -- node[right=0pt]{$\phi_2$} +(0,1.8) -- +(1.8,0.9) -- +(0,0);
	\draw (1,-10.6) to[out=left,in=right] (-4,-10.6);
	\draw[white, line width=4pt] (1,-13.6) to[out=left,in=right] (-4,-15);
	\draw (1,-13.6) to[out=left,in=right] (-4,-15);
	\draw[white, line width=4pt] (1,-12.4) to[out=left,in=right] (-4,-13.6);
	\draw (1,-12.4) to[out=left,in=right] (-4,-13.6);
	\draw[white, line width=4pt] (1,-15) to[out=left,in=right] (-4,-12.4);
	\draw (1,-15) to[out=left,in=right] (-4,-12.4);
	\draw[white, line width=4pt] (0,-11.4) to[out=left,in=right] (-4,-13);
	\draw (0,-11.4) to[out=left,in=right] (-4,-13);
	\draw[white, line width=4pt] (1,-13) to[out=left,in=right] (-4,-11.4);
	\draw (1,-13) to[out=left,in=right] (-4,-11.4);
	\draw (0,-11.4) -- (1,-11.4);	
	\draw (5.3,-15.9) -- node[right=6pt]{$\psi_2$} +(0,2.8) -- +(2.8,1.4) -- +(0,0);
	\draw (5.3,-12.9) -- node[right=6pt]{$\psi_1$} +(0,2.8) -- +(2.8,1.4) -- +(0,0);	
	\draw (5.3,-15.2) to[out=left,in=right] (3.8,-15);
	\draw (5.3,-13.8) to[out=left,in=right] (3.8,-13);
	\draw (5.3,-11.5) to[out=left,in=right] (3.8,-11);
	\draw[dotted] (1,-16) -- +(0,6);
	\draw[dashed] (2,-11.9) -- +(3.3,-1) -- +(6.2,-1) -- +(6.2,1.8) -- 
		+(0,1.8) -- +(-0.5,1.8) -- +(-0.5,0) -- +(0,0);
	\draw[dashed] (2,-12.1) -- +(3.3,-1) -- +(6.2,-1) -- +(6.2,-3.8) -- 
		+(0,-3.8) -- +(-0.5,-3.8) -- +(-0.5,0) -- +(0,0);
\end{tikzpicture}\end{center}
\caption{Composition in $\calS(\calO)$.}\label{56806}
\end{figure}
}{TIKZ}

So let $\calO$ be a planar operad with a set of colors $C$. The objects of $\calS(\calO)$ are free words
in the colors, i.e.~finite sequences of colors in $C$. An arrow in $\calS(\calO)$ is a finite sequence
of operations in $\calO$: If $X_1,\ldots,X_n$ are operations in $\calO$, the (ordered) input colors of $X_i$
are $(c_i^1,\ldots,c_i^{k_i})$ and the output color of $X_i$ is $d_i$, then the $X_i$ give an arrow
\[(X_1,\ldots,X_n)\co (c_1^1,\ldots,c_1^{k_1},c_2^1,\ldots,c_n^{k_n})\rightarrow(d_1,\ldots,d_n)\]
in $\calS(\calO)$. Composition is induced by the composition in the operad $\calO$ and the identities are given
by the identity operations in $\calO$. The tensor product is given by juxtaposition.

Now let $\calO$ be a braided operad with set of colors $C$. By forgetting the action of the braid groups, we 
get a planar operad $\calO_{\mathrm{pl}}$. The braided monoidal category $\calS(\calO)$ is a certain product 
$\mathfrak{Braid}(C)\boxtimes\calS(\calO_\mathrm{pl})$. The objects of $\calS(\calO)$ are once more finite sequences of 
colors in $C$. Arrows in $\calS(\calO)$ are equivalence classes of pairs 
$(\beta,X)\in\mathfrak{Braid}(C)\times\calS(\calO_\mathrm{pl})$ consisting of a $C$-colored braid $\beta$ and a sequence 
$X=(X_1,\ldots,X_n)$ of operations of $\calO$ where the codomain of $\beta$ equals the domain of $X$
(Figure \ref{05041}). The equivalence relation on such pairs is the following: Let $(\beta,X)$ be such a pair with 
$X=(X_1,\ldots,X_n)$. For each $i=1,\ldots,n$ let $\sigma_i$ be a $C$-colored braid such that $\sigma_i\cdot X_i$ is defined.
Let $\sigma:=\sigma_1\otimes\ldots\otimes \sigma_n$ and define
\[\sigma\cdot(\beta,X):=\big(\beta*\sigma^{-1},(\sigma_1\cdot X_1,\ldots,\sigma_n\cdot X_n)\big)\]
We require $(\beta,X)$ and $(\beta',X')$ to be equivalent if there exists a $\sigma$ as above such that
$(\beta',X')=\sigma\cdot(\beta,X)$. In other words, it is the smallest equivalence relation respecting juxtaposition and which
is generated by the relation 
\[(\beta*\sigma,X)\sim(\beta,\sigma\cdot X)\]
with $X$ a single operation. 
This is visualized in Figure \ref{46525}.

Composition in $\calS(\calO)$ is defined on representatives $(\beta,X)$ and $(\delta,Y)$. Loosely speaking, we push the sequence 
$X$ of operations through the colored braid $\delta$ just as in the definition of equivariance for operads, obtain another 
colored braid $X{\curvearrowright}\delta$ which is obtained from $\delta$ by multiplying the strands according to $X$ and another
sequence of operations $X{\curvearrowleft}\delta$ which is obtained from $X$ by permuting the operations according to
$\delta$, and finally compose the left and right side in $\mathfrak{Braid}(C)$ and $\calS(\calO_\mathrm{pl})$
respectively:
\[(\beta,X)*(\delta,Y):=\big(\beta*X{\curvearrowright}\delta,X{\curvearrowleft}\delta*Y\big)\]
See Figure \ref{56806} for a visualization of this procedure. That this definition is independent of the chosen representatives 
follows from the equivariance properties of operads.

Last but not least, the tensor product is defined on representatives $(\beta,X)$ and $(\delta,Y)$ via juxtaposition,
i.e.~$(\beta,X)\otimes(\delta,Y):=(\beta\otimes\delta,X\otimes Y)$. The identity arrows are those represented by a pair 
of identities.

\begin{defi}
	The degree of an operation is its number of inputs. The degree of an object in $\calS(\calO)$ is the length of the corresponding
	color word. The degree of an arrow in $\calS(\calO)$ is the degree of its domain. A higher degree operation resp.~object
	resp.~arrow is one with degree at least $2$.
\end{defi}

\begin{defi}
	Let $\calO$ be a planar, symmetric or braided operad and let $X$ be an object in $\calS(\calO)$.
	Then the group
	\[\pi_1(\calO,X):=\pi_1\big(\calS(\calO),X\big)\]
	is called the operad group associated to $\calO$ based at $X$.
\end{defi}

\subsection{Normal forms}\label{88061}

In case $\calO$ is a planar operad, arrows in $\calS(\calO)$ are just tensor products of operations. In the symmetric
and braided case, however, arrows are equivalence classes of pairs $(\beta,X)$. In this subsection, we want to give a
normal form of such arrows, i.e.~canonical representatives $(\beta,X)$. We will treat the braided case, the symmetric case 
is similar and simpler.

Consider a colored braid $\beta$ with $n$ strands. The $i$'th strand is the strand starting from the node with
index $i\in\{1,\ldots,n\}$. Let $S$ be a subset of the index set $\{1,\ldots,n\}$. Deleting all strands in $\beta$ other than those with
an index in $S$ yields another colored braid $\beta|_S$. We say that $\beta$ is unbraided on $S$ if $\beta|_S$ is trivial.

Let $(n_1,\ldots,n_k)$ be a sequence of natural numbers with $1=n_1< n_2< \ldots< n_k=n+1$. A sequence like this is called a partition 
of $n$, denoted by $[n_1,\ldots,n_k]$, because the sets $S_i:=\{n_i,\ldots,n_{i+1}-1\}$ form a partition of the set $\{1,\ldots,n\}$. We
say $\beta$ is unbraided with respect to the partition $[n_1,\ldots,n_k]$ if it is unbraided on the sets $S_i$.

\begin{lem}
	Let $[n_1,\ldots,n_k]$ be a partition of $n$ and $\beta$ a colored braid with $n$ strands. Then there is a
	unique decomposition $\beta=\beta_p*\beta_u$ into colored braids $\beta_p$ and $\beta_u$ such that 
	$\beta_p=\beta^1_p\otimes\ldots\otimes\beta^{k-1}_p$ is a tensor product of colored braids $\beta^i_p$ with
	$|S_i|$ strands and $\beta_u$ is unbraided with respect to $[n_1,\ldots,n_k]$.
\end{lem}
\begin{proof}
	Define $\beta_p^i:=\beta|_{S_i}$ and
	\[\beta_u:=\left(\beta|_{S_1}^{-1}\otimes\ldots\otimes\beta|_{S_{k-1}}^{-1}\right)*\beta\]
	Then we have $\beta=\beta_p*\beta_u$ and $\beta_u$ is unbraided with respect to $[n_1,\ldots,n_k]$.
	The uniqueness statement is left to the reader.
\end{proof}

Now let $[\beta,X]$ be an arrow in $\calS(\calO)$ with $X=(X_1,\ldots,X_k)$. Assume $\deg(X_i)=d_i$ and 
$d_1+\ldots+d_k=n$. Define $n_i=1+\sum_{j=1}^{i-1}d_j$ for $i=1,\ldots,k+1$ and observe the partition 
$[n_1,\ldots,n_{k+1}]$. Decompose the colored braid $\beta^{-1}$ as in the previous lemma to obtain
$\beta=\tau*\rho$ where $\tau^{-1}$ is unbraided with respect to $[n_1,\ldots,n_{k+1}]$ and
$\rho=\rho_1\otimes\ldots\otimes\rho_k$ is a tensor product of colored braids $\rho_i$ with $d_i$ 
strands. Define $Y_i=\rho_i\cdot X_i$. Then from the definition of arrows in $\calS(\calO)$ it follows that
\[[\beta,X]=[\tau,Y]\]
with $Y=(Y_1,\ldots,Y_k)$. So each arrow has a representative $(\tau,Y)$ such that $\tau^{-1}$ is unbraided
in the ranges defined by the domains of the operations in the second component. It is easy to see that there
is at most one such pair.

Similarly, in the symmetric case, for each arrow in $\calS(\calO)$, there is a unique representative $(\tau,Y)$
such that the colored permutation $\tau^{-1}$ is unpermuted on the domains of the operations in the second
component.

\begin{defi}
	The unique representative $(\tau,Y)$ of an arrow in $\calS(\calO)$ with $\tau^{-1}$ unpermuted resp.~unbraided
	on the domains of the operations in $Y$ is called the normal form of that arrow.
\end{defi}

\subsection{Calculus of fractions and cancellation properties}

In the following, we write $\theta\approx\psi$ if two operations $\theta,\psi$ in an operad are equivalent
modulo the action of the symmetric resp.~braid groups, i.e.~there exists a permutation resp.~braid $\gamma$ such that
$\theta=\gamma\cdot\psi$. Of course, in the planar case, this just means equality of operations.

\begin{defi}\label{81202}
	Let $\calO$ be a (symmetric/braided) operad. We say that $\calO$ satisfies the calculus of fractions if the following two 
	conditions are satisfied:
	\begin{itemize}
		\item ({\it Square filling}) For every pair of operations $\theta_1$ and $\theta_2$ with the same output color,
		there are sequences of operations $\Psi_1=(\psi^1_1,\ldots,\psi^{k_1}_1)$ and $\Psi_2=(\psi^1_2,\ldots,\psi^{k_2}_2)$ such that 
		$\Psi_i*\theta_i$ is defined for $i=1,2$ and such that $\Psi_1*\theta_1\approx\Psi_2*\theta_2$.
		\item ({\it Equalization}) Assume we have an operation $\theta$ and sequences of operations 
		$\Psi_1=(\psi^1_1,\ldots,\psi^k_1)$ and $\Psi_2=(\psi^1_2,\ldots,\psi^k_2)$ such that $\Psi_1*\theta\approx\Psi_2*\theta$,
		i.e.~there is a $\gamma$ with $\Psi_1*\theta=\gamma\cdot(\Psi_2*\theta)$. Then $\gamma$ is already of the form
		$\gamma=\gamma_1\otimes\ldots\otimes\gamma_k$
		such that $\gamma_j\cdot\psi^j_2$ is defined for each $j=1,\ldots,k$ and there is a sequence
		of operations $\Xi_j$ for each $j=1,\ldots,k$ such that $\Xi_j*\psi^j_1=\Xi_j*(\gamma_j\cdot\psi^j_2)$.
	\end{itemize}
\end{defi}

\begin{defi}
	Let $\calO$ be a (symmetric/braided) operad. We define right cancellativity and left cancellativity for $\calO$ as follows:
	\begin{itemize}
		\item ({\it Right cancellativity)} Assume we have an operation $\theta$ and sequences of operations 
		$\Psi_1=(\psi^1_1,\ldots,\psi^k_1)$ and $\Psi_2=(\psi^1_2,\ldots,\psi^k_2)$ such that $\Psi_1*\theta\approx\Psi_2*\theta$,
		i.e.~there is a $\gamma$ with $\Psi_1*\theta=\gamma\cdot(\Psi_2*\theta)$. Then $\gamma$ is already of the form
		$\gamma=\gamma_1\otimes\ldots\otimes\gamma_k$
		such that $\gamma_j\cdot\psi^j_2$ is defined and equal to $\psi^j_1$ for each $j=1,\ldots,k$.
		\item ({\it Left cancellativity)} Assume we have operations $\theta_1$ and $\theta_2$ and a sequence of operations $\Psi$
		such that $\Psi*\theta_1=\Psi*\theta_2$. Then $\theta_1=\theta_2$.
	\end{itemize}
	We say that $\calO$ is cancellative if it is both left and right cancellative.
\end{defi}

These two definitions are designed such that the following two propositions hold. The proofs are straightforward and left to the reader
(see also \cite{thesis}).

\begin{prop}\label{16866}
	$\calO$ satisfies the calculus of fractions if and only if $\calS(\calO)$ does.
\end{prop}

\begin{prop}
	$\calO$ satisfies the left resp.~right cancellation property if and only if $\calS(\calO)$ does.
\end{prop}
%
%

\subsection{Operads with transformations}

Observe that the colors of an operad $\calO$ together with the degree $1$ operations form a category $\calI(\calO)$. 
In general, this category could be any category. Thus, to prove certain theorems, it is often necessary to impose restrictions on the
degree $1$ operations.

\begin{defi}
	A planar resp.~symmetric resp.~braided operad $\calO$ is called a \emph{planar resp.~symmetric resp.~braided 
	operad with transformations} if the category $\calI(\calO)$ is a groupoid. In other words, all the degree $1$
	operations are invertible.
\end{defi}

For such an operad, a transformation is an arrow in $\calS(\calO)$ of the form $[\sigma,X]$ where
$X=(X_1,\ldots,X_n)$ is a sequence of operations of degree $1$. The transformations form a groupoid which we
call $\calT(\calO)$.

We say that two operations $\theta_1$ and $\theta_2$ are transformation equivalent if there is a transformation $\alpha$
such that $\theta_2=\alpha*\theta_1$. We denote by $\mathcal{TC}(\calO)$ the set of equivalence classes of 
operations modulo transformation. Note that two transformation equivalent operations have the same degree. Thus, we also have
a notion of degree for elements in $\mathcal{TC}(\calO)$. We define a partial order on the set $\mathcal{TC}(\calO)$ as follows: 
Write $\Theta_1\leq\Theta_2$ if there is an operation $\theta_1$ with $[\theta_1]=\Theta_1$ and operations $\psi_1,\ldots,\psi_n$ 
such that $(\psi_1,\ldots,\psi_n)*\theta_1\in\Theta_2$. Then, for \emph{every} $\theta_1$ with $[\theta_1]=\Theta_1$ there are operations 
$\psi_1,\ldots,\psi_n$ such that $(\psi_1,\ldots,\psi_n)*\theta_1\in\Theta_2$. It is not hard to prove that this relation is indeed a 
partial order. Note that the degree function on $\mathcal{TC}(\calO)$ strictly respects this order relation which means
\[\Theta_1<\Theta_2\hspace{5mm}\Longrightarrow\hspace{5mm}\op{deg}(\Theta_1)<\op{deg}(\Theta_2)\]

The following observation, which easily follows from the definitions, reinterpretes the square filling property of Definition \ref{81202} 
in terms of the poset $\mathcal{TC}(\calO)$ of transformation classes:
\begin{obs}\label{77706}
	Let $\calO$ be a (symmetric/braided) operad with transformations. Then $\calO$ satisfies the square filling property if and only if for each pair 
	$\Theta_1,\Theta_2$ of transformation classes with the same codomain color there is another transformation class $\Theta$ with 
	$\Theta_1\leq\Theta\geq\Theta_2$.
\end{obs}


\subsubsection{Spines in graded posets}

We call a poset $P$ graded if there is degree function $\deg\co P\rightarrow\NN$ such that $\deg(x)<\deg(y)$ whenever $x<y$. For example,
$\mathcal{TC}(\calO)$ above is graded.

\begin{defi}\label{79040}
	Let $P$ be a graded poset and $M\subset P$ be the subset of minimal elements in $P$. The spine $S$ of $P$ is the smallest subset 
	$S\subset P$ such that $M\subset S$ and which satisfies the following property: Whenever $v\in P\setminus S$, then there is a 
	greatest element $g\in S$ such that $g<v$.
\end{defi}

We want to prove that the spine of a graded poset always exists.

\begin{cons}\label{45209}
	We define $S_i\subset P$ for $i\in\{0,1,2,\ldots\}$ inductively. Set $S_0=M$. Assume that $S_i$ has been constructed.
	For each pair $x,y\in S_i$ with $x\neq y$, define $M_{i+1}(x,y)\subset P$ to be the set consisting of all the minimal elements $z$
	with the property $x\leq z\geq y$. Now, let $S_{i+1}$ be the union of all the $M_{i+1}(x,y)$. Finally, define $S=\bigcup_{i=0}^\infty S_i$.
\end{cons}

In the following, we want to show that this $S$ satisfies the defining properties of the spine of $P$.

\begin{obs}\label{61724}
	Let $A\subset P$. Assume that $v\in P$ satisfies $a<v$ for all $a\in A$. We claim that there is a minimal element $p$ in the set 
	$\{z\in P\mid\forall_{a\in A}\ a\leq z\}$ which also satisfies $p\leq v$. If $v$ is already minimal, then we can set $p=v$. If it is not 
	minimal, there must be another element $v'\in P$ with $a\leq v'<v$ for all $a\in A$. Then $v'$ has strictly smaller degree than $v$. If we 
	repeat this argument with $v'$, we have to end up with a minimal element $p$ at some time, because the degree function is bounded below.
	This $p$ surely satisfies $p\leq v$.
\end{obs}

Let $v\in P\setminus S$. We want to find the greatest element in the set 
\[V:=\{z\in S\mid z<v\}\]
For each $i$, set 
$S^\downa_i=S_i\cap V$. We claim: There exists exactly one $i_0$ such that $|S^\downa_j|>1$ for $j<i_0$,
$|S^\downa_{i_0}|=1$ and $S^\downa_j=\emptyset$ for $j>i_0$ and the unique element in $S^\downa_{i_0}$ is the greatest element 
in $V$. 

Observation \ref{61724} applied to $A=\emptyset$ reveals that $S^\downa_0\neq\emptyset$. Note that either all but finitely many of the
$S_i$ are empty or the sequence of numbers 
\[d_i:=\op{min}\{\op{deg}(z)\mid z\in S_i\}\]
tends to infinity. But the degree of all the 
elements in all the $S^\downa_i$ is bounded by $\op{deg}(v)$. It follows that in any case there must be an $i_0$ such that 
$S^\downa_j=\emptyset$ for all $j>i_0$. Choose the $i_0$ which is minimal with respect to this property, 
i.e.~$S^\downa_{i_0}\neq\emptyset$. Assume $|S^\downa_{i_0}|>1$ and let $x\neq y$ be two elements in this set. Write $A=\{x,y\}$ and recall that 
$x,y<v$. Thus, by Observation \ref{61724}, we know that there must be a $p\in M_{i_0+1}(x,y)$ with $p\leq v$. Since $v\not\in S$, we have indeed 
$p<v$. Consequently, $p\in S^\downa_{i_0+1}$, a contradiction. So we have indeed $|S^\downa_{i_0}|=1$. Next, observe that for any $j$, if $S_j\neq\emptyset$,
then $S_{j-1}$ consists of at least two elements. This follows directly from the definitions. Consequently, the same holds for the $S^\downa_j$.
From this, it easily follows $|S^\downa_j|>1$ for $j<i_0$.

We now use this to prove that the unique element $g\in S^\downa_{i_0}$ is the greatest element in $V$, i.e.~$x\leq g$ whenever 
$x\in S$ with $x<v$. Let $x$ be such an element. If $x\neq g$, then there must be some $j<i_0$ such that $x\in S^\downa_j$. There is another element 
$x'$ in this $S^\downa_j$. Observation \ref{61724} applied to $A=\{x,x'\}$ shows that there is $p\in S^\downa_{j+1}$ with $x\leq p$. If $j+1=i_0$, that 
$p$ must be $g$ and we are done. Else, we repeat this process with $p$ in place of $x$ until we reach level $i_0$. This completes the proof
that $S$ satisfies the last property in Definition \ref{79040}.

Remains to prove that $S$ is the smallest subset containing $M$ and satisfying this property. So let $S'\subset P$ be another subset containing $M$ and
satisfying this property. We have to show $S\subset S'$. We will prove $S_i\subset S'$ by induction over $i$. The induction start is trivial because
$S_0=M$. For the induction step, assume $S_i\subset S'$. Let $v\in S_{i+1}$. Assume that $v\not\in S'$. Then there is a greatest element $p\in S'$ with $p<v$. 
Furthermore, there must be $x,y\in S_i$ with $x\neq y$ and $v\in M_{i+1}(x,y)$. This means that $v$ is minimal with respect to $x\leq v\geq y$.
Since $x,y\in S'$ but $v\not\in S'$ we have indeed $x<v>y$. Since $p$ is the greatest element in $S'$ with $p<v$, we obtain $x\leq p\geq y$.
This contradicts the minimality of $v$. So we must have $v\in S'$ and thus $S_{i+1}\subset S'$.

\subsubsection{Elementary and very elementary operations}

Denote by $\mathcal{TC}^\ast(\calO)$ the full subposet of $\mathcal{TC}(\calO)$ spanned by the higher degree classes 
(i.e.~the elements of degree at least $2$).

\begin{defi}\label{13449}
	Let $\calO$ be a (symmetric/braided) operad with transformations. The minimal elements in $\mathcal{TC}^\ast(\calO)$
	are called very elementary transformation classes. Denote the set of very elementary classes by $VE$.
\end{defi}

Let $\Theta,\Theta_1,\ldots,\Theta_k\in\mathcal{TC}(\calO)$ be (not necessarily distinct) transformation
classes. We say that $\Theta$ is decomposable into the classes $\Theta_i$ if we find operations 
$\theta_i\in\Theta_i$ for $i=1,\ldots,k$ which can be partially composed (see Remark \ref{73668})
in a certain way to an operation in $\Theta$. It can be shown that any class in $\mathcal{TC}^\ast(\calO)$
decomposes into very elementary classes.

\begin{defi}\label{24121}
	Let $\calO$ be a (symmetric/braided) operad with transformations. The elements in the spine of $\mathcal{TC}^\ast(\calO)$ 
	are called elementary transformation classes. Denote the set of elementary classes by $E$.
\end{defi}

An operation in $\calO$ is called (very) elementary if it is contained in a (very) elementary transformation class. 
We will call the elementary but not very elementary classes resp.~operations strictly elementary.

\begin{defi}\label{26536}
	$\calO$ is finitely generated if there are only finitely many very elementary transformation classes. It is of finite type if 
	there are only finitely many elementary transformation classes.
\end{defi}

The following proposition states that the subsets $VE$ and $E$ are invariant under the right action of degree $1$ operations.
\begin{prop}\label{07546}
	Let $\calO$ be a (symmetric/braided) operad with transformations. Let $\theta$ be a higher degree operation and $\gamma$ be a 
	degree $1$ operation. Then the transformation class $[\theta]$ is (very) elementary if and only if the class $[\theta]*\gamma:=[\theta*\gamma]$ 
	is (very) elementary. In particular, the operation $\theta$ is (very) elementary if and only if $\theta*\gamma$ is (very) elementary.
\end{prop}
\begin{proof}
	The main observation is that if $\Theta,\Theta'$ are two transformation classes, then $\Theta<\Theta'$ holds if and only if
	$\Theta*\gamma<\Theta'*\gamma$ holds. This implies that $\Theta\in VE$ if and only if $\Theta*\gamma\in VE$ or, in other words,
	$VE*\gamma=VE$. Now write $E'=E*\gamma$. We then have $VE\subset E'$. Let $\Theta\in\mathcal{TC}^\ast(\calO)\setminus E'$.
	Then $\Theta*\gamma^{-1}\in\mathcal{TC}^\ast(\calO)\setminus E$. Thus, by the definition of $E$ as the spine of $\mathcal{TC}^\ast(\calO)$, 
	we have that there is a greatest element $\Psi\in E$ with $\Psi<\Theta*\gamma^{-1}$. Then $\Psi*\gamma\in E'$ is the greatest element
	with $\Psi*\gamma<\Theta$. Consequently, $E'$ satisfies the defining properties of the spine $E$. It follows $E\subset E'=E*\gamma$.
	Since this holds for arbitrary $\gamma$, we obtain $E=E*\gamma$.
\end{proof}

\subsection{Examples}\label{13887}

In this subsection, we want to present some examples of operads leading to already well-known operad groups as well as to
new groups to which the finiteness result of Section \ref{17057} is applicable.

\subsubsection{Free operads}\label{28776}

Only very briefly we want to remark that the operad groups associated to operads \emph{freely} generated by
operations of degree at least $2$ correspond exactly to the so-called diagram groups defined in \cite{gu-sa:dg}.
When considering free \emph{symmetric} operads, we get symmetric versions of diagram groups which are called 
``braided'' in \cite{gu-sa:dg}*{Definition 16.2}. The truly braided diagram groups are the ones arising from free \emph{braided} operads.

In particular, the operad group associated to the operad $\calO_F$ freely generated by one color and a single binary operation
is isomorphic to Thompson's group $F$. This has first been observed in \cite{fi-le:aac}. Moreover, if we consider
the symmetric resp.~braided operad $\calO_V$ resp.~$\calO_{BV}$ freely generated by one color and a single binary operation, we obtain
Thompson's group $V$ resp.~the braided Thompson group $BV$.

More details on the free case can be found in \cite{thesis}.

\subsubsection{Suboperads of endomorphism operads}\label{42483}

Recall that there is a planar resp.~symmetric resp.~braided operad $\op{End}(\calC)$ naturally associated to each planar resp.~symmetric resp.~braided
monoidal category $\calC$,
called the endomorphism operad. The colors of $\op{End}(\calC)$ are the objects in $\calC$ and the sets of operations are given by
\[\op{End}(\calC)(a_1,\ldots,a_n;b)=\op{Hom}_\mathcal{C}(a_1\otimes\ldots\otimes a_n,b)\]
Let $\calG$ be a subgroupoid of $\calC$ and $S$ be a set of higher degree operations in $\calE:=\op{End}(\calC)$ with outputs
and inputs being objects of $\calG$. Then we can look at the suboperad of $\calE$ generated by this data: It is the smallest suboperad $\calO$ such 
that $\calI(\calO)=\calG$ and such that the elements in $S$ are operations in $\calO$. These suboperads are in general not free in the sense of 
\ref{28776} though we have only specified generators. The relations are automatically modelled by the ambient category $\calC$.

Not always is the map $S\rightarrow\mathcal{TC}(\calO)$
sending an operation to its transformation class a bijection onto the set of very elementary classes. However, this is true if the following conditions are 
satisfied:
\begin{itemize}
	\item[$(\calV_1)$] If $\theta,\theta'\in S$ with $\theta\neq\theta'$, then $[\theta]$ and $[\theta']$ are incomparable, i.e.~$[\theta]\not\leq[\theta']$ 
		and $[\theta]\not\geq[\theta']$. In particular, they are not equal.
	\item[$(\calV_2)$] The set of transformation classes represented by operations in $S$ is closed under right multiplication with operations in $\calG$,
		i.e.~for each $\theta\in S$ and $\gamma\in\calG$ there is $\theta'\in S$ with $[\theta*\gamma]=[\theta']$.
\end{itemize}

We want to be a bit more explicit now and observe suboperads of the endomorphism operad $\calE$ of the symmetric monoidal category
$(\mathtt{TOP},\sqcup)$ where $\sqcup$ is the coproduct (i.e.~the disjoint union) of topological spaces. We call an operation
$(f_1,\ldots,f_k)$ in $\calE$ \emph{mono} if the images of the maps $f_i\co X_i\rightarrow X$ are pairwise disjoint in $X$ and the $f_i$
are injective. We call it \emph{epi} if the images cover $X$. It is not hard to prove that if all operations in a suboperad of $\calE$ 
are mono, then it satisfies the right cancellation property. Likewise, if all the operations are epi, then it satisfies the left cancellation
property.

We give an explicit example to illustrate the above procedure. Consider the unit square and the right angled triangle obtained by
halving the unit square:
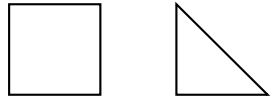
\ifthenelse{\boolean{drawtikz}}{
\begin{center}
	\begin{tikzpicture}[scale=0.4]
		\draw[thick] (-2.5,1.5) -- (0.5,1.5) -- (0.5,-1.5) -- (-2.5,-1.5) -- cycle;
		\draw[thick] (3,1.5) -- (3,-1.5) -- (6,-1.5) -- cycle;
	\end{tikzpicture}
\end{center}
}{TIKZ}
Consider all isometries of the square and the triangle, i.e.~the dihedral group $D_4$ and $\ZZ/2\ZZ$. The disjoint union of these isometry
groups forms a groupoid $\calG$ lying in $\mathtt{TOP}$. Consider the following subdivisions, called very elementary subdivisions:
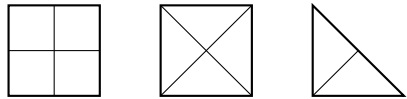
\ifthenelse{\boolean{drawtikz}}{
\begin{center}
	\begin{tikzpicture}[scale=0.4]
		\draw[thick] (2.5,1.5) -- (5.5,1.5) -- (5.5,-1.5) -- (2.5,-1.5) -- cycle;
		\draw (2.5,-1.5) -- (5.5,1.5);
		\draw[thick] (-2.5,1.5) -- (0.5,1.5) -- (0.5,-1.5) -- (-2.5,-1.5) -- cycle;
		\draw (2.5,1.5) -- (5.5,-1.5);
		\draw[thick] (7.5,1.5) -- (7.5,-1.5) -- (10.5,-1.5) -- cycle;
		\draw (-1,1.5) -- (-1,-1.5);
		\draw (-2.5,0) -- (0.5,0);
		\draw (9,0) -- (7.5,-1.5);
	\end{tikzpicture}
\end{center}
}{TIKZ}
The set $S$ has three elements, one for each sudivision: The first one maps four squares to each square in the first subdivision via coordinate-wise
linear transformations. The second one maps four triangles to each triangle in the second subdivision via orientation preserving similarities.
The third one maps two triangles to each triangle in the third subdivision via orientation preserving similarities. As above, the groupoid
$\calG$ together with the set $S$ generate a suboperad $\calO$ of the symmetric operad $\calE=\op{End}(\mathtt{TOP},\sqcup)$.

The transformation classes are in one to one correspondence with subdivisions of the square or the triangle which can be obtained
by iteratively applying the three subdivisions above. We have $\Theta_1\leq\Theta_2$ if and only if $\Theta_2$ can be obtained from $\Theta_1$
by performing further subdivisions. For example, we have
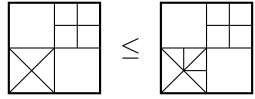
\ifthenelse{\boolean{drawtikz}}{
\begin{center}
	\begin{tikzpicture}[scale=0.2]
		\draw[thick] (3,4.5) -- (9,4.5) -- (9,-1.5) -- (3,-1.5) -- cycle;		
		\draw (6,4.5) -- (6,-1.5);
		\draw (3,1.5) -- (9,1.5);		
		\draw (7.5,4.5) -- (7.5,1.5);
		\draw (6,3) -- (9,3);
		\draw (3,1.5) -- (6,-1.5);
		\draw (6,1.5) -- (3,-1.5);
		\draw[thick] (13,4.5) -- (19,4.5) -- (19,-1.5) -- (13,-1.5) -- cycle;
		\node at (11,1.5) {$\leq$};
		\draw (16,4.5) -- (16,-1.5);
		\draw (13,1.5) -- (19,1.5);
		\draw (17.5,4.5) -- (17.5,1.5);
		\draw (16,3) -- (19,3);
		\draw (13,1.5) -- (16,-1.5);
		\draw (16,1.5) -- (13,-1.5);
		\draw (14.5,0) -- (14.5,1.5);
		\draw (16,0) -- (14.5,0);
	\end{tikzpicture}
\end{center}
}{TIKZ}
From this it follows easily that the transformation classes represented by the very elementary subdivisions
are not comparable, i.e.~$(\calV_1)$ is satisfied. Furthermore, when applying an isometry of the square or the triangle to one of the 
operations in $S$, we obtain the same operation with a transformation precomposed. Thus, also $(\calV_2)$ is satisfied. It follows that 
the very elementary subdivisions correspond exactly to the very elementary classes of $\calO$.

To find all the elementary transformation classes, we have to follow the construction in \ref{45209}.
There is exactly one minimal subdivision of the square which refines the two very elementary subdivisions of the square:
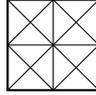
\ifthenelse{\boolean{drawtikz}}{
\begin{center}
	\begin{tikzpicture}[scale=0.4]
		\draw[thick] (2.5,1.5) -- (5.5,1.5) -- (5.5,-1.5) -- (2.5,-1.5) -- cycle;	
		\draw (4,1.5) -- (4,-1.5);
		\draw (2.5,0) -- (5.5,0);
		\draw (2.5,1.5) -- (5.5,-1.5);
		\draw (5.5,1.5) -- (2.5,-1.5);
		\draw (4,1.5) -- (2.5,0);
		\draw (4,1.5) -- (5.5,0);
		\draw (2.5,0) -- (4,-1.5);
		\draw (4,-1.5) -- (5.5,0);
	\end{tikzpicture}
\end{center}
}{TIKZ}
Thus, this subdivision represents the only elementary class which is not very elementary.

All the operations in $\calO$ are clearly epi, so it satisfies the left cancellation property. Not all of them are mono, but we can change
the definitions a little bit and obtain an isomorphic operad where all operations are mono: Instead of the closed square and triangle, we
can consider the open square and triangle and also subdivisions into open squares and triangles. Thus, $\calO$ also satisfies the
right cancellation property. Moreover, we claim that it satisfies square filling. To see this, consider the following chains of subdivisions:
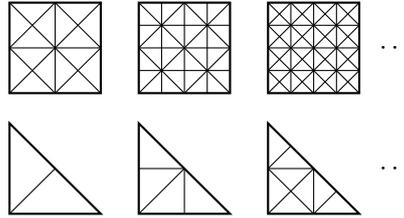
\ifthenelse{\boolean{drawtikz}}{
\begin{center}
	\begin{tikzpicture}[scale=0.2]
		\draw[thick] (11.5,4.5) -- (17.5,4.5) -- (17.5,-1.5) -- (11.5,-1.5) -- cycle;		
		\draw (14.5,4.5) -- (14.5,-1.5);
		\draw (11.5,1.5) -- (17.5,1.5);		
		\draw (14.5,4.5) -- (17.5,1.5);
		\draw (17.5,4.5) -- (11.5,-1.5);
		\draw (11.5,4.5) -- (17.5,-1.5);
		\draw (14.5,4.5) -- (11.5,1.5);
		\draw (11.5,1.5) -- (14.5,-1.5);
		\draw (17.5,1.5) -- (14.5,-1.5);
		\draw[thick] (20,4.5) -- (26,4.5) -- (26,-1.5) -- (20,-1.5) -- cycle;		
		\draw (23,4.5) -- (23,-1.5);
		\draw (20,1.5) -- (26,1.5);		
		\draw (23,4.5) -- (26,1.5);
		\draw (26,4.5) -- (20,-1.5);
		\draw (20,4.5) -- (26,-1.5);
		\draw (23,4.5) -- (20,1.5);
		\draw (20,1.5) -- (23,-1.5);
		\draw (26,1.5) -- (23,-1.5);
		\draw[thick] (3,4.5) -- (9,4.5) -- (9,-1.5) -- (3,-1.5) -- cycle;		
		\draw (6,4.5) -- (6,-1.5);
		\draw (3,1.5) -- (9,1.5);		
		\draw (6,4.5) -- (9,1.5);
		\draw (9,4.5) -- (3,-1.5);
		\draw (3,4.5) -- (9,-1.5);
		\draw (6,4.5) -- (3,1.5);
		\draw (3,1.5) -- (6,-1.5);
		\draw (9,1.5) -- (6,-1.5);
		\draw (11.5,3) -- (17.5,3);
		\draw (13,4.5) -- (13,-1.5);
		\draw (16,4.5) -- (16,-1.5);
		\draw (11.5,0) -- (17.5,0);
		\draw (21.5,4.5) -- (21.5,-1.5);
		\draw (20,3) -- (26,3);
		\draw (20,0) -- (26,0);
		\draw (24.5,4.5) -- (24.5,-1.5);
		\draw (21.5,4.5) -- (26,0);
		\draw (24.5,4.5) -- (26,3);
		\draw (20,3) -- (24.5,-1.5);
		\draw (20,0) -- (21.5,-1.5);
		\draw (21.5,4.5) -- (20,3);
		\draw (24.5,4.5) -- (20,0);
		\draw (26,3) -- (21.5,-1.5);
		\draw (26,0) -- (24.5,-1.5);
		\node at (28.5,1.5) {$\cdots$};	
		\draw[thick] (3,-3.5) -- (3,-9.5) -- (9,-9.5) -- cycle;
		\draw[thick] (11.5,-3.5) -- (11.5,-9.5) -- (17.5,-9.5) -- cycle;
		\draw[thick] (20,-3.5) -- (20,-9.5) -- (26,-9.5) -- cycle;
		\draw (3,-9.5) -- (6,-6.5);
		\draw (11.5,-9.5) -- (14.5,-6.5);
		\draw (11.5,-6.5) -- (14.5,-6.5);
		\draw (14.5,-9.5) -- (14.5,-6.5);
		\draw (20,-9.5) -- (23,-6.5);
		\draw (20,-6.5) -- (23,-6.5);
		\draw (23,-6.5) -- (23,-9.5);
		\draw (20,-6.5) -- (23,-9.5);
		\draw (20,-6.5) -- (21.5,-5);
		\draw (23,-9.5) -- (24.5,-8);
		\node at (28.5,-6.5) {$\cdots$};
	\end{tikzpicture}
\end{center}
}{TIKZ}
These are cofinal in the sense that every subdivision of the square resp.~triangle is smaller than or equal to one of the subdivisions of the first
resp.~second chain. From Observation \ref{77706} it follows that $\calO$ satisfies square filling. All in all, $\calO$ satisfies the cancellative 
calculus of fractions.

\vspace{2mm}
$\vartriangleright$ {\it Cube cutting operads.} Let $N$ be a finite set of natural numbers greater than or equal to $2$. 
Denote by $\langle N\rangle$ the multiplicative submonoid of $\NN$ generated
by the numbers in $N$. We say that the numbers in $N$ are independent if, whenever a natural number $n$ can be written as a product 
$n_1^{r_1}\cdots n_k^{r_k}$ of pairwise distinct numbers $n_i\in N$, then the exponents $r_i$ are already uniquely determined by $n$. In other words,
$N$ is a basis for $\langle N\rangle$. This is satisfied for example if the numbers in $N$ are pairwise coprime or, even stronger, if they
are prime. For later reference, we record the following two trivial observations:
\begin{itemize}
	\item[$(\calB_1)$] No number $n\in N$ is a product of other numbers in $N$.
	\item[$(\calB_2)$] Whenever $n_1,\ldots,n_k\in N$ are pairwise distinct numbers and $n\in\langle N\rangle$ is divisible by each $n_i$ in
		$\langle N\rangle$, i.e.~there is $m_i\in\langle N\rangle$ with $n=n_im_i$, then $n$ is also divisible by the product $n_1\cdots n_k$
		in $\langle N\rangle$.
\end{itemize}
There are non-bases $N$ which satisfy $(\calB_1)$ but not $(\calB_2)$, for example $N=\{2,6,7,21\}$. For this $N$ we have $6\cdot 7=42=2\cdot 21$.

In the same vein as above, we now construct cube cutting operads. For $d\geq 1$, consider the $d$-dimensional unit cube and a subgroup if its group of 
isometries. Define this group to be the groupoid $\calG$ lying in $\mathtt{TOP}$. Next, we want to specify very elementary subdivisions of the cube.
For each $j\in\{1,\ldots,d\}$, let $N_j\subset\NN$ be a set of natural numbers as in the preceding paragraph. For each such $j$ and $n\in N_j$, there
is a very elementary subdivision of the cube given by cutting it, perpendicularly to the $j$'th coordinate axis, into $n$ congruent subbricks.
The following are the very elementary subdivisions in the case $d=2$, $N_1=\{2\}$ and $N_2=\{3\}$:
\ifthenelse{\boolean{drawtikz}}
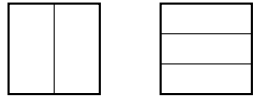
{
\begin{center}
	\begin{tikzpicture}[scale=0.4]
		\draw[thick] (8,3.5) -- (11,3.5) -- (11,0.5) -- (8,0.5) -- cycle;
		\draw[thick] (13,3.5) -- (16,3.5) -- (16,0.5) -- (13,0.5) -- cycle;
		\draw (9.5,3.5) -- (9.5,0.5); \draw (13,2.5) -- (16,2.5);
		\draw (13,1.5) -- (16,1.5); 
	\end{tikzpicture}
\end{center}
}{TIKZ}
There is one operation in $S$ for each such very elementary subdivision: Cubes are coordinate-wise linearly rescaled to fit into the subbricks
of the subdivisions. The groupoid $\calG$ together with the set $S$ generate a suboperad $\calO$ of $\calE=\op{End}(\mathtt{TOP},\sqcup)$ which
we call a symmetric cube cutting operad since we will also define planar cube cutting operads below.

The transformation classes are in one to one correspondence with subdivisions of the cube obtained by iteratively applying $n$-cuts
in direction $j$ as above. Two transformation classes are comparable if and only if one is a subdivision of the other. 
From $(\calB_1)$ it follows that two very elementary subdivisions are not comparable. Consequently, $(\calV_1)$ is satisfied.
It is not always true that right multiplication of elements in $\calG$ with operations in $S$ yields another operation in $S$ up to transformation.
For example, a rotation of a vertically cutted square by an angle of $\pi/2$ yields a horizontally cutted square. Whether $(\calV_2)$ is
satisfied or not depends on the interplay between the isometries in $\calG$ and the sets $N_j$. For example, it is satisfied if $\calG=1$ or
if $N_1=\ldots=N_d$. Let us always assume that $\calG$ and the $N_j$ are compatible in a way such that $(\calV_2)$ is satisfied. Then the very elementary 
subdivisions are in one to one correspondence with the very elementary transformation classes.

We want to identify the elementary transformation classes. For each element $T=(T_1,\ldots,T_d)\in 2^{N_1}\times\ldots\times 2^{N_d}$ of the product of the 
power sets such that $T\neq(\emptyset,\ldots,\emptyset)$, there is a transformation class $\Theta_T$ which is obtained by iteratively performing, for each 
$j\in\{1,\ldots,d\}$ and each $n\in T_j$, an $n$-cut in direction $j$ on every subbrick. The result is independent of the order of the cuts. These classes 
are exactly the elementary classes. To see this, we make the following claim: If $\Theta_{T}$ and $\Theta_{T'}$ are two such classes, then $\Theta_{T\cup T'}$ 
is the smallest class $\Theta$ satisfying $\Theta_T\leq\Theta\geq\Theta_{T'}$. Here, the inclusion $T\subset T'$ and the union $T\cup T'$ is meant to be coordinate-wise.
The figure below pictures the elementary operations in the case $d=2$, $N_1=\{2,3\}$ and $N_2=\{2,3\}$.
\ifthenelse{\boolean{drawtikz}}
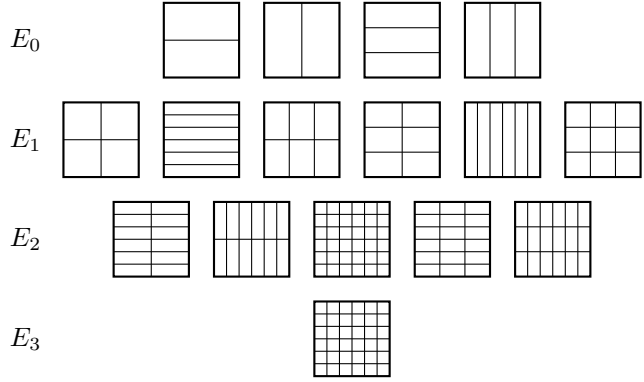
{
\begin{center}
	\begin{tikzpicture}[scale=0.33]
		\draw[thick] (-0.5,-0.5) -- (2.5,-0.5) -- (2.5,-3.5) -- (-0.5,-3.5) -- cycle;
		\draw[thick] (3.5,-0.5) -- (6.5,-0.5) -- (6.5,-3.5) -- (3.5,-3.5) -- cycle;
		\draw[thick] (7.5,-0.5) -- (10.5,-0.5) -- (10.5,-3.5) -- (7.5,-3.5) -- cycle;
		\draw[thick] (11.5,-0.5) -- (14.5,-0.5) -- (14.5,-3.5) -- (11.5,-3.5) -- cycle;
		\draw[thick] (19.5,-0.5) -- (22.5,-0.5) -- (22.5,-3.5) -- (19.5,-3.5) -- cycle;
		\draw[thick] (15.5,-0.5) -- (18.5,-0.5) -- (18.5,-3.5) -- (15.5,-3.5) -- cycle;
		\draw[thick] (3.5,3.5) -- (3.5,0.5) -- (6.5,0.5) -- (6.5,3.5) -- cycle;
		\draw[thick] (7.5,3.5) -- (10.5,3.5) -- (10.5,0.5) -- (7.5,0.5) -- cycle;
		\draw[thick] (11.5,3.5) -- (14.5,3.5) -- (14.5,0.5) -- (11.5,0.5) -- cycle;
		\draw[thick] (15.5,3.5) -- (18.5,3.5) -- (18.5,0.5) -- (15.5,0.5) -- cycle;
		\draw[thick] (1.5,-4.5) -- (4.5,-4.5) -- (4.5,-7.5) -- (1.5,-7.5) -- cycle;
		\draw[thick] (5.5,-4.5) -- (8.5,-4.5) -- (8.5,-7.5) -- (5.5,-7.5) -- cycle;
		\draw[thick] (9.5,-4.5) -- (12.5,-4.5) -- (12.5,-7.5) -- (9.5,-7.5) -- cycle;
		\draw[thick] (13.5,-4.5) -- (16.5,-4.5) -- (16.5,-7.5) -- (13.5,-7.5) -- cycle;
		\draw[thick] (17.5,-4.5) -- (20.5,-4.5) -- (20.5,-7.5) -- (17.5,-7.5) -- cycle;
		\draw[thick] (9.5,-8.5) -- (12.5,-8.5) -- (12.5,-11.5) -- (9.5,-11.5) -- cycle;
		\draw (3.5,2) -- (6.5,2); \draw (9,3.5) -- (9,0.5); \draw (11.5,2.5) -- (14.5,2.5);
		\draw (16.5,3.5) -- (16.5,0.5); \draw (11.5,1.5) -- (14.5,1.5); \draw (17.5,3.5) -- (17.5,0.5);
		\draw (-0.5,-2) -- (2.5,-2); \draw (1,-0.5) -- (1,-3.5); \draw (3.5,-1) -- (6.5,-1);
		\draw (3.5,-1.5) -- (6.5,-1.5); \draw (3.5,-2) -- (6.5,-2); \draw (3.5,-2.5) -- (6.5,-2.5);
		\draw (3.5,-3) -- (6.5,-3); \draw (7.5,-2) -- (10.5,-2); \draw (8.5,-0.5) -- (8.5,-3.5);
		\draw (9.5,-0.5) -- (9.5,-3.5); \draw (11.5,-1.5) -- (14.5,-1.5); \draw (11.5,-2.5) -- (14.5,-2.5);
		\draw (13,-0.5) -- (13,-3.5); \draw (20.5,-0.5) -- (20.5,-3.5); \draw (21.5,-0.5) -- (21.5,-3.5);
		\draw (19.5,-1.5) -- (22.5,-1.5); \draw (19.5,-2.5) -- (22.5,-2.5); \draw (16,-0.5) -- (16,-3.5);
		\draw (16.5,-0.5) -- (16.5,-3.5); \draw (17,-0.5) -- (17,-3.5); \draw (17.5,-0.5) -- (17.5,-3.5);
		\draw (18,-0.5) -- (18,-3.5); \draw (1.5,-5) -- (4.5,-5); \draw (1.5,-5.5) -- (4.5,-5.5);
		\draw (1.5,-6) -- (4.5,-6); \draw (1.5,-6.5) -- (4.5,-6.5); \draw (1.5,-7) -- (4.5,-7);
		\draw (3,-4.5) -- (3,-7.5); \draw (6,-4.5) -- (6,-7.5); \draw (6.5,-4.5) -- (6.5,-7.5);
		\draw (7,-4.5) -- (7,-7.5); \draw (7.5,-4.5) -- (7.5,-7.5); \draw (8,-4.5) -- (8,-7.5);
		\draw (5.5,-6) -- (8.5,-6); \draw (10,-4.5) -- (10,-7.5); \draw (10.5,-4.5) -- (10.5,-7.5);
		\draw (11,-4.5) -- (11,-7.5); \draw (11.5,-4.5) -- (11.5,-7.5); \draw (12,-4.5) -- (12,-7.5);
		\draw (9.5,-5) -- (12.5,-5); \draw (9.5,-5.5) -- (12.5,-5.5); \draw (9.5,-6) -- (12.5,-6);
		\draw (9.5,-6.5) -- (12.5,-6.5); \draw (9.5,-7) -- (12.5,-7); \draw (13.5,-5) -- (16.5,-5);
		\draw (13.5,-5.5) -- (16.5,-5.5); \draw (13.5,-6) -- (16.5,-6); \draw (13.5,-6.5) -- (16.5,-6.5);
		\draw (13.5,-7) -- (16.5,-7); \draw (14.5,-4.5) -- (14.5,-7.5); \draw (15.5,-4.5) -- (15.5,-7.5);
		\draw (17.5,-5.5) -- (20.5,-5.5); \draw (17.5,-6.5) -- (20.5,-6.5); \draw (18,-4.5) -- (18,-7.5);
		\draw (18.5,-4.5) -- (18.5,-7.5); \draw (19,-4.5) -- (19,-7.5); \draw (19.5,-4.5) -- (19.5,-7.5);
		\draw (20,-4.5) -- (20,-7.5); \draw (10,-8.5) -- (10,-11.5); \draw (10.5,-8.5) -- (10.5,-11.5);
		\draw (11,-8.5) -- (11,-11.5); \draw (11.5,-8.5) -- (11.5,-11.5); \draw (12,-8.5) -- (12,-11.5);
		\draw (9.5,-9) -- (12.5,-9); \draw (9.5,-9.5) -- (12.5,-9.5); \draw (9.5,-10) -- (12.5,-10);
		\draw (9.5,-10.5) -- (12.5,-10.5); \draw (9.5,-11) -- (12.5,-11);
		\node at (-2,2) {$E_0$};
		\node at (-2,-2) {$E_1$};
		\node at (-2,-6) {$E_2$};
		\node at (-2,-10) {$E_3$};
	\end{tikzpicture}
\end{center}
}{TIKZ}
To see the above claim, we consider the case $d=1$ and set $N:=N_1$. The case $d>1$ can be derived by applying the following observations 
coordinate-wise. Call a transformation class regular if all the subintervals in the corresponding subdivision of the unit interval have the 
same length. Now, let $\Theta$ be a transformation class with $\Theta_T\leq\Theta\geq\Theta_{T'}$. It is not hard to find the greatest regular
class $\Theta_r$ with $\Theta_r\leq\Theta$. Since $\Theta_T$ and $\Theta_{T'}$ are regular, we have $\Theta_T\leq\Theta_r\geq\Theta_{T'}$.
There is a unique $n\in\langle N\rangle$ such that $\frac{1}{n}$ is the length of the subintervals in the subdivision of $\Theta_r$.
Then $\Theta_T\leq\Theta_r$ means that the product of the numbers in $T$ divides $n$ in $\langle N\rangle$. In particular, each $t\in T$
divides $n$ in $\langle N\rangle$. Likewise, each $t'\in T'$ divides $n$ in $\langle N\rangle$. It follows from $(\calB_2)$ that
the product of the numbers in $T\cup T'$ divides $n$ in $\langle N\rangle$. This implies $\Theta_{T\cup T'}\leq\Theta_r$ and it follows
$\Theta_{T\cup T'}\leq\Theta$, q.e.d.

All the operations in $\calO$ are epi and $\calO$ is isomorphic to a suboperad of $\calE$ where all operations are mono by considering open cubes instead of closed ones.
Consequently, $\calO$ satisfies the left and right cancellation property. We also find a cofinal chain of subdivisions: The first subdivision
in this chain is obtained by iteratively applying, for each $j\in\{1,\ldots,d\}$ and each $n\in N_j$, an $n$-cut in direction $j$ on every subbrick. Then the whole
chain is obtained by iterating this with every subbrick. For example, in the case $d=2$, $N_1=\{2\}$ and $N_2=\{3\}$, we can take the following chain:
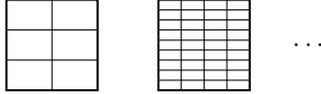
\ifthenelse{\boolean{drawtikz}}{
\begin{center}
	\begin{tikzpicture}[scale=0.2]
		\draw[thick] (13,6) -- (19,6) -- (19,0) -- (13,0) -- cycle;			
		\draw (16,6) -- (16,0);
		\draw (3,2) -- (9,2);
		\draw[thick] (3,6) -- (9,6) -- (9,0) -- (3,0) -- cycle;		
		\draw (6,6) -- (6,0);
		\draw (3,4) -- (9,4);	
		\node at (23,3) {$\cdots$};	
		\draw (14.5,6) -- (14.5,0);
		\draw (17.5,6) -- (17.5,0);
		\draw (13,4) -- (19,4);
		\draw (13,2) -- (19,2);
		\draw (13,5.332) -- (19,5.332);
		\draw (13,4.666) -- (19,4.666);
		\draw (13,3.322) -- (19,3.322);
		\draw (13,2.666) -- (19,2.666);
		\draw (13,1.332) -- (19,1.332);
		\draw (13,0.666) -- (19,0.666);
	\end{tikzpicture}
\end{center}
}{TIKZ}
Thus, $\calO$ satisfies the square filling property. All in all, it satisfies the cancellative calculus of fractions.

Note that the symmetric cube cutting operads are symmetric operads with transformations. When forgetting the symmetric structure on $\calE$,
we obtain a planar operad $\calE_{\mathrm{pl}}$ and we can define suboperads, which are then planar operads with transformations and which
we call planar cube cutting operads, as follows: Consider the case $d=1$. Set $\calG=1$. Let $N\subset\NN$ be a set of natural numbers as in the
first paragraph. There is one very elementary subdivision of the unit interval for each $n\in N$, cutting it into $n$ pieces of equal length. 
The operations in $S$ linearly map unit intervals to the subintervals of very elementary subdivisions. This time, however, we specify the order 
of these maps. We require that they are ordered by their images via the natural ordering on the unit interval. Denote by $\calO$ the suboperad of 
$\calE_{\mathrm{pl}}$ generated by this data. Note that $\calO$ is a planar operad with transformations which is degenerate in the sense that there 
are no degree $1$ operations besides the identities. Thus, a transformation class is the same as an operation. Operations in $\calO$ are in one to 
one correspondence with subdivision of the unit interval which are obtained by iteratively applying $n$-cuts for various $n\in N$. Two operations 
are related if and only if one is a subdivision of the other. The very elementary operations are in one to one correspondence with the 
very elementary subdivisions and the elementary operations can be described just as in the case of symmetric cube cutting operads. Furthermore, 
$\calO$ satisfies the cancellative calculus of fractions.

We now look at the operad groups associated to these planar resp.~symmetric cube cutting operads. Using the fact that arrows in the fundamental groupoid of a category 
satisfying the calculus of fractions can be represented by spans, it is easy to identify the following operad groups (where $\calG=1$ in each case):
\begin{itemize}
	\item The Higman-Thompson groups $F_{n,r}$ resp.~$V_{n,r}$ arise as the operad groups (based at the object represented by a
		disjoint union of $r$ unit intervals) associated to the 
		planar resp.~symmetric cube cutting operads with $d=1$ and $N=\{n\}$.
	\item The groups of piecewise linear homeomorphisms of the (Cantor) unit interval 
		$F\big(r,\ZZ[\frac{1}{n_1\cdots n_k}],\langle n_1,\ldots,n_k\rangle\big)$
		resp.~$G\big(r,\ZZ[\frac{1}{n_1\cdots n_k}],\langle n_1,\ldots,n_k\rangle\big)$
		considered in \cite{ste:gop} arise as the operad groups (based at the object represented by a disjoint union of $r$ unit intervals)
		associated to the planar resp.~symmetric cube cutting 
		operads with $d=1$ and $N=\{n_1,\ldots,n_k\}$.
	\item The higher dimensional Thompson groups $nV$ (see \cite{bri:hdt}) arise as the operad groups (based at the object represented
		by the $n$-dimensional unit cube) associated to the symmetric
		cube cutting operads with $d=n$ and $N_j=\{2\}$ for all $j=1,\ldots, d$.
\end{itemize}

\vspace{2mm}
$\vartriangleright$ {\it Local similarity operads.} In \cite{hug:lsa} groups were defined which act in a certain way on compact ultrametric spaces. 
We recall the definition of a finite similarity structure:
\begin{defi}
	Let $X$ be a compact ultrametric space. A finite similarity structure $\mathrm{Sim}_X$ on $X$ consists of a 
	finite set $\mathrm{Sim}_X(B_1,B_2)$ of similarities $B_1\rightarrow B_2$ for every ordered pair of balls 
	$(B_1,B_2)$ such that the following axioms are satisfied: 
	\begin{itemize}
		\item ({\it Identities}) Each $\mathrm{Sim}_X(B,B)$ contains the identity.
		\item ({\it Inverses}) If $\gamma\in\mathrm{Sim}_X(B_1,B_2)$, then also $\gamma^{-1}\in\mathrm{Sim}_X(B_2,B_1)$.
		\item ({\it Compositions}) If $\gamma_1\in\mathrm{Sim}_X(B_1,B_2)$ and $\gamma_2\in\mathrm{Sim}_X(B_2,B_3)$, then 
			also $\gamma_1\gamma_2\in\mathrm{Sim}_X(B_1,B_3)$.
		\item ({\it Restrictions}) If $\gamma\in\mathrm{Sim}_X(B_1,B_2)$ and $B_3\subset B_1$ is a subball, then also 
			$\gamma|_{B_3}\in\mathrm{Sim}_X(B_3,\gamma(B_3))$.
	\end{itemize}
\end{defi}
Here, a similarity $\gamma\co X\rightarrow Y$ of metric spaces is a homeomorphism such that there is a $\lambda>0$ with 
$d(\gamma(x_1),\gamma(x_2))=\lambda d(x_1,x_2)$ for all $x_1,x_2\in X$. Let $\mathrm{Sim}_X$ be a finite similarity
structure on the compact ultrametric space $X$. A homeomorphism $\gamma\co X\rightarrow X$ is said to be locally determined by 
$\mathrm{Sim}_X$ if for every $x\in X$ there is a ball $x\in B\subset X$ such that $\gamma(B)$ is a ball and 
$\gamma|_B\in\mathrm{Sim}_X(B,\gamma(B))$. The set of all such homeomorphisms forms a group which we denote by 
$\Gamma(\mathrm{Sim}_X)$.

To a finite similarity structure $\mathrm{Sim}_X$, we can associate a symmetric operad with transformations $\calO$, 
a suboperad of $\calE=\op{End}(\mathtt{TOP},\sqcup)$, and reobtain the groups $\Gamma(\mathrm{Sim}_X)$ as operad groups. We do this 
by appealing to the procedure above.
Two balls $B_1,B_2$ in $X$ are called $\mathrm{Sim}_X$-equivalent if $\mathrm{Sim}_X(B_1,B_2)\neq\emptyset$. Choose
one ball in each $\mathrm{Sim}_X$-equivalence class (the isomorphism class of the operad we will define does not depend
on this choice). Consider the groupoid $\calG$ lying in $\mathtt{TOP}$ which is the disjoint union of the groups 
$\mathrm{Sim}_X(B,B)$ with $B$ a chosen ball. The set $S$ contains one operation in $\calE$ for each chosen ball $B$:
Consider the maximal proper subballs $A_1,\ldots,A_k$ of $B$. For each $i=1,\ldots,k$ choose a similarity $\gamma_i\in
\mathrm{Sim}_X(B_i,A_i)$ where $B_i$ is the unique chosen ball equivalent to $A_i$. Now the operation associated
to $B$ maps the chosen balls $B_i$ to $A_i$ using the similarities $\gamma_i$. The data $(\calG,S)$ generates a suboperad
$\calO$ of $\calE$.

Each transformation class in $\calO$ is uniquely determined by a chosen ball together with a subdivision into subballs.
Two such subdivisions are related if and only if one can be obtained from the other by further subdividing the subballs.
Condition $(\calV_1)$ is trivially true since the operations in $S$ have different codomains. A similarity in $\mathrm{Sim}_X(B,B)$
is an isometry $\gamma\co B\rightarrow B$ which permutes the maximal proper subballs and the restriction of $\gamma$ to
a maximal proper subball is again a similarity in $\mathrm{Sim}_X$. It follows that right multiplication of an element
in $\calG$ with an operation in $S$ gives the same operation modulo transformation. In particular, $(\calV_2)$ is satisfied.
Thus, the very elementary classes of $\calO$ are in one to one correspondence with the chosen balls together with their subdivisions 
into the proper maximal subballs. Since every two very elementary classes have different colors as codomains, there are no elementary 
classes which are not very elementary.

All the operations in $\calO$ are both mono and epi. Thus, it satisfies both left and right cancellation. It also satisfies
square filling and thus the cancellative calculus of fractions since we again find a cofinal sequence of subdivisions for each chosen
ball $B$: Define the chain inductively by subdividing each subball by their maximal proper subballs.

Using the fact that arrows in the fundamental groupoid of a category satisfying the calculus of fractions can
be represented by spans, it is not hard to establish an isomorphism $\pi_1(\calO,X)\cong\Gamma(\mathrm{Sim}_X)$ where we
assume that $X$ is the chosen ball of its $\mathrm{Sim}_X$-equivalence class.

\subsubsection{Ribbon Thompson group}\label{94422}

To close this subsection, we briefly want to discuss an operad yielding an operad group $RV$ which naturally fits into the sequence
of well-known groups $F,V,BV$. First observe the free braided operad with transformations generated by a single color, the group $\ZZ$ as groupoid 
of degree $1$ operations and a single binary operation. The components of the corresponding groupoid of transformations are the groups $B_n\ltimes\ZZ^n$. 
Think of elements of these groups as ribbons which can braid and twist. A single twist corresponds to a generator in $\ZZ$. Then we impose the
following relation on this operad:
\ifthenelse{\boolean{drawtikz}}{
\begin{center}
	\begin{tikzpicture}[scale=0.45]
		\draw (-3.4,-1.2) .. controls (-2.8,-1.2) and (-2.6,-0.8) .. (-2,-0.8);
		\draw[white, line width=3pt] (-3.4,-0.8) .. controls (-2.8,-0.8) and (-2.6,-1.2) .. (-2,-1.2);
		\draw (-3.4,-0.8) .. controls (-2.8,-0.8) and (-2.6,-1.2) .. (-2,-1.2);
		\draw (-3.4,-0.8) -- (-3.4,-1.2);
		\draw (-2,-0.8) -- (-2,-1.2);
		\draw (1,-2.2) -- (2.4,-2.2) -- (2.4,-2.6) -- (1,-2.6) -- cycle;
		\draw (1,0.6) -- (1,0.2) -- (2.4,0.2) -- (2.4,0.6) -- cycle;
		\draw (5.6,-0.8) -- (5.6,-1.2) -- (7,-1.2) -- (7,-0.8) -- cycle;
		\draw (10,-2.2) .. controls (10.8,-2.2) and (12,0.6) .. (13,0.6);
		\draw (10,-2.6) .. controls (11,-2.6) and (12.2,0.2) .. (13,0.2);
		\draw[white, line width=8pt] (10,0.4) .. controls (10.9,0.4) and (12.1,-2.4) .. (13,-2.4);
		\draw (10,0.6) .. controls (11,0.6) and (12.2,-2.2) .. (13,-2.2);
		\draw (10,0.2) .. controls (10.8,0.2) and (12,-2.6) .. (13,-2.6);
		\draw (13,0.2) .. controls (13.6,0.2) and (13.8,0.6) .. (14.4,0.6);
		\draw[white, line width=3pt] (13,0.6) .. controls (13.6,0.6) and (13.8,0.2) .. (14.4,0.2);
		\draw (13,0.6) .. controls (13.6,0.6) and (13.8,0.2) .. (14.4,0.2);
		\draw (14.4,0.6) -- (14.4,0.2);
		\draw (13,-2.6) .. controls (13.6,-2.6) and (13.8,-2.2) .. (14.4,-2.2);
		\draw[white, line width=3pt] (13,-2.2) .. controls (13.6,-2.2) and (13.8,-2.6) .. (14.4,-2.6);
		\draw (13,-2.2) .. controls (13.6,-2.2) and (13.8,-2.6) .. (14.4,-2.6);
		\draw (14.4,-2.2) -- (14.4,-2.6);
		\node at (4,-1) {$=$};
		\draw[thick] (0.8,0.4) -- (-1.8,-1) -- (0.8,-2.4);
		\draw (10,0.6) -- (10,0.2);
		\draw (10,-2.2) -- (10,-2.6);
		\draw[thick] (9.8,0.4) -- (7.2,-1) -- (9.8,-2.4);
	\end{tikzpicture}
\end{center}
}{TIKZ}
The caret corresponds to the generating binary operation. The operations in this braided operad with transformations are in one to one
correspondence with binary trees together with braiding and twisting ribbons attached to the leaves. The transformation classes are in
one to one correspondence with binary trees. The only very elementary class is represented by the binary tree with two leaves (the caret). 
There are no strictly elementary classes. It satisfies the cancellative calculus of fractions. Consequently, elements in the associated operad
group based at $1$ can be represented by pairs of binary trees where the leaves are connected by braiding and twisting ribbons.
Composition is modelled by concatenating two such tree pair diagrams, removing all dipoles formed by carets and then applying the above
relation in order to obtain another tree pair diagram.

\section{A topological finiteness result}\label{17057}

Before we state the main theorem of this article, we have to introduce two more definitions.

\begin{defi}
	We say that a group $G$ is of type $F^+_\infty$ if $G$ and all of its subgroups are of type $F_\infty$. We then say that a groupoid is of type 
	$F^+_\infty$ (or $F_\infty$) if its automorphism groups are of type $F^+_\infty$ (or $F_\infty$).
\end{defi}

For example, all finite groups and $\ZZ$ are of type $F^+_\infty$.

\begin{defi}\label{77101}
	Let $\calO$ be a (symmetric/braided) operad with transformations. An object $X$ in $\calS(\calO)$ is called
	reduced if no non-transformation arrow in $\calS(\calO)$ has $X$ as its domain. We call $\calO$ color-tame if the
	degree of all reduced objects is bounded from above.
\end{defi}

Note that if $\calO$ is monochromatic and there exists at least one higher degree operation, then it is automatically color-tame.

\begin{thm}\label{41762}
	Let $\calO$ be a planar or symmetric or braided operad with transformations. Assume that $\calO$ has only finitely many colors and
	is color-tame. Assume further that $\calO$ satisfies the cancellative calculus of fractions, is of finite type and $\calI(\calO)$ 
	is a groupoid of type $F_\infty^+$. Then for every object $X$ in $\calS(\calO)$ the operad group $\pi_1(\calO,X)$ is of type $F_\infty$.
\end{thm}

\begin{que}
	Can the requirement color-tameness be dropped?
\end{que}

\begin{rem}
	There is also a version of this theorem for free operads: Assume that $\calO$ is free as a (symmetric/braided) operad with transformations,
	has only finitely many colors, is color-tame, finitely generated and that $\calI(\calO)$ is a groupoid of type $F_\infty$.
	Then $\pi_1(\calO,X)$ is of type $F_\infty$.
	
	The proof of the free case is parallel to the proof in this article (with small modifications and additions,
	see \cite{thesis}). Parts of the theorem for the free case are also proven in \cites{far:fac,far:haf} (in the language of diagram groups).
	
	Concerning the examples in Subsection \ref{13887}, it should be noted that the free operads $\calO_F$, $\calO_V$ and $\calO_{BV}$
	also satisfy the conditions in Theorem \ref{41762}.
\end{rem}

The main tool to prove this theorem is, as usual, Brown's criterion \cite{bro:fpo}. More precisely, we will need the following special
version of it:

\begin{thm}\label{46200}
	Let $\Gamma$ be a discrete group and $X$ be a contractible $\Gamma$-CW-complex
	with isotropy groups of type $F_\infty$. Assume we have a filtration $(X_n)_{n\in\NN}$ of $X$ such that each $X_n$ is
	a $\Gamma$-CW-subcomplex of finite type and such that the connectivity of the pairs 
	$(X_n,X_{n-1})$ tends to infinity as $n\rightarrow\infty$. Then $\Gamma$ is of type $F_\infty$.
\end{thm}

We sketch a geometric proof of this criterion using a blow-up construction of L\"uck \cite{lue:tto}*{Lemma 4.1}:
For each conjugacy class $[H]$ of isotropy groups of $X$, we choose a free contractible $H$-CW-complex $EH$ of finite type.
Using these, we can construct a \emph{free} $\Gamma$-CW-complex $\calF(X)$ which is homotopy equivalent to $X$. The idea
is to replace the equivariant cell $\Gamma/H\times D^n$ in $X$ by the $\Gamma$-CW-complex $(\Gamma\times_{H}EH)\times D^n$. 
More details can be found in the proofs of \cite{lue:tto}*{Lemma 4.1 and Theorem 3.1}. We can also apply this construction
to each $\Gamma$-CW-subcomplex $X_n$ and obtain free $\Gamma$-CW-complexes $\calF(X_i)$ homotopy equivalent to $X_i$ and
of finite type. For each $n\in\NN$ we find $k\in\NN$ big enough so that $X_k$ and hence $\calF(X_k)$ is $n$-connected.
By equivariantly gluing cells in dimensions $n+2$ and higher, we obtain a free contractible $\Gamma$-CW-complex with
finitely many equivariant cells up to dimension $n+1$. Consequently, $\Gamma$ is of type $F_{n+1}$.
Since $n$ was arbitrary, it follows that $\Gamma$ is of type $F_\infty$ (see e.g.~\cite{geo:tmi}*{Proposition 7.2.2}).

The remaining subsections are devoted to the proof of Theorem \ref{41762}.

\subsection{Three types of arc complexes}\label{53852}

Let $d\in\{1,2,3\}$ and $C$ be a set of colors. Let $X=(c_1,\ldots,c_n)$ be a word in the colors of $C$. 
An archetype consists of a unique identifier together with a word in the colors of $C$ of length at least $2$.
Let $A$ be a set of archetypes. To this data, we will associate a simplicial complex $\mathcal{AC}_d(C,A;X)$.

Consider the points $1,\ldots,n\in\RR$ and embed them into $\mathbb{R}^d$ via the first component embedding 
$\RR\rightarrow\RR^d$. Color these points with the colors in the word $X$ (i.e.~the point $i$ is colored with the
color $c_i$) and call them nodes. Denote the set of nodes by $N$. A link is the image of an embedding $\gamma\co[0,1]
\rightarrow\RR^d$ such that $\gamma(0)$ and $\gamma(1)$ are nodes. Note that a link may contain more than two nodes.
Two links connecting the same set of nodes are equivalent if there is an isotopy of $\RR^d\setminus N$ which takes one link 
to the other. An  equivalence class of links is called an arc. Note that in the case $d=1$, arcs and links are the same since 
each arc is represented by a unique link. We say that two arcs are disjoint if there are representing links which are disjoint.
In the cases $d=2,3$, we can choose representing links of a collection of arcs such that the links are in minimal position:
\begin{lem}\label{20331}
	Assume $d=2$ or $d=3$. Let $\mathfrak{a}_0,\ldots,\mathfrak{a}_k$ be arcs with $\mathfrak{a}_i\neq\mathfrak{a}_j$ for each 
	$i\neq j$. Then there are representing links $\alpha_0,\ldots,\alpha_k$ such that $|\alpha_i\cap\alpha_j|$ is finite and minimal
	for each $i\neq j$.
\end{lem}
\begin{proof}
	In the case $d=3$, we can always find representing links which only intersect at nodes, if at all. The case $d=2$ is a bit 
	more complicated. We use the ideas from \cite{b-f-m-w-z:tbt}*{Lemma 3.2}: Consider the nodes as punctures in the plane $\RR^2$. Then 
	we can find a hyperbolic metric on that punctured plane. Now define $\alpha_i$ to be the geodesic within the class
	$\mathfrak{a}_i$.
\end{proof}

A link connecting a set of nodes $M$ is called admissible if there is an isotopy of $\RR^d\setminus M$ taking the link 
into the image of the first component embedding $\RR\rightarrow\RR^d$. In the case $d=1$, this is vacuous. In the case 
$d=2$, this implies in particular that, when travelling the link starting from the lowest node, the nodes are visited in 
ascending order. This last property is even equivalent to being admissible in the case $d=3$. An arc is called admissible 
if one and consequently all of its links are admissible. Now label an admissible arc with the identifier of an archetype in 
$A$. We require that the word formed by the colors of the connected nodes (in ascending order) equals the color word of the 
archetype. Call such a labelled admissible arc an archetypal arc.

The vertices of $\mathcal{AC}_d(C,A;X)$ are the archetypal arcs. Two vertices are joined by an edge if the corresponding arcs 
are disjoint. This determines the complex as a flag complex. A $k$-simplex is therefore a set of $k+1$ pairwise disjoint archetypal 
arcs. We call this an archetypal arc system. It follows from Lemma \ref{20331} above that if $\{\mathfrak{a}_0,\ldots,\mathfrak{a}_k\}$ 
is an archetypal arc system, then we always find representing links $\alpha_i$ of $\mathfrak{a}_i$ such that the $\alpha_i$ are 
pairwise disjoint. The following are examples of $2$-simplices in the cases $d=1,2,3$ (where we have omitted the labels on the arcs).
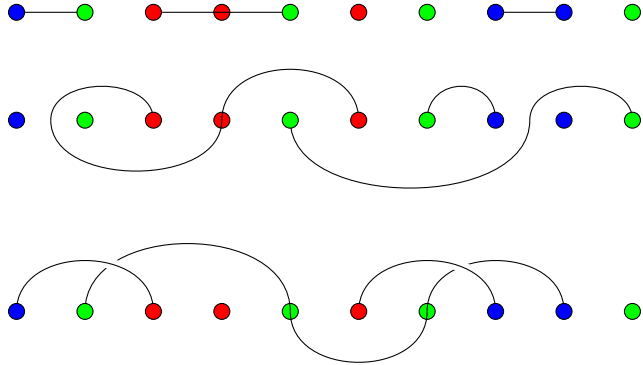
\ifthenelse{\boolean{drawtikz}}{
	\vspace{2mm}
	\begin{center}
		\begin{tikzpicture}[scale=0.6]
			\node[circle,draw=black,fill=red,inner sep=0pt,minimum size=6pt] at (4.5,0) {};
			\draw (0,0) -- (1.5,0);
			\draw (3,0) -- (6,0);
			\draw (10.5,0) -- (12,0);
			\node[circle,draw=black,fill=blue,inner sep=0pt,minimum size=6pt] at (0,0) {};
			\node[circle,draw=black,fill=green,inner sep=0pt,minimum size=6pt] at (1.5,0) {};
			\node[circle,draw=black,fill=red,inner sep=0pt,minimum size=6pt] at (3,0) {};
			\node[circle,draw=black,fill=green,inner sep=0pt,minimum size=6pt] at (6,0) {};
			\node[circle,draw=black,fill=red,inner sep=0pt,minimum size=6pt] at (7.5,0) {};
			\node[circle,draw=black,fill=green,inner sep=0pt,minimum size=6pt] at (9,0) {};
			\node[circle,draw=black,fill=blue,inner sep=0pt,minimum size=6pt] at (10.5,0) {};
			\node[circle,draw=black,fill=blue,inner sep=0pt,minimum size=6pt] at (12,0) {};
			\node[circle,draw=black,fill=green,inner sep=0pt,minimum size=6pt] at (13.5,0) {};
		\end{tikzpicture}
	\end{center}
	\vspace{2mm}
	\begin{center}
		\begin{tikzpicture}[scale=0.6]
			\node[circle,draw=black,fill=red,inner sep=0pt,minimum size=6pt] at (4.5,0) {};
			\draw (0.75,0) .. controls (0.75,1) and (3,1) .. (3,0);
			\draw (0.75,0) .. controls (0.75,-1.5) and (4.5,-1.5) .. (4.5,0);
			\draw (4.5,0) .. controls (4.5,1.5) and (7.5,1.5) .. (7.5,0);
			\draw (6,0) .. controls (6,-2) and (11.25,-2) .. (11.25,0);
			\draw (11.25,0) .. controls (11.25,1) and (13.5,1) .. (13.5,0);
			\draw (9,0) .. controls (9,1) and (10.5,1) .. (10.5,0);
			\node[circle,draw=black,fill=blue,inner sep=0pt,minimum size=6pt] at (0,0) {};
			\node[circle,draw=black,fill=green,inner sep=0pt,minimum size=6pt] at (1.5,0) {};
			\node[circle,draw=black,fill=red,inner sep=0pt,minimum size=6pt] at (3,0) {};
			\node[circle,draw=black,fill=green,inner sep=0pt,minimum size=6pt] at (6,0) {};
			\node[circle,draw=black,fill=red,inner sep=0pt,minimum size=6pt] at (7.5,0) {};
			\node[circle,draw=black,fill=green,inner sep=0pt,minimum size=6pt] at (9,0) {};
			\node[circle,draw=black,fill=blue,inner sep=0pt,minimum size=6pt] at (10.5,0) {};
			\node[circle,draw=black,fill=blue,inner sep=0pt,minimum size=6pt] at (12,0) {};
			\node[circle,draw=black,fill=green,inner sep=0pt,minimum size=6pt] at (13.5,0) {};
		\end{tikzpicture}
	\end{center}
	\begin{center}
		\begin{tikzpicture}[scale=0.6]
			\node[circle,draw=black,fill=green,inner sep=0pt,minimum size=6pt] at (6,0) {};
			\node[circle,draw=black,fill=green,inner sep=0pt,minimum size=6pt] at (9,0) {};
			\draw (1.5,0) .. controls (1.5,2) and (6,2) .. (6,0);
			\draw[white, line width=4pt] (0,0) .. controls (0,1.5) and (3,1.5) .. (3,0);
			\draw (0,0) .. controls (0,1.5) and (3,1.5) .. (3,0);
			\draw (6,0) .. controls (6,-1.5) and (9,-1.5) .. (9,0);
			\draw (9,0) .. controls (9,1.5) and (12,1.5) .. (12,0);
			\draw[white, line width=4pt] (7.5,0) .. controls (7.5,1.5) and (10.5,1.5) .. (10.5,0);
			\draw (7.5,0) .. controls (7.5,1.5) and (10.5,1.5) .. (10.5,0);
			\node[circle,draw=black,fill=blue,inner sep=0pt,minimum size=6pt] at (0,0) {};
			\node[circle,draw=black,fill=green,inner sep=0pt,minimum size=6pt] at (1.5,0) {};
			\node[circle,draw=black,fill=red,inner sep=0pt,minimum size=6pt] at (3,0) {};
			\node[circle,draw=black,fill=red,inner sep=0pt,minimum size=6pt] at (4.5,0) {};
			\node[circle,draw=black,fill=red,inner sep=0pt,minimum size=6pt] at (7.5,0) {};
			\node[circle,draw=black,fill=blue,inner sep=0pt,minimum size=6pt] at (10.5,0) {};
			\node[circle,draw=black,fill=blue,inner sep=0pt,minimum size=6pt] at (12,0) {};
			\node[circle,draw=black,fill=green,inner sep=0pt,minimum size=6pt] at (13.5,0) {};
		\end{tikzpicture}
	\end{center}
}{TIKZ}
The following are \emph{non}-examples of simplices in the case $d=2$. In the first diagram, the two arcs are not disjoint
and in the second diagram, the arc is not admissible. However, the second diagram would represent an admissible arc in the
case $d=3$.
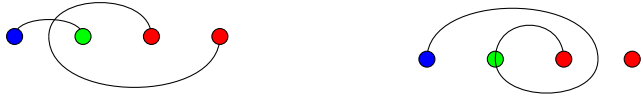
\ifthenelse{\boolean{drawtikz}}{
	\begin{center}
		\begin{tikzpicture}[scale=0.6]
			\draw (0.75,0) .. controls (0.75,1) and (3,1) .. (3,0);
			\draw (0.75,0) .. controls (0.75,-1.5) and (4.5,-1.5) .. (4.5,0);
			\draw (0,0) .. controls (0,0.5) and (1.5,0.5) .. (1.5,0);
			\node[circle,draw=black,fill=blue,inner sep=0pt,minimum size=6pt] at (0,0) {};
			\node[circle,draw=black,fill=green,inner sep=0pt,minimum size=6pt] at (1.5,0) {};
			\node[circle,draw=black,fill=red,inner sep=0pt,minimum size=6pt] at (3,0) {};
			\node[circle,draw=black,fill=red,inner sep=0pt,minimum size=6pt] at (4.5,0) {};
		\end{tikzpicture}\hspace{25mm}\begin{tikzpicture}[scale=0.6]
			\node[circle,draw=black,fill=green,inner sep=0pt,minimum size=6pt] at (1.5,0) {};
			\draw (3.75,0) .. controls (3.75,-1) and (1.5,-1) .. (1.5,0);
			\draw (0,0) .. controls (0,1.5) and (3.75,1.5) .. (3.75,0);
			\draw (1.5,0) .. controls (1.5,1) and (3,1) .. (3,0);
			\node[circle,draw=black,fill=blue,inner sep=0pt,minimum size=6pt] at (0,0) {};
			\node[circle,draw=black,fill=red,inner sep=0pt,minimum size=6pt] at (3,0) {};
			\node[circle,draw=black,fill=red,inner sep=0pt,minimum size=6pt] at (4.5,0) {};
		\end{tikzpicture}
	\end{center}
}{TIKZ}

\begin{defi}
	Let $C$ be a set of colors and $A$ be a set of archetypes. A word in the colors of $C$ is called
	reduced if it admits no archetypal arc on it. The set of archetypes $A$ is called tame if the length of all reduced
	words is bounded from above. The length of an archetype is the length of its color word. The set of archetypes $A$ is
	of finite type if the length of all archetypes is bounded from above.
\end{defi}

\begin{thm}\label{85504}
	Let $d\in\{1,2,3\}$. Let $C$ be a set of colors and $A$ be a set of archetypes. Assume that $A$ is tame and of 
	finite type. Let $m_r$ be the smallest natural number greater than the length of any reduced color word and $m_a$ 
	be the maximal length of archetypes in $A$. Define
	\[\nu_\kappa(l):=\left\lfloor\frac{l-m_r}{\kappa}\right\rfloor-1\]
	Let $X$ be a word in the colors of $C$ and denote by $lX$ the length of $X$. Then the complex $\mathcal{AC}_d(C,A;X)$ is 
	$\nu_{\kappa_d}(lX)$-connected where
	\begin{align*}
		\kappa_1 & := 2m_a+m_r-2 \\
		\kappa_2 & := 2m_a-1 \\
		\kappa_3 & := 2m_a-1
	\end{align*}
\end{thm}

For the proof in the case $d=2$ we have to pass to a slightly larger class of complexes: Instead of $\RR^2$ we consider links and
arcs in the punctured plane $S=\RR^2\setminus\{p_1,\ldots,p_l\}$ with finitely many punctures $p_i\in\RR^2$ disjoint from the nodes. 
Here, we define two links connecting the same set of nodes to be equivalent if they differ by an isotopy of $S\setminus N$ and
a link connecting a set of nodes $M$ to be admissible if there is an isotopy of $\RR^2\setminus M$ taking the link into the image 
of the first component embedding $\RR\rightarrow\RR^2$. Note that in the latter case, we require an isotopy of $\RR^2\setminus M$
and not of $S\setminus M$, i.e.~we allow the links to be pulled over punctures. We denote the corresponding complex of archetypal arc 
systems again by $\mathcal{AC}_d(C,A;X)$, suppressing the additional data of punctures since, as we will see, Theorem \ref{85504} is 
still valid for this larger class of complexes.

\subsubsection{Proof of the connectivity theorem}\label{48981}

The proof essentially consists of slightly modified ideas from \cite{b-f-m-w-z:tbt}*{Subsection 3.3}.

We induct over the length $lX$ of $X$. The induction start is $lX\geq m_r$. This implies that $X$ is not reduced and thus admits 
an archetypal arc on it. It follows that $\mathcal{AC}_d(C,A;X)$ is non-empty, i.e.~$(-1)$-connected. For the induction step,
assume $lX\geq m_r+\kappa_d$. We look at the cases $d=1,2,3$ separately, starting with the case $d=2$ since it is the hardest one.

\vspace{2mm}
$\vartriangleright$ {\it The two-dimensional case.} Choose a vertex of $\mathcal{AC}_2:=\mathcal{AC}_2(C,A;X)$ represented by an archetypal arc 
$\mathfrak{b}$. Let $v_1<\ldots<v_t$ be the nodes connected by $\mathfrak{b}$. Let $\mathcal{AC}^0_2$ be the full subcomplex of 
$\mathcal{AC}_2$ spanned by the archetypal arcs which do not meet the nodes $v_i$. 

We want to estimate the connectivity of the pair $(\mathcal{AC}_2,\mathcal{AC}^0_2)$ using the Morse method for simplicial 
complexes (see e.g.~Subsection \ref{00484}). Let $\mathfrak{a}$ be an archetypal arc. Define
\[s_i(\mathfrak{a}):=\begin{cases}1&\text{ if $\mathfrak{a}$ meets $v_i$}\\0&\text{ else}\end{cases}\]
for each $i=1,\ldots,t$. Now set
\[h(\mathfrak{a}):=\big(s_1(\mathfrak{a}),\ldots,s_t(\mathfrak{a})\big)\]
Note that the right side is a sequence of $t$ numbers in $\{0,1\}$. Interpret these sequences as binary numbers and order them 
accordingly. Then $h$ is a Morse function building up $\mathcal{AC}_2$ from $\mathcal{AC}^0_2$ since
archetypal arcs with $h$-value equal to $(0,\ldots,0)$ are exactly the archetypal arcs in $\mathcal{AC}^0_2$ and two
archetypal arcs with the same $h$-value different from $(0,\ldots,0)$ are not connected by an edge.

We want to inspect the descending links with respect to this Morse function $h$. Let $\mathfrak{a}$ be an archetypal arc with
Morse height greater than $(0,\ldots,0)$. Find the smallest $\tau\in\{1,\ldots,t\}$ such that $s_\tau(\mathfrak{a})=1$. It is not
hard to prove that $lk_\downa(\mathfrak{a})$ is the full subcomplex of $\mathcal{AC}_2$ spanned by archetypal arcs disjoint from 
$\mathfrak{a}$ and not meeting any $v_i$ with $i<\tau$. Let $X'$ be the color word which is obtained from $X$ by removing the
colors corresponding to nodes which are contained in $\mathfrak{a}$ and to the nodes $v_i$ with $i<\tau$. Then we see that
$lk_\downa(\mathfrak{a})$ is isomorphic to $\mathcal{AC}_2(C,A;X')$ with an additional puncture corresponding to $\mathfrak{a}$ 
and further additional punctures corresponding to the nodes $v_i$ with $i<\tau$. By induction, it follows that 
$lk_\downa(\mathfrak{a})$ is $\nu_{\kappa_2}(lX')$-connected. Denote by $l\mathfrak{a}$ the length of $\mathfrak{a}$,
i.e.~the number of nodes it meets. Then we can estimate
\begin{align*}
	lX' &= lX-l\mathfrak{a}-(\tau-1)\\
		&\geq lX-l\mathfrak{a}-t+1\\
		&\geq lX-m_a-t+1\\
		&\geq lX-2m_a+1
\end{align*}
Thus, $lk_\downa(\mathfrak{a})$ is $\nu_{\kappa_2}(lX-2m_a+1)$-connected. Consequently, by the Morse method, the connectivity of
the pair $(\mathcal{AC}_2,\mathcal{AC}^0_2)$ is
\[\nu_{\kappa_2}(lX-2m_a+1)+1=\nu_{\kappa_2}(lX)\]
because of $\kappa_2=2m_a-1$.

The second step of the proof consists of showing that the inclusion $\iota\co\mathcal{AC}_2^0\rightarrow\mathcal{AC}_2$ induces
the trivial map in $\pi_m$ for $m\leq\nu_{\kappa_2}(lX)$. It then follows from the long exact homotopy sequence of the pair 
$(\mathcal{AC}_2,\mathcal{AC}^0_2)$ that $\mathcal{AC}_2$ is $\nu_{\kappa_2}(lX)$-connected which completes the proof in the
case $d=2$.

Let $\varphi\co S^m\rightarrow\mathcal{AC}_2^0$ be a map with $m\leq\nu_{\kappa_2}(lX)$. We have to show that 
$\psi:={\varphi*\iota}\co S^m\rightarrow \mathcal{AC}_2$ is homotopic to a constant map. Think of $S^m$ as the boundary of an
$(m+1)$-simplex. By simplicial approximation \cite{spa:at}*{Theorem 3.4.8} we can subdivide $S^m$ and homotope $\varphi$ to
a simplicial map. So we will assume in the following that $\varphi$ is simplicial.
Next, we want to apply \cite{b-f-m-w-z:tbt}*{Lemma 3.9} in order to subdivide $S^m$ further and 
homotope $\psi$ to a simplexwise injective map. This means that whenever vertices $v\neq w$ in $S^m$ are joined by an edge, then
$\psi(v)\neq\psi(w)$. To apply the lemma, we have to show that the link of every $k$-simplex $\sigma$ in $\mathcal{AC}_2$
is $(m-2k-2)$-connected. So let $\mathfrak{a}_0,\ldots,\mathfrak{a}_k$ be pairwise disjoint archetypal arcs representing a $k$-simplex
$\sigma$. The link of this simplex is the full subcomplex spanned by the archetypal arcs which are disjoint from every 
$\mathfrak{a}_i$. Deleting every color corresponding to nodes which are contained in one of the $\mathfrak{a}_i$ from $X$, we obtain
a color word  $X'$ and it is easy to see that the link of $\sigma$ is isomorphic to $\mathcal{AC}_2(C,A;X')$ with one additional 
puncture for each $\mathfrak{a}_i$. By induction, we obtain that $lk(\sigma)$ is $\nu_{\kappa_2}(lX')$-connected. We have 
the estimate $lX'\geq lX-(k+1)m_a$ and thus
\begin{align*}
	\nu_{\kappa_2}(lX') &\geq \nu_{\kappa_2}\big(lX-(k+1)m_a\big)\\
		&= \left\lfloor\frac{lX-(k+1)m_a-m_r}{2m_a-1}\right\rfloor-1\\
		&= \left\lfloor\frac{lX-m_r}{2m_a-1}-\frac{(k+1)m_a}{2m_a-1}\right\rfloor-1\\
		&\geq \left\lfloor\frac{lX-m_r}{2m_a-1}-\frac{(2k+2)(2m_a-1)}{2m_a-1}\right\rfloor-1\\
		&= \nu_{\kappa_2}(lX)-(2k+2)\\
		&\geq m-2k-2
\end{align*}
So the hypothesis of the lemma is satisfied and we will assume in the following that $\psi$ is simplexwise injective.

We now want to show that $\psi$ can be homotoped so that the image is contained in the star of $\mathfrak{b}$. Since the star 
of a vertex is always contractible, this will finish the proof. We will homotope $\psi$ by moving single vertices of $S^m$ step 
by step, eventually landing in the star of $\mathfrak{b}$. Consider the vertices $\mathfrak{a}_1,\ldots,\mathfrak{a}_l$ of $\psi(S^m)$ which 
do net yet lie in the star of $\mathfrak{b}$, i.e.~which are not disjoint to $\mathfrak{b}$. Choose representing links $\alpha_i$ of
$\mathfrak{a}_i$ and $\beta$ of $\mathfrak{b}$ such that the system of links $(\beta,\alpha_1,\ldots,\alpha_l)$ is in minimal position
as in Lemma \ref{20331}. Note the little subtlety that archetypal arcs may have the same underlying arc but are different because
they have different labels. In this case, homotope the corresponding links a little bit so that they intersect only at nodes. Note also that
each $\alpha_i$ intersects $\beta$, but not at nodes since each $\mathfrak{a}_i$ comes from $\mathcal{AC}_2^0$. Last but not least,
we can assume that whenever $p$ is an intersection point of $\beta$ with one of the $\alpha_i$, then there is at most one $\alpha_i$
meeting the point $p$.

Now look at the intersection point $p$ of one of the $\alpha_i$ with $\beta$ which is closest to $v_1$ along $\beta$. Write $\alpha$
for the link which intersects $\beta$ at this point and $\mathfrak{a}$ for the corresponding arc.
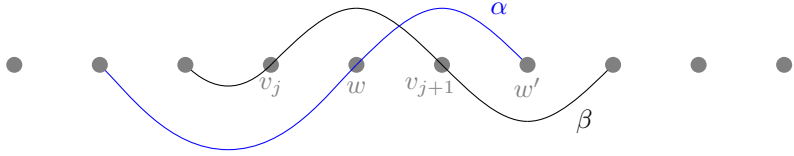
\ifthenelse{\boolean{drawtikz}}{
	\begin{center}
		\begin{tikzpicture}[scale=0.75]
			\node[circle,fill=gray,inner sep=0pt,minimum size=6pt] at (0,0) {};
			\node[circle,fill=gray,inner sep=0pt,minimum size=6pt] at (4.5,0) {};
			\node[circle,fill=gray,inner sep=0pt,minimum size=6pt] at (6,0) {};
			\node[circle,fill=gray,inner sep=0pt,minimum size=6pt] at (7.5,0) {};
			\node[circle,fill=gray,inner sep=0pt,minimum size=6pt] at (12,0) {};
			\node[circle,fill=gray,inner sep=0pt,minimum size=6pt] at (13.5,0) {};
			\draw (3,0) .. controls (3.5,-0.5) and (4,-0.5) .. (4.5,0);
			\draw (4.5,0) .. controls (5,0.5) and (5.5,1) .. (6,1);
			\draw (7.5,0) .. controls (8,-0.5) and (8.5,-1) .. (9,-1);
			\draw (7.5,0) .. controls (7,0.5) and (6.5,1) .. (6,1);
			\draw (9,-1) .. controls (9.5,-1) and (10,-0.5) .. (10.5,0);
			\draw[blue] (1.5,0) .. controls (2,-0.5) and (2.7,-1.5) .. (3.75,-1.5);
			\draw[blue] (3.75,-1.5) .. controls (4.8,-1.5) and (5.5,-0.5) .. (6,0);
			\draw[blue] (6,0) .. controls (6.5,0.5) and (7,1) .. (7.5,1);
			\draw[blue] (7.5,1) .. controls (8,1) and (8.5,0.5) .. (9,0);
			\node[circle,fill=gray,inner sep=0pt,minimum size=6pt] at (1.5,0) {};
			\node[circle,fill=gray,inner sep=0pt,minimum size=6pt] at (9,0) {};
			\node[circle,fill=gray,inner sep=0pt,minimum size=6pt] at (10.5,0) {};
			\node[circle,fill=gray,inner sep=0pt,minimum size=6pt] at (3,0) {};
			\node[blue] at (8.5,1) {$\alpha$};
			\node at (10,-1) {$\beta$};
			\node[gray] at (6,-0.4) {$w$};
			\node[gray] at (9,-0.4) {$w'$};
			\node[gray] at (4.5,-0.4) {$v_j$};
			\node[gray] at (7.3,-0.4) {$v_{j+1}$};
		\end{tikzpicture}
	\end{center}
}{TIKZ}
Choose a vertex $x$ in $S^m$ which maps to $\mathfrak{a}$ via $\psi$. Define another link $\alpha'$ as follows: Let $j$ be such 
that the intersection point $p$ lies on the segment of $\beta$ connecting $v_j$ with $v_{j+1}$. Denote by $w<w'$ the nodes such that 
$p$ lies on the segment of $\alpha$ connecting $w$ with $w'$. Now push this segment of $\alpha$ along $\beta$ over the node $v_j$
such that $\alpha$ and $\alpha'$ bound a disk whose interior does not contain any puncture or node other than $v_j$.
\ifthenelse{\boolean{drawtikz}}
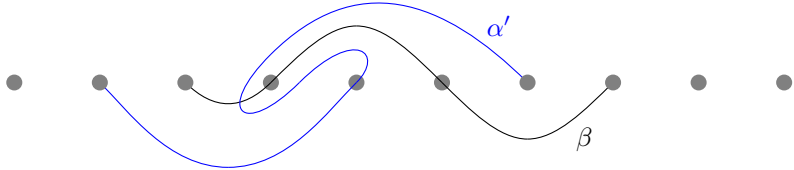
{
	\begin{center}
		\begin{tikzpicture}[scale=0.75]
			\node[circle,fill=gray,inner sep=0pt,minimum size=6pt] at (0,0) {};
			\node[circle,fill=gray,inner sep=0pt,minimum size=6pt] at (4.5,0) {};
			\node[circle,fill=gray,inner sep=0pt,minimum size=6pt] at (6,0) {};
			\node[circle,fill=gray,inner sep=0pt,minimum size=6pt] at (7.5,0) {};
			\node[circle,fill=gray,inner sep=0pt,minimum size=6pt] at (12,0) {};
			\node[circle,fill=gray,inner sep=0pt,minimum size=6pt] at (13.5,0) {};
			\draw (3,0) .. controls (3.5,-0.5) and (4,-0.5) .. (4.5,0);
			\draw (4.5,0) .. controls (5,0.5) and (5.5,1) .. (6,1);
			\draw (7.5,0) .. controls (8,-0.5) and (8.5,-1) .. (9,-1);
			\draw (7.5,0) .. controls (7,0.5) and (6.5,1) .. (6,1);
			\draw (9,-1) .. controls (9.5,-1) and (10,-0.5) .. (10.5,0);
			\draw[blue] (1.5,0) .. controls (2,-0.5) and (2.7,-1.5) .. (3.75,-1.5);
			\draw[blue] (3.75,-1.5) .. controls (4.8,-1.5) and (5.5,-0.5) .. (6,0);
			\draw[blue] (6,0) .. controls (6.5,0.5) and (6,1) .. (5,0);
			\draw[blue] (4.5,0.5) .. controls (6,2) and (7.5,1.5) .. (9,0);
			\node[circle,fill=gray,inner sep=0pt,minimum size=6pt] at (1.5,0) {};
			\node[circle,fill=gray,inner sep=0pt,minimum size=6pt] at (9,0) {};
			\node[circle,fill=gray,inner sep=0pt,minimum size=6pt] at (10.5,0) {};
			\node[circle,fill=gray,inner sep=0pt,minimum size=6pt] at (3,0) {};
			\node[blue] at (8.5,1) {$\alpha'$};
			\node at (10,-1) {$\beta$};
			\draw[blue] (4.5,0.5) .. controls (3.5,-0.5) and (4,-1) .. (5,0);
		\end{tikzpicture}
	\end{center}
}{TIKZ}
Note that $\alpha'$ is still admissible. Denote by $\mathfrak{a}'$ the archetypal arc with link $\alpha'$ and the same label as $\mathfrak{a}$.
Our goal is now to homotope $\psi$ to a simplicial map $\psi'$ such that $\psi'(x)=\mathfrak{a}'$ and $\psi'(y)=\psi(y)$ for all 
other vertices $y$. Iterating this procedure often enough, we arrive at a map $\psi^*$ homotopic to $\psi$ such that 
$\psi^*(y)\in st(\mathfrak{b})$ for each vertex $y$. For example, the next step would be to move $x$ to the vertex $\alpha''$:
\ifthenelse{\boolean{drawtikz}}
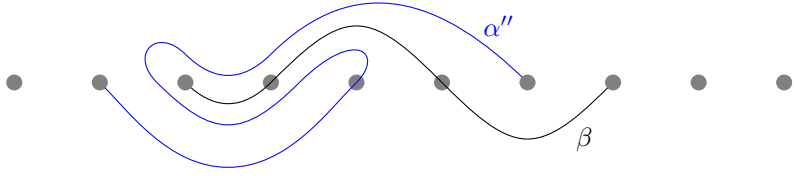
{
	\begin{center}
		\begin{tikzpicture}[scale=0.75]
			\node[circle,fill=gray,inner sep=0pt,minimum size=6pt] at (0,0) {};
			\node[circle,fill=gray,inner sep=0pt,minimum size=6pt] at (4.5,0) {};
			\node[circle,fill=gray,inner sep=0pt,minimum size=6pt] at (6,0) {};
			\node[circle,fill=gray,inner sep=0pt,minimum size=6pt] at (7.5,0) {};
			\node[circle,fill=gray,inner sep=0pt,minimum size=6pt] at (12,0) {};
			\node[circle,fill=gray,inner sep=0pt,minimum size=6pt] at (13.5,0) {};
			\draw (3,0) .. controls (3.5,-0.5) and (4,-0.5) .. (4.5,0);
			\draw (4.5,0) .. controls (5,0.5) and (5.5,1) .. (6,1);
			\draw (7.5,0) .. controls (8,-0.5) and (8.5,-1) .. (9,-1);
			\draw (7.5,0) .. controls (7,0.5) and (6.5,1) .. (6,1);
			\draw (9,-1) .. controls (9.5,-1) and (10,-0.5) .. (10.5,0);
			\draw[blue] (1.5,0) .. controls (2,-0.5) and (2.7,-1.5) .. (3.75,-1.5);
			\draw[blue] (3.75,-1.5) .. controls (4.8,-1.5) and (5.5,-0.5) .. (6,0);
			\draw[blue] (6,0) .. controls (6.5,0.5) and (6,1) .. (5,0);
			\draw[blue] (4.5,0.5) .. controls (6,2) and (7.5,1.5) .. (9,0);
			\node[circle,fill=gray,inner sep=0pt,minimum size=6pt] at (1.5,0) {};
			\node[circle,fill=gray,inner sep=0pt,minimum size=6pt] at (9,0) {};
			\node[circle,fill=gray,inner sep=0pt,minimum size=6pt] at (10.5,0) {};
			\node[circle,fill=gray,inner sep=0pt,minimum size=6pt] at (3,0) {};
			\node[blue] at (8.5,1) {$\alpha''$};
			\node at (10,-1) {$\beta$};
			\draw[blue] (2.5,0) .. controls (3.5,-1) and (4,-1) .. (5,0);
			\draw[blue] (4.5,0.5) .. controls (4,0) and (3.5,0) .. (3,0.5);
			\draw[blue] (3,0.5) .. controls (2.5,1) and (2,0.5) .. (2.5,0);
		\end{tikzpicture}
	\end{center}
}{TIKZ}

By simplexwise injectivity, no vertex of $lk(x)$ is mapped to $\mathfrak{a}$. Furthermore, a vertex of $\mathcal{AC}_2$ in the 
image of $\psi$ disjoint to $\mathfrak{a}$ must also be disjoint to $\mathfrak{a}'$ because we have chosen $\alpha$ such that 
no other $\alpha_i$ intersects $\beta$ between $p$ and $v_1$. From these observations, it follows that
\[\psi\big(lk(x)\big)\subset lk(\mathfrak{a})\cap lk(\mathfrak{a}')\]
This inclusion enables us to define a simplicial map $\psi'\co S^m\rightarrow \mathcal{AC}_2$ with $\psi'(x)=\mathfrak{a}'$ and 
$\psi'(y)=\psi(y)$ for all other vertices $y$. Let $X'$ be the color word obtained from $X$ by removing all colors corresponding 
to nodes which are contained in $\mathfrak{a}$ or to the node $v_j$. Then $lk(\mathfrak{a})\cap lk(\mathfrak{a}')$ is isomorphic to 
$\mathcal{AC}_2(C,A;X')$ with an additional puncture corresponding to the disk bounded by $\alpha\cup\alpha'$. Thus, by induction, 
it is $\nu_{\kappa_2}(lX')$-connected. We have the estimate $lX'\geq lX-m_a-1$ and therefore
\begin{align*}
	\nu_{\kappa_2}(lX') &\geq \nu_{\kappa_2}\big(lX-m_a-1\big)\\
		&= \left\lfloor\frac{lX-m_a-1-m_r}{2m_a-1}\right\rfloor-1\\
		&= \left\lfloor\frac{lX-m_r}{2m_a-1}-\frac{m_a+1}{2m_a-1}\right\rfloor-1\\
		&\geq \left\lfloor\frac{lX-m_r}{2m_a-1}-\frac{2m_a-1}{2m_a-1}\right\rfloor-1\\
		&= \nu_{\kappa_2}(lX)-1\\
		&\geq m-1
\end{align*}
Since $lk(x)$ is an $(m-1)$-sphere, this connectivity bound for $lk(\mathfrak{a})\cap lk(\mathfrak{a}')$ implies that the map
$\psi|_{lk(x)}\co lk(x)\rightarrow lk(\mathfrak{a})\cap lk(\mathfrak{a}')$ can be extended to the star $st(x)$ of $x$ which is
an $m$-disk. So we obtain a map $\vartheta\co st(x)\rightarrow lk(\mathfrak{a})\cap lk(\mathfrak{a}')$ coinciding with $\psi$ on
the boundary $lk(x)$. We can now homotope $\psi|_{st(x)}$ rel $lk(x)$ to $\vartheta$ within $st(\mathfrak{a})$ and further to 
$\psi'$ within $st(\mathfrak{a}')$. This finishes the proof of the theorem in the case $d=2$.

\vspace{2mm}
$\vartriangleright$ {\it The three-dimensional case.} Choose an archetypal arc $\mathfrak{b}$ connecting the nodes $v_1<\ldots<v_t$ and let 
$\mathcal{AC}_3^0$ be the full subcomplex of $\mathcal{AC}_3:=\mathcal{AC}_3(C,A;X)$ spanned by the archetypal arcs which do not meet the nodes $v_i$.

With a very similar Morse argument as in the case $d=2$ above, we can show that the pair $(\mathcal{AC}_3,\mathcal{AC}_3^0)$
is $\nu_{\kappa_3}(lX)$-connected.

Again, the second step consists of showing that the inclusion $\iota\co \mathcal{AC}_3^0\rightarrow \mathcal{AC}_3$ induces
the trivial map in $\pi_m$ for $m\leq\nu_{\kappa_3}(lX)$. This is much easier in the case $d=3$: Let $\varphi\co S^m\rightarrow
\mathcal{AC}_3^0$ be a map and assume without loss of generality that it is simplicial. But then the map
$\psi:=\varphi*\iota\co S^m\rightarrow\mathcal{AC}_3$ already lies in the star $st(\mathfrak{b})$ of $\mathfrak{b}$ since
an archetypal arc not meeting any of the nodes $v_i$ is already disjoint to $\mathfrak{b}$. Consequently, $\psi$ can be homotoped
to a constant map and this concludes the proof in the case $d=3$.

\vspace{2mm}
$\vartriangleright$ {\it The one-dimensional case.} Choose an archetypal arc $\mathfrak{b}$ connecting the nodes $v_1<\ldots<v_t$ such that the 
color word formed by the first $r$ nodes $w<v_1$ is reduced. Let $\mathcal{AC}^0_1$ be the full subcomplex of $\mathcal{AC}_1:=\mathcal{AC}_1(C,A;X)$ 
spanned by the archetypal arcs which do not meet the nodes $v_i$. This condition is equivalent to not meeting any nodes 
$w\leq v_t$. These are simply the first $s$ nodes $w_1<\ldots<w_s$ where $s=r+t$. In other words, $w_i$ is the point $i\in\RR$ 
colored with the color $c_i$ from $X$.

For each archetypal arc $\mathfrak{a}$ not contained in $\mathcal{AC}_1^0$ there exists a unique $1\leq q\leq s$ such that 
$\mathfrak{a}$ meets $w_q$ but not $w_1,\ldots,w_{q-1}$. In this case, define $h(\mathfrak{a})=-q$. Then $h$ is a Morse function 
building up $\mathcal{AC}_1$ from $\mathcal{AC}^0_1$. So let $\mathfrak{a}$ be such an archetypal arc. Let $X'$ be the
color word obtained from $X$ by removing all colors corresponding to the nodes contained in $\mathfrak{a}$ and to the nodes
$w_1,\ldots,w_{q-1}$. Then the descending link $lk_\downa(\mathfrak{a})$ is isomorphic to $\mathcal{AC}_1(C,A;X')$
and by induction, it is $\nu_{\kappa_1}(lX')$-connected. We can estimate
\begin{align*}
	lX' &= lX-l\mathfrak{a}-(q-1)\\
		&= lX-l\mathfrak{a}-q+1\\
		&\geq lX-m_a-q+1\\
		&\geq lX-m_a-(m_a+m_r-1)+1\\
		&= lX-2m_a-m_r+2
\end{align*}
Thus, $lk_\downa(\mathfrak{a})$ is $\nu_{\kappa_1}(lX-2m_a-m_r+2)$-connected. Consequently, by the Morse method, the connectivity
of the pair $(\mathcal{AC}_1,\mathcal{AC}_1^0)$ is
\[\nu_{\kappa_1}(lX-2m_a-m_r+2)+1=\nu_{\kappa_1}(lX)\]
because of $\kappa_1=2m_a+m_r-2$.

Just as in the case $d=3$, one can show that the inclusion $\iota\co\mathcal{AC}^0_1\rightarrow\mathcal{AC}_1$ induces
the trivial map in $\pi_m$ for $m\leq\nu_{\kappa_1}(lX)$. This proves the theorem in the case $d=1$.

\begin{rem}
	The method used in the proof of \cite{far:fac}*{Proposition 4.11} yields the better connectivity $\nu_\kappa(lX)$
	with $\kappa=m_a+m_r-1$ for the case $d=1$.
\end{rem}

\subsection{A contractible complex}

From now on, let $\calO$ be an operad as in Theorem \ref{41762}. Furthermore, let $X$ be an object in $\calS:=\calS(\calO)$. By abuse of notation, 
the connected component of $\calS$ containing the object $X$ will again be denoted by $\calS$. Furthermore, we abbreviate $\Gamma:=\pi_1(\calO,X)$.

As already noted above, the strategy to prove Theorem \ref{41762} is to apply Brown's criterion \ref{46200} to a suitable contractible complex 
on which the group in question acts. Consider the universal covering 
category $\calU:=\calU_X(\calS)$ of $\calS$ based at $X$. We claim:
\[\calU\text{ is a generalized poset and contractible.}\]
This follows from Proposition \ref{20889}. The first claim together with the remarks after Definition \ref{55528} implies that we can form the quotient category
$\calU/\calG$ where $\calG$ is the subgroupoid of $\calU$ consisting of the transformations in $\calS$ (lifted to $\calU$) and that $\calU/\calG$ is
a poset, the underlying poset of $\calU$. Recall that $\Gamma$ acts on $\calU$ which is encoded in a functor $\Gamma\rightarrow\mathtt{CAT}$ sending the unique 
object of $\Gamma$ to $\calU$. One can easily see that this functor induces a functor $\Gamma\rightarrow\mathtt{CAT}$ sending the unique object to
$\calU/\calG$. In other words, $\Gamma$ also acts on $\calU/\calG$. More concretely, an arrow $f\co X\rightarrow X$ in $\Gamma$ acts on an object 
$[g\co X\rightarrow Y]$ of $\calU/\calG$ from the right via $[g]\cdot f:=[f^{-1}g]$. This will be the action to which we want to apply Brown's criterion.
Since $\calU/\calG$ is homotopy equivalent to $\calU$, the second claim implies that also $\calU/\calG$ is contractible. So the first condition in Brown's
criterion is satisfied.

\subsection{Isotropy groups}

We continue to verify the conditions in Brown's criterion for the action of $\Gamma$ on $\calU/\calG$. In this
subsection, we show that cell stabilizers are of type $F_\infty$. First, we note that $\calU/\calG$ is indeed
a $\Gamma$-CW-complex. This follows from the following general remark.

\begin{rem}
	Let $G$ be a group acting on a category $\calC$. An element $g\in G$ fixing a cell setwise already fixes its vertices
	and so fixes the cell pointwise. Consequently, a category with an action of a discrete group $G$ is a $G$-CW-complex.
	If $\calC$ is a generalized poset, a cell stabilizer is equal to the intersection of the vertex stabilizers of that cell.
\end{rem}

In the following, we abbreviate $\calT:=\calT(\calO)$ and $\calI:=\calI(\calO)$.

\begin{lem}\label{32724}
	The groupoid $\calT$ formed by the transformations in $\calS$ is of type $F_\infty$.
\end{lem}
\begin{proof}
	By assumption, the groupoid $\calI$ formed by the degree $1$ operations is of type $F_\infty$.
	The groupoid $\calT$ is $\mathfrak{Mon}(\calI)$ in the planar case, $\mathfrak{Sym}(\calI)$ in 
	the symmetric case and $\mathfrak{Braid}(\calI)$ in the braided case (see Subsection \ref{50967} for
	the definitions of these categories).
	
	Choose a color in each component of $\calI$. Let $Y$ be an object in $\calT$. We have to show that $\op{Aut}_\calT(Y)$ is of
	type $F_\infty$. We can assume without loss of generality that
	$Y$ decomposes as a tensor product of chosen colors: $Y=c_1\otimes\ldots\otimes c_k$.
	In the planar case we have
	\[\op{Aut}_{\mathfrak{Mon}(\calI)}(Y)=\op{Aut}_\calI(c_1)\times\ldots\times\op{Aut}_\calI(c_k)\]
	and the claim follows because the $\op{Aut}_\calI(c_i)$ are of type $F_\infty$. For the symmetric and braided
	case, first assume that all the colors $c_i$ are equal to one chosen color $c$. In the symmetric case, we then have
	\[\op{Aut}_{\mathfrak{Sym}(\calI)}(Y)=S_k\ltimes\op{Aut}_\calI(c)^k\]
	where $S_k$, the symmetric group on $k$ strands, acts by permutation of the factors. More precisely, we have the
	group homomorphism
	\[\varphi\co S_k\rightarrow\op{Aut}(G^k)\hspace{8mm}\sigma\mapsto\big[(g_1,\ldots,g_k)\mapsto(g_{1\into\sigma^{-1}},\ldots,
		g_{k\into\sigma^{-1}})\big]\]
	which gives a right action of $S_k$ on $G^k$ by the definition $g\cdot\sigma=g\into(\sigma\into\varphi)$. The multiplication 
	in the semidirect product $S_k\ltimes G^k$ is then given by
	\[(\sigma,g)*(\sigma',g'):=\big(\sigma*\sigma',(g\cdot\sigma')*g'\big)\]
	Since $S_k$ is a finite group, it is also of type $F_\infty$. Since semidirect products of type
	$F_\infty$ groups are of type $F_\infty$ \cite{geo:tmi}*{Exercise 1 on page 176 and Proposition 7.2.2}, it
	follows that $\op{Aut}_{\mathfrak{Sym}(\calI)}(Y)$ is of type $F_\infty$. In the braided case we have
	\[\op{Aut}_{\mathfrak{Braid}(\calI)}(Y)=B_k\ltimes\op{Aut}_\calI(c)^k\]
	where $B_k$, the braid group on $k$ strands, acts via permutation of the factors through the projection
	$B_k\rightarrow S_k$. The braid groups $B_k$ are of type $F_\infty$ \cite{squ:tha}*{Theorem A}. As above, it
	follows that $\op{Aut}_{\mathfrak{Braid}(\calI)}(Y)$ is of type $F_\infty$.
	
	Remains to handle the case where not all the colors $c_i$ lie in the same component of $\calI$. Denote by $B'_k$ the
	finite index subgroup of $B_k$ consisting of the elements $\sigma$ with the property $c_{i\into\sigma}=c_i$. Since different
	$c_i$ are not connected by an isomorphism in $\calI$, we now have
	\[\op{Aut}_{\mathfrak{Braid}(\calI)}(Y)=B'_k\ltimes\big(\op{Aut}_\calI(c_1)\times..\times\op{Aut}_\calI(c_k)\big)\]
	where $B'_k$ still acts by permuting the factors. This action is well-defined due to the definition of $B'_k$. 
	Recall that a group is of type $F_\infty$ if and only if a finite index subgroup is of type $F_\infty$ 
	\cite{geo:tmi}*{Corollary 7.2.4}. It follows that $B'_k$ and thus $\op{Aut}_{\mathfrak{Braid}(\calI)}(Y)$ is of type 
	$F_\infty$. The symmetric case can be treated similarly.
\end{proof}

\begin{lem}\label{06276}
	Let $\calP$ be an object in $\calU/\calG$. Then the stabilizer subgroup $\op{Stab}_\Gamma(\calP)$ is of type 
	$F_\infty$.
\end{lem}
\begin{proof}
	Fix an arrow $p\co X\rightarrow Y$ in $\pi_1(\calS)$ which represents the object $\calP$ in $\calU/\calG$,
	i.e.~$[p]=\calP$. Let $\gamma\in\Gamma$ fix the point $\calP$. This means $[\gamma^{-1}p]=[p]\cdot\gamma=
	[p]$. It follows that there is some transformation $t\co Y\rightarrow Y$ such that $\gamma^{-1}p=pt$. This is equivalent to
	$p^{-1}\gamma p=t^{-1}$ which implies that $p^{-1}\gamma p$ is an element in $\op{Aut}_\calT(Y)$. Conversely,
	for $\tau$ a transformation in $\op{Aut}_\calT(Y)$, the element $p\tau p^{-1}$ is contained in $\op{Stab}_\Gamma(\calP)$.
	Thus, the map
	\[\op{Stab}_\Gamma(\calP)\rightarrow\op{Aut}_\calT(Y)\hspace{8mm}\gamma\mapsto p^{-1}\gamma p\]
	is an isomorphism with inverse given by $\tau\mapsto p\tau p^{-1}$. Since $\op{Aut}_\calT(Y)$ is of type
	$F_\infty$ by the previous lemma, the claim follows. Note that this isomorphism depends on the choice of $p$. However,
	two such choices differ by a transformation $\tau$ and the two corresponding isomorphisms differ by conjugation
	with $\tau$.
\end{proof}

We say that two operations $\theta_1$ and $\theta_2$ are \emph{two-sided} transformation equivalent if there are transformations
$\alpha,\gamma$ such that $\theta_2=\alpha*\theta_1*\gamma$.

\begin{prop}
	The stabilizer subgroups of cells in $\calU/\calG$ are of type $F_\infty$.
\end{prop}
\begin{proof}
	In the following, we restrict ourselves to the braided case. The planar and symmetric cases are similar and simpler.
	
	We first choose a color in each connected component of $\calI$. Next, we choose an operation in each two-sided transformation
	class such that the output of the chosen operation is a chosen color. 
	
	A non-degenerate cell in the geometric realization of $\calU/\calG$ is a sequence of composable non-trivial 
	arrows in $\calU/\calG$
	\begin{displaymath}\xymatrix{
		\calP_0\ar[r]^{\epsilon_0}&\calP_1\ar[r]^{\epsilon_1}&\cdots\ar[r]^{\epsilon_{k-1}}&\calP_k
	}\end{displaymath}
	Let $p_k\co X\rightarrow Y_k$ be a representing path 
	of $\calP_k$ such that  $Y_k=c_1\otimes\ldots\otimes c_l$ is a tensor product of chosen colors. In the proofs of Lemmas \ref{32724}
	and \ref{06276}, we have seen that $p_k$ induces an isomorphism
	\[\varphi\co\op{Stab}_\Gamma(\calP_k)\rightarrow B'_l\ltimes\big(\op{Aut}_\calI(c_1)\times..\times\op{Aut}_\calI(c_l)\big)\]
	
	Choose some $\calP_i=:\calP$ different from $\calP_k$ and observe the arrow $\epsilon\co\calP\rightarrow\calP_k$ 
	which is the composition of the $\epsilon_j$ in between. Choose a representing path $p\co X\rightarrow Y$ of $\calP$. 
	Then there is exactly one arrow $e\co Y\rightarrow Y_k$ representing $\epsilon$. One can compose $p$ and $p_k$ with
	transformations $\eta$ and $\lambda$ such that $\lambda\co Y_k\rightarrow Y_k$ is a tensor product of degree $1$ operations 
	$\lambda_i\co c_i\rightarrow c_i$ and $e$ is a tensor product of chosen operations. Write $p'_k=p_k\lambda$ for the new representative
	of $\calP_k$. To $p'_k$ corresponds another isomorphism
	\[\varphi'\co\op{Stab}_\Gamma(\calP_k)\rightarrow B'_l\ltimes\big(\op{Aut}_\calI(c_1)\times..\times\op{Aut}_\calI(c_l)\big)\]
	which differs from $\varphi$ by conjugation with $\lambda$. Denote the new representative $p\eta$ of $\calP$ again by $p$. 
	
	Now let $\gamma\in\op{Stab}_\Gamma(\calP_k)$. Then $\gamma$ fixes also $\calP$, i.e.~$\calP\cdot\gamma=\calP$, if and only if
	\begin{align*}
		[p'_ke^{-1}] = [p] & = [p]\cdot\gamma\\
		& = [\gamma^{-1}p] \\
		& = [\gamma^{-1}p'_ke^{-1}] \\
		& = [p'_k{p'_k}^{-1}\gamma^{-1}p'_ke^{-1}] \\
		& = [p'_kt_\gamma^{-1} e^{-1}]
	\end{align*}
	where we have set $t_\gamma:={p'_k}^{-1}\gamma p'_k$, an element in the image of the isomorphism $\varphi'$. Therefore,
	we have to identify all such $t_\gamma$ which satisfy this equation. In other words, we look for all $t_\gamma$ such that 
	there is a transformation $\tau$ with 
	\[e t_\gamma=\tau e\]
	Roughly speaking, we look for all $t_\gamma$ which can be pulled through $e$ from the codomain to the domain.
	
	For better readability, we assume without loss of generality that the colors $c_i$ are all equal to one color $c$. In 
	particular, the codomains of $\varphi$ and $\varphi'$ are of the form $B_l\ltimes\op{Aut}_\calI(c)^l$. Then write 
	$e=\theta_1\otimes\ldots\otimes\theta_l$ where the $\theta_i$ are chosen operations with codomain the chosen color $c$.
	Define $H_i$ to be the subgroup of $\op{Aut}_\calI(c)$ consisting of elements $h$ which can be pulled through the operation 
	$\theta_i$, i.e.~there exists a transformation $\tau$ with $\theta_i h=\tau\theta_i$. Furthermore, let $B^*_l$ be the finite 
	index subgroup of $B_l$ consisting of the elements $\sigma$ with the property $\theta_{i\into\sigma}=\theta_i$. Denote by 
	$\op{Stab}_\Gamma(\calP,\calP_k)$ the subgroup of $\op{Stab}_\Gamma(\calP_k)$ which 
	also fixes $\calP$. Then the isomorphism $\varphi'$ restricts to an isomorphism
	\[\varphi'_\calP\co\op{Stab}_\Gamma(\calP,\calP_k)\rightarrow B^*_l\ltimes(H_1\times\ldots\times H_l)=:\Lambda_\calP\]
	where the subgroup $B^*_l$ still acts via permutation of the factors and this is well-defined due to the definition of $B^*_l$. 
	The proof of this is straightforward and uses the fact that two two-sided transformation equivalent $\theta_i$ must be equal.
	
	Recall that $\varphi'$ differs from $\varphi$ by conjugation with $\lambda$. So the image of $\op{Stab}_\Gamma(\calP,\calP_k)$
	under $\varphi$ is its image under $\varphi'$ conjugated with $\lambda$. More precisely, $\varphi$ restricts to an isomorphism
	\[\varphi_\calP\co\op{Stab}_\Gamma(\calP,\calP_k)\rightarrow\lambda\Lambda_\calP\lambda^{-1}=:\Omega_\calP\]
	Consider the pure braid group $P_l$ which is a finite index subgroup of $B_l$. It is also a finite index subgroup of $B^*_l$.
	Recall that we have $\lambda=\lambda_1\otimes\ldots\otimes\lambda_l$ where the $\lambda_i$ are degree $1$ operations. We have
	\begin{align*}	
		\lambda\big(P_l\ltimes(H_1\times\ldots\times H_l)\big)\lambda^{-1} 
			&= \lambda\big(P_l\times(H_1\times\ldots\times H_l)\big)\lambda^{-1}\\
			&= P_l\times\big(\lambda_1 H_1\lambda_1^{-1}\times\ldots\times\lambda_l H_l\lambda_l^{-1}\big)\\
			&= P_l\times\big(H^\calP_1\times\ldots\times H^\calP_l\big)
	\end{align*}
	where $H^\calP_i:=\lambda_i H_i\lambda_i^{-1}$ is isomorphic to $H_i$. This is a finite index subgroup of $\Omega_\calP$.
	
	Remains to consider the case when $\gamma\in\op{Stab}_\Gamma(\calP_k)$ fixes more than one additional
	vertex $\calP_i$. For this we have to show that the intersection
	\[\Omega_{\calP_0}\cap\ldots\cap\Omega_{\calP_{k-1}}\subset B_l\ltimes\op{Aut}_\calI(c)^l\]
	is of type $F_\infty$. For better readability, we assume without loss of generality that $k=2$. Then the last statement is
	equivalent to
	\begin{multline*}
		\left(P_l\times\big(H^{\calP_0}_1\times\ldots\times H^{\calP_0}_l\big)\right)\cap 
		\left(P_l\times\big(H^{\calP_1}_1\times\ldots\times H^{\calP_1}_l\big)\right)=\\
		P_l\times\left(\big(H^{\calP_0}_1\cap H^{\calP_1}_1\big)\times\ldots\times\big(H^{\calP_0}_l\cap H^{\calP_1}_l\big)\right)
	\end{multline*}
	being of type $F_\infty$ since it is a finite index subgroup. This is true because $P_l$ is of type $F_\infty$ and 
	all the groups $H^{\calP_0}_i\cap H^{\calP_1}_i$ are of type $F_\infty$. The latter statement is true because
	the groups $H^{\calP_0}_i\cap H^{\calP_1}_i$ are subgroups of $\op{Aut}_\calI(c)$ which is of type $F^+_\infty$.
	This completes the proof of the proposition.
\end{proof}

\subsection{Finite type filtration}

To apply Brown's criterion to the $\Gamma$-CW-complex $\calU/\calG$, we need a filtration by
$\Gamma$-CW-subcomplexes $(\calU/\calG)_n$ which are of finite type. Observe that the degree function on $\calS$ induces degree
functions on $\calU$ and $\calU/\calG$. Define $\calS_n$ resp.~$\calU_n$ resp.~$(\calU/\calG)_n$ to be the full
subcategories spanned by the objects of degree at most $n$. Note that we have $\calU_n/\calG_n=(\calU/\calG)_n$ where $\calG_n=\calG\cap\calU_n$. 
In the following, we want to show that $(\calU/\calG)_n$ only has finitely many $\Gamma$-equivariant cells in each dimension.

\vspace{2mm}
Choose one operation in each very elementary transformation class and denote the resulting set of operations by $S$. 
By the assumptions in Theorem \ref{41762}, $S$ is a finite set. 
\begin{obs}\label{37111}
	Let $\theta\in S$ and $\gamma$ be a degree $1$ operation such that $\theta*\gamma$ is defined. Then, by Proposition \ref{07546},
	$\theta*\gamma$ is again very elementary and there is a $\theta'\in S$ and a transformation $\tau$ with $\theta*\gamma=\tau*\theta'$.
\end{obs}
Denote by $\Omega$ the set of all identity operations together with all operations of degree at most $n$ which are obtained 
by partially composing operations in $S$. Note that $\Omega$ is finite because $S$ is finite and there are only finitely
many colors by assumption. Denote by $\Lambda$ the set of arrows in $\calS_n$ which are obtained by taking tensor products of
operations in $\Omega$. Again, the set $\Lambda$ is finite. 

Let $\Lambda^{*p}\subset\Lambda^p$ be the subset of $p$-tuples of composable arrows in $\Lambda$. We claim that there is a surjective function
\[\Lambda^{*p}\twoheadrightarrow\big\{p\text{-cells in }(\calU/\calG)_n\big\}\big/\Gamma\]
which proves that there are only finitely many $\Gamma$-equivariant cells in $(\calU/\calG)_n$. Let
$(e_0,\ldots,e_{p-1})\in\Lambda^{*p}$. Choose a path $p_0\co X\rightarrow\op{dom}(e_0)$. Define paths $p_k\co X\rightarrow
\op{dom}(e_k)$ by the composite $p_k:=p_0e_0\ldots e_{k-1}$. The $p_i$ represent objects $\calP_i$ and the $e_i$ 
represent arrows $\epsilon_i\co\calP_i\rightarrow\calP_{i+1}$ in $(\calU/\calG)_n$. Thus, the sequence 
$\epsilon_0,\ldots,\epsilon_{p-1}$ gives a $p$-cell in $(\calU/\calG)_n$. This $p$-cell surely depends on the choice
of $p_0$ but two such choices give equivalent $p$-cells modulo the action of $\Gamma$. So we get a well-defined function
as above.

Remains to show that this function is indeed surjective. Consider a $p$-cell in $(\calU/\calG)_n$ in the form of a string
\[\calP_0\xrightarrow{\epsilon_0}\calP_1\xrightarrow{\epsilon_1}\cdots\xrightarrow{\epsilon_{p-1}}\calP_{p}\]
of composable arrows in $(\calU/\calG)_n$. For each $\calP_i$ we can choose representatives $P_i$ in $\calU_n$. Then each 
$\epsilon_i$ is represented by a unique arrow $e_i\co P_i\rightarrow P_{i+1}$ in $\calU_n$. We now want to change these 
representatives so that each $e_i$ lies in $\Lambda$. 

Start with the last arrow $e_{p-1}=[\sigma,\Theta]$. Let $T$ be the set of operations of the form $\tau*\theta$ 
where $\tau$ is a transformation and $\theta\in S$. In other words, $T$ is the set of all very 
elementary operations. Each higher degree operation $\theta$ in the sequence $\Theta$ can be written,
up to transformation, as a partial composition of operations in $T$ (see the remarks after Definition \ref{13449}).
It follows $\theta=s*\psi$ where $s$ is a transformation and $\psi$ is an operation decomposable into operations of
the form $(\gamma_1,\ldots,\gamma_k)*\xi$ with $\xi\in S$ and $\gamma_i$ of degree $1$. Using Observation \ref{37111}, we can pull the
degree $1$ operations to the domain of $\psi$, starting with the rightmost degree $1$ operations, and obtain
$\theta=s*\psi$ where $s$ is a transformation and $\psi$ is an operation decomposable into 
operations of $S$. We now have $e_{p-1}=\tau*\Psi$ where $\tau$ is some transformation and $\Psi$ is
simply a tensor product of identities or higher degree operations decomposable into operations of $S$. By changing the 
representives $P_{p-1}$, $e_{p-1}$ and $e_{p-2}$ in their respective classes modulo the subgroupoid $\calG$, we can assume
$\tau=\id$ and thus that $e_{p-1}$ lies in $\Lambda$. We can now repeat this argument with $e_{p-2}$ and then
with $e_{p-3}$ and so forth until we have changed each $e_i$ to lie in $\Lambda$. This proves surjectivity.

\subsection{Connectivity of the filtration}

It remains to show the connectivity statement in Brown's criterion, i.e.~we have to show that the connectivity
of the pair $\big((\calU/\calG)_n,(\calU/\calG)_{n-1}\big)$ tends to infinity as $n\rightarrow\infty$. To show this,
we apply the Morse method for categories. The degree function on $\calU/\calG$ is a Morse
function and the corresponding filtration is exactly $(\calU/\calG)_n$. Thus, we have to prove that the 
connectivity of the descending link $lk_\downa(\calK)$ tends to infinity as the degree of the object $\calK$ tends to infinity. 
Note that the descending up link $\overline{lk}_\downa(\calK)$ is always empty, so we have $lk_\downa(\calK)=
\underline{lk}_\downa(\calK)$.

\begin{defi}
	An arrow $[\sigma,\Theta]$ in $\calS$ is called (very) elementary if it is not a transformation and every higher degree
	operation in $\Theta$ is (very) elementary. An arrow in $\calU$ is called (very) elementary if the
	corresponding arrow in $\calS$ is (very) elementary. An arrow in $\calU/\calG$ is called (very) elementary if
	there is a (very) elementary representative in $\calU$.
\end{defi}

It follows from Proposition \ref{07546} that the number of (very) elementary operations in an arrow 
$a\in\calU$ does does not change if we replace $a$ by another representative in the class $[a]\in\calU/\calG$. In particular,
if the arrow $\alpha\in\calU/\calG$ is (very) elementary, then \emph{all} representing arrows $a\in\calU$ of $\alpha$ are (very) 
elementary.

\vspace{2mm}
The data of an object in $\underline{lk}_\downa(\calK)$ consists of an object $\calY$ in $\calU/\calG$ with 
$\op{deg}(\calY)<\op{deg}(\calK)$ and an arrow $\alpha\co\calK\rightarrow\calY$ in $\calU/\calG$. Now we define $\op{Core}(\calK)$ 
to be the full subcategory of $\underline{lk}_\downa(\calK)$ spanned by the objects $(\calY,\alpha)$ where $\alpha$ is a very 
elementary arrow. Denote by $\op{Corona}(\calK)$ the full subcategory of $\underline{lk}_\downa(\calK)$ spanned by the objects 
$(\calY,\alpha)$ with $\alpha$ an elementary arrow. So we have
\[\op{Core}(\calK)\subset\op{Corona}(\calK)\subset\underline{lk}_\downa(\calK)\]
and we will study the connectivity of these spaces successively.

\subsubsection{The core}

In this subsubsection, we adopt the normal form point of view of Subsection \ref{88061}: Arrows in $\calS$ are always represented 
by a unique pair $(\sigma,\Theta)$ such that $\sigma^{-1}$ is unpermuted resp.~unbraided on the domains of the operations in the sequence 
$\Theta$.

We say that two operations $\theta_1$ and $\theta_2$ are \emph{right} transformation equivalent if there is a transformation
$\gamma$ such that $\theta_2=\theta_1*\gamma$. Recall from Proposition \ref{07546} that being elementary or
very elementary is invariant under right transformations.

The object $\calK$ in $\calU/\calG$ is a class of objects in $\calU$ modulo transformations. Fix some
representing object $K$. Then the objects in $\op{Core}(\calK)$ are in one to one correspondence with pairs $(Y,a)$
where $Y$ is an object in $\calU$ with $\op{deg}(Y)<\op{deg}(K)$ and $a\co K\rightarrow Y$ is a very elementary arrow in 
$\calU$ modulo transformations on the codomain (compare with Remark \ref{64237}). 
Choose one operation in each right transformation equivalence class and denote the resulting
set of operations by $R$. We choose the identity for a class of degree $1$ operations so that the degree $1$ operations
in $R$ are identities. Now define a very elementary $R$-arrow to be a very elementary
arrow $(\sigma,\Theta)$ in $\calS$ such that the operations in $\Theta$ are elements of $R$. Thus, $\Theta$ is a tensor 
product of identities and at least one very elementary operation lying in $R$. This notion of very elementary $R$-arrows 
can be lifted to arrows in $\calU$. Now the objects in $\op{Core}(\calK)$ are in one to one correspondence with pairs 
$(Y,a)$ where $Y$ is an object with $\op{deg}(Y)<\op{deg}(K)$ and $a\co K\rightarrow Y$ is 
\begin{itemize}
	\item (planar case) a very elementary $R$-arrow.
	\item (symmetric case) a very elementary $R$-arrow modulo colored permutations on the codomain.
	\item (braided case) a very elementary $R$-arrow modulo colored braidings on the codomain.
\end{itemize}
The equivalence relation modulo braidings on the codomain is called ``dangling'' in \cite{b-f-m-w-z:tbt} because these objects may be 
visualized as a braiding where some strands at one end are connected by very elementary operations in $R$, called ``feet'', and
these are allowed to dangle freely (see \cite{b-f-m-w-z:tbt}*{Figure 9}).

Now let $C$ be the set of colors of the operad $\calO$. We define a set of archetypes $A$ as follows: For each
operation in $R$, form an archetype with identifier this operation and with color word the domain of that
operation. The object $K$ in $\calU$ is a path of arrows in $\calS$ modulo homotopy. It starts at the
color word $X$ and ends at some other color word $T$. Consider the simplicial complex $\mathcal{AC}_d(C,A;T)$ from Subsection 
\ref{53852}. It can be seen as a poset of simplices with an arrow from a simplex $\sigma$ to another simplex $\sigma'$
if and only if $\sigma$ is a face of $\sigma'$.

\begin{prop}\label{19549}
	The category $\op{Core}(\calK)$ is a poset and isomorphic, as a poset, to $\mathcal{AC}_d(C,A;T)$
	where $d=1$ in the planar case, $d=2$ in the braided case and $d=3$ in the symmetric case.
\end{prop}
\begin{proof}
	We restrict our attention to the braided case, i.e.~$d=2$. The other two cases are much simpler.
	
	First, it is clear that $\op{Core}(\calK)$ is a poset since $\calU/\calG$ is a poset. We want to understand
	the poset structure a bit better:
	Let $\Lambda$ be an object of $\op{Core}(\calK)$ in the form of a very elementary $R$-arrow $K\rightarrow Y$
	modulo dangling. Fix some very elementary $R$-arrow $\lambda$ representing this class with the property that 
	the colored braiding of that arrow is unbraided not only on the sets of strands connected to
	single operations but also on the set of strands connected to identity operations. Then arrows in $\op{Core}(\calK)$
	with domain $\Lambda$ are in one to one correspondence with very elementary $R$-arrows $\alpha$ in $\calU$, modulo dangling,
	such that the very elementary operations of $\alpha$ only connect to identity operations of $\lambda$ in the composition 
	$\lambda*\alpha$ (since compositions of very elementary operations are not very elementary anymore). The following
	diagram, in which the gray triangles are identity operations, illustrates such a situation:
	\ifthenelse{\boolean{drawtikz}}
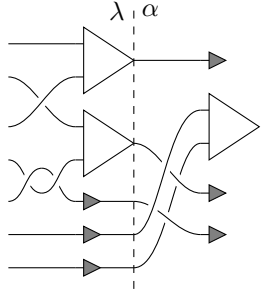
{
	\begin{center}
		\begin{tikzpicture}[scale=0.22]
			\draw (-2,3) -- (1,1) -- (-2,-1) -- cycle;
			\draw (-2,-2) -- (1,-4) -- (-2,-6) -- cycle;
			\draw[fill=gray] (-2,-7) -- (-1,-7.5) -- (-2,-8) -- cycle;
			\draw[fill=gray] (-2,-9) -- (-1,-9.5) -- (-2,-10) -- cycle;
			\draw[fill=gray] (-2,-11) -- (-1,-11.5) -- (-2,-12) -- cycle;
			\draw (-2,2) -- (-2.5,2);
			\draw (-2,0) -- (-2.5,0);
			\draw (-2,-3) -- (-2.5,-3);
			\draw (-2,-5) -- (-2.5,-5);
			\draw (-2,-7.5) -- (-2.5,-7.5);
			\draw (-2,-9.5) -- (-2.5,-9.5);
			\draw (-2,-11.5) -- (-2.5,-11.5);
			\draw (-2.5,2) -- (-6.5,2);
			\draw (-2.5,-11.5) -- (-6.5,-11.5);
			\draw (-2.5,-9.5) -- (-6.5,-9.5);
			\draw (-2.5,-7.5) .. controls (-3.5,-7.5) and (-3.5,-5.5) .. (-4.5,-5.5);
			\draw[white, line width=4pt] (-2.5,-5) .. controls (-3.5,-5) and (-3.5,-7) .. (-4.5,-7);
			\draw (-2.5,-5) .. controls (-3.5,-5) and (-3.5,-7) .. (-4.5,-7);
			\draw (-4.5,-7) .. controls (-5.5,-7) and (-5.5,-5) .. (-6.5,-5);
			\draw (-2.5,-3) .. controls (-4,-3) and (-5,0) .. (-6.5,0);
			\draw[white, line width=4pt] (-2.5,0) .. controls (-4,0) and (-5,-3) .. (-6.5,-3);
			\draw (-2.5,0) .. controls (-4,0) and (-5,-3) .. (-6.5,-3);
			\draw[white, line width=4pt] (-4.5,-5.5) .. controls (-5.5,-5.5) and (-5.5,-7.5) .. (-6.5,-7.5);
			\draw (-4.5,-5.5) .. controls (-5.5,-5.5) and (-5.5,-7.5) .. (-6.5,-7.5);
			\draw (-1,-7.5) -- (1,-7.5);
			\draw (-1,-9.5) -- (1,-9.5);
			\draw (-1,-11.5) -- (1,-11.5);
			\draw (5.5,-1) -- (5.5,-5) -- (8.5,-3) -- cycle;
			\draw[fill=gray] (5.5,0.5) -- (5.5,1.5) -- (6.5,1) -- cycle;
			\draw[fill=gray] (5.5,-6.5) -- (5.5,-7.5) -- (6.5,-7) -- cycle;
			\draw[fill=gray] (5.5,-9) -- (5.5,-10) -- (6.5,-9.5) -- cycle;
			\draw (5.5,1) -- (5,1);
			\draw (5.5,-2) -- (5,-2);
			\draw (5.5,-4) -- (5,-4);
			\draw (5.5,-7) -- (5,-7);
			\draw (5.5,-9.5) -- (5,-9.5);
			\draw (5,1) -- (1,1);
			\draw (5,-4) .. controls (3.5,-4) and (3,-11.5) .. (1,-11.5);
			\draw[white, line width=4pt] (5,-7) .. controls (3.5,-7) and (2.5,-4) .. (1,-4);
			\draw (5,-7) .. controls (3.5,-7) and (2.5,-4) .. (1,-4);
			\draw[white, line width=4pt] (5,-9.5) .. controls (3.5,-9.5) and (2.5,-7.5) .. (1,-7.5);
			\draw (5,-9.5) .. controls (3.5,-9.5) and (2.5,-7.5) .. (1,-7.5);
			\draw[white, line width=4pt] (5,-2) .. controls (3,-2) and (2.5,-9.5) .. (1,-9.5);
			\draw (5,-2) .. controls (3,-2) and (2.5,-9.5) .. (1,-9.5);
			\draw[dashed] (1,4) -- (1,-13);
			\node at (0,4) {$\lambda$};
			\node at (2,4) {$\alpha$};
		\end{tikzpicture}
	\end{center}
	}{TIKZ}	
	These considerations yield the following interpretation of the poset structure: We have $\Lambda\rightarrow\Lambda'$ if and only 
	if there is a very elementary $R$-arrow $\lambda$ representing the dangling class $\Lambda$ such that adding very elementary operations
	of $R$ to loose strands of $\lambda$ (i.e.~strands connected to identity operations) gives a very elementary $R$-arrow representing
	the dangling class $\Lambda'$.
	
	We will consider an isomorphism of posets
	\[\mathrm{comb}\co \op{Core}(\calK)\rightarrow \calA\calC_2(C,A;T)\]
	called ``combing'' as in \cite{b-f-m-w-z:tbt}*{Section 4} and its inverse
	\[\mathrm{weave}\co \calA\calC_2(C,A;T)\rightarrow \op{Core}(\calK)\]
	which we call ``weaving''. 
	
	To define the first map, start with an object $\Lambda$ in $\op{Core}(\calK)$. As above, it is a very
	elementary $R$-arrow in normal form modulo dangling. Thus, it is represented by a colored braid with unbraided 
	strands connected by very elementary operations in $R$. Think of the domain of the braid as being fixed on the 
	line 
	\[L_1:=\{(x,0,1)\mid x\in\RR\}\subset\RR^3\]
	the codomain as being fixed on the line
	\[L_0:=\{(x,0,0)\mid x\in\RR\}\subset\RR^3\]
	and visualize the operations as straight lines in $L_0$ connecting
	the ends of the corresponding strands. Now ``combing straight'' the braid means moving around the ends
	of the braid in the plane $P:=\RR^2\times\{0\}\subset\RR^3$ such that the whole braid becomes unbraided.
	The segments representing the operations get deformed in $P$ this way and in fact become the archetypal arcs
	in $\mathrm{comb}(\Lambda)$. They are admissible because the braid was required to be unbraided on the 
	domains of the operations. This process is visualized in \cite{b-f-m-w-z:tbt}*{Figure 17}. Note that combing does 
	not depend on the representative under dangling, so it is a well-defined map on the objects of $\op{Core}(\calK)$.
	It also respects the poset structures, so it is a map of posets.
	
	To define the second map, start with an archetypal arc system $\calA$. This is a priori embedded in $\RR^2$ but
	embed it in $\RR^3$ via the embedding $\RR^2\times\{0\}\subset\RR^3$. Connect the nodes of the archetypal arc 
	system with the line $L_1$ by straight lines parallel to the third coordinate axis. The process of weaving first 
	tries to separate the archetypal arcs by moving the nodes in the plane $P$. Here, being separate means being separated
	by a straight line in $P$ parallel to the second coordinate axis. Also, the set of nodes which are not contained in
	an arc should be separated from the arcs. By doing these moves, the vertical strands connecting the nodes with
	the line $L_1$ become braided in a certain way. The separation process is always possible but the resulting braid 
	is not unique (think of two nodes connected by an arc and turn around the arc several times). To make the resulting 
	braid unique (up to dangling), we additionally require that the subbraid determined by an archetypal arc never 
	becomes braided during the separation process. This can be achieved for example by the following additional movement 
	rule: The nodes of an archetypal arc always have to stay on the same line $L$ in $P$ parallel to the first coordinate 
	axis. This line $L$ may move up and down and the nodes of the archetypal arc may move left and right on $L$ but they 
	must never cross each other on $L$. Then, when the archetypal arcs are separated from each other and from the isolated 
	nodes, the property admissible of the archetypal arcs ensures that they can be homotoped to straight lines lying in
	$L_0$. The following figure visualizes this process:
	\ifthenelse{\boolean{drawtikz}}
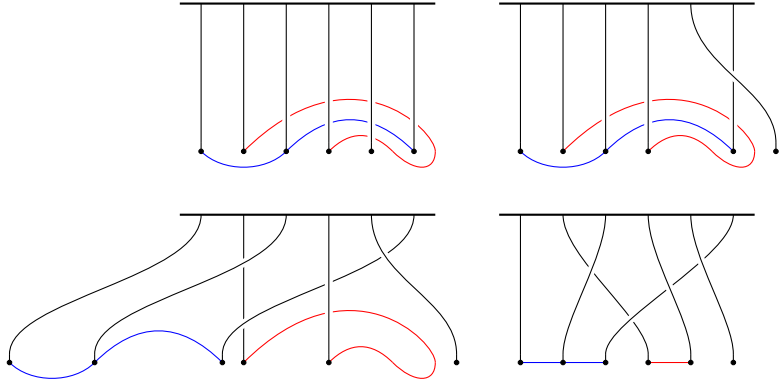
{
	\begin{center}
		\begin{tikzpicture}[scale=0.28]
			\draw[blue] (0,-10) .. controls (1,-11) and (3,-11) .. (4,-10);
			\draw[blue] (4,-10) .. controls (6,-8) and (8,-8) .. (10,-10);
			\draw[red] (6,-10) .. controls (7,-9) and (8,-9) .. (9,-10);
			\draw[red] (9,-10) .. controls (10,-11) and (11,-11) .. (11,-10);
			\draw[red] (11,-10) .. controls (11,-9) and (7,-5) .. (2,-10);
			\draw[white, line width=3pt] (2,-3) -- (2,-9.6);
			\draw (2,-3) -- (2,-9.6);
			\draw[white, line width=3pt] (0,-9.6) -- (0,-3);
			\draw (0,-9.6) -- (0,-3);
			\draw[white, line width=3pt] (8,-3) -- (8,-9.6);
			\draw (8,-3) -- (8,-9.6);
			\draw[white, line width=3pt] (6,-3) -- (6,-9.6);
			\draw (6,-3) -- (6,-9.6);
			\draw[white, line width=3pt] (4,-3) -- (4,-9.6);
			\draw (4,-3) -- (4,-9.6);
			\draw[white, line width=3pt] (10,-3) -- (10,-9.6);
			\draw (10,-3) -- (10,-9.6);
			\filldraw (0,-10) circle (3pt);
			\filldraw (2,-10) circle (3pt);
			\filldraw (4,-10) circle (3pt);
			\filldraw (6,-10) circle (3pt);
			\filldraw (8,-10) circle (3pt);
			\filldraw (10,-10) circle (3pt);
			\draw[thick] (-1,-3) -- (11,-3);
			\draw (0,-9.6) -- (0,-10);
			\draw (8,-10) -- (8,-9.6);
			\draw (10,-9.6) -- (10,-10);
			\draw (6,-10) -- (6,-9.6);
			\draw (4,-9.6) -- (4,-10);
			\draw (2,-10) -- (2,-9.6);
			
			\draw[blue] (15,-10) .. controls (16,-11) and (18,-11) .. (19,-10);
			\draw[blue] (19,-10) .. controls (21,-8) and (23,-8) .. (25,-10);
			\draw[red] (21,-10) .. controls (22,-9) and (23,-9) .. (24,-10);
			\draw[red] (24,-10) .. controls (25,-11) and (26,-11) .. (26,-10);
			\draw[red] (26,-10) .. controls (26,-9) and (22,-5) .. (17,-10);
			\draw[white, line width=3pt] (17,-3) -- (17,-9.6);
			\draw (17,-3) -- (17,-9.6);
			\draw[white, line width=3pt] (15,-9.6) -- (15,-3);
			\draw (15,-9.6) -- (15,-3);
			\draw[white, line width=3pt] (21,-3) -- (21,-9.6);
			\draw (21,-3) -- (21,-9.6);
			\draw[white, line width=3pt] (19,-3) -- (19,-9.6);
			\draw (19,-3) -- (19,-9.6);
			\draw[white, line width=3pt] (25,-3) -- (25,-9.6);
			\draw (25,-3) -- (25,-9.6);
			\draw[white, line width=3pt] (23,-3) .. controls (23,-6) and (27,-7) .. (27,-9.6);
			\draw (23,-3) .. controls (23,-6) and (27,-7) .. (27,-9.6);
			\filldraw (15,-10) circle (3pt);
			\filldraw (17,-10) circle (3pt);
			\filldraw (19,-10) circle (3pt);
			\filldraw (21,-10) circle (3pt);
			\filldraw (27,-10) circle (3pt);
			\filldraw (25,-10) circle (3pt);
			\draw[thick] (14,-3) -- (26,-3);
			\draw (15,-9.6) -- (15,-10);
			\draw (27,-10) -- (27,-9.6);
			\draw (25,-9.6) -- (25,-10);
			\draw (21,-10) -- (21,-9.6);
			\draw (19,-9.6) -- (19,-10);
			\draw (17,-10) -- (17,-9.6);
			
			\draw[blue] (-9,-20) .. controls (-8,-21) and (-6,-21) .. (-5,-20);
			\draw[blue] (-5,-20) .. controls (-3,-18) and (-1,-18) .. (1,-20);
			\draw[red] (6,-20) .. controls (7,-19) and (8,-19) .. (9,-20);
			\draw[red] (9,-20) .. controls (10,-21) and (11,-21) .. (11,-20);
			\draw[red] (11,-20) .. controls (11,-19) and (7,-15) .. (2,-20);
			\draw[white, line width=3pt] (2,-13) -- (2,-19.6);
			\draw (2,-13) -- (2,-19.6);
			\draw (0,-13) .. controls (0,-16) and (-9,-17) .. (-9,-19.6);
			\draw[white, line width=3pt] (4,-13) .. controls (4,-16) and (-5,-17) .. (-5,-19.6);
			\draw (4,-13) .. controls (4,-16) and (-5,-17) .. (-5,-19.6);
			\draw[white, line width=3pt] (10,-13) .. controls (10,-16) and (1,-17) .. (1,-19.6);
			\draw (10,-13) .. controls (10,-16) and (1,-17) .. (1,-19.6);
			\draw[white, line width=3pt] (8,-13) .. controls (8,-16) and (12,-17) .. (12,-19.6);
			\draw (8,-13) .. controls (8,-16) and (12,-17) .. (12,-19.6);
			\draw[white, line width=3pt] (6,-13) -- (6,-19.6);
			\draw (6,-13) -- (6,-19.6);
			\filldraw (-9,-20) circle (3pt);
			\filldraw (2,-20) circle (3pt);
			\filldraw (-5,-20) circle (3pt);
			\filldraw (6,-20) circle (3pt);
			\filldraw (12,-20) circle (3pt);
			\filldraw (1,-20) circle (3pt);
			\draw[thick] (-1,-13) -- (11,-13);
			\draw (-9,-19.6) -- (-9,-20);
			\draw (12,-20) -- (12,-19.6);
			\draw (1,-19.6) -- (1,-20);
			\draw (6,-20) -- (6,-19.6);
			\draw (-5,-19.6) -- (-5,-20);
			\draw (2,-20) -- (2,-19.6);
			
			\draw[blue] (15,-20) -- (19,-20);
			\draw[red] (21,-20) -- (23,-20);	
			\draw[white, line width=3pt]	 (17,-13) .. controls (17,-15) and (21,-18) .. (21,-19.6);
			\draw (17,-13) .. controls (17,-15) and (21,-18) .. (21,-19.6);
			\draw (15,-13) .. controls (15,-15) and (15,-18) .. (15,-19.6);
			\draw[white, line width=3pt] (19,-13) .. controls (19,-15) and (17,-18) .. (17,-19.6);
			\draw (19,-13) .. controls (19,-15) and (17,-18) .. (17,-19.6);
			\draw[white, line width=3pt] (25,-13) .. controls (25,-15) and (19,-18) .. (19,-19.6);
			\draw (25,-13) .. controls (25,-15) and (19,-18) .. (19,-19.6);
			\draw[white, line width=3pt] (23,-13) .. controls (23,-15) and (25,-18) .. (25,-19.6);
			\draw (23,-13) .. controls (23,-15) and (25,-18) .. (25,-19.6);
			\draw[white, line width=3pt]	 (21,-13) .. controls (21,-15) and (23,-18) .. (23,-19.6);
			\draw (21,-13) .. controls (21,-15) and (23,-18) .. (23,-19.6);
			\filldraw (15,-20) circle (3pt);
			\filldraw (21,-20) circle (3pt);
			\filldraw (17,-20) circle (3pt);
			\filldraw (23,-20) circle (3pt);
			\filldraw (25,-20) circle (3pt);
			\filldraw (19,-20) circle (3pt);
			\draw[thick] (14,-13) -- (26,-13);
			\draw (15,-19.6) -- (15,-20);
			\draw (25,-20) -- (25,-19.6);
			\draw (19,-19.6) -- (19,-20);
			\draw (23,-20) -- (23,-19.6);
			\draw (17,-19.6) -- (17,-20);
			\draw (21,-20) -- (21,-19.6);
		\end{tikzpicture}
	\end{center}
	}{TIKZ}	
	Replacing the archetypal arcs by the identifier operations of the corresponding archetype yields 
	a representative of $\mathrm{weave}(\calA)$ and the class modulo dangling does not depend on the weaving process.
	Thus, we get a well-defined map on the objects of $\mathcal{AC}_2(C,A;T)$. It also respects
	the poset structures, so it is a map of posets.
\end{proof}

It follows from the finiteness of $VE$ that the set of archetypes $A$ is of finite type (though not finite in general)
and from the color-tameness of $\calO$ that it is tame. More precisely, let $m_V$ be the largest degree of very elementary classes and 
$m_C$ be the smallest natural number greater than the degree of any reduced object in $\calS$. Then we can set 
$m_a=m_V$ and $m_r=m_C$ in Theorem \ref{85504}. We thus get the following
\begin{cor}\label{13058}
	$\op{Core}(\calK)$ is $\nu_d(\op{deg}\calK)$-connected where
	\begin{gather*}
		\nu_1(l):=\left\lfloor\frac{l-m_C}{2m_V+m_C-2}\right\rfloor-1\\
		\nu_2(l):=\left\lfloor\frac{l-m_C}{2m_V-1}\right\rfloor-1\\
		\nu_3(l):=\left\lfloor\frac{l-m_C}{2m_V-1}\right\rfloor-1
	\end{gather*}
	Here, $d=1$ corresponds to the planar case, $d=2$ to the braided case and $d=3$ to the symmetric case.
\end{cor}

\subsubsection{The corona}

We build up $\op{Corona}(\calK)$ from $\op{Core}(\calK)$ using again the Morse method for categories. We then get a
connectivity result for the corona from the connectivity result for the core. The idea is attributed to \cite{f-m-w-z:tbg}.

We assumed $\calO$ to be of finite type, i.e.~the set of elementary classes $E$ is finite. Let $m_E$ be the largest degree of 
elementary classes. An object in $\op{Corona}(\calK)$ is a pair $(\calY,\alpha\co\calK\rightarrow\calY)$ where 
$\op{deg}(\calY)<\op{deg}(\calK)$ and $\alpha$ is an elementary arrow in $\calU/\calG$. For $2\leq k\leq m_E$ denote by 
$\#_{se}^k(\alpha)$ the number of strictly elementary operations of degree $k$ in any representative of $\alpha$. Define
\[f\big((\calY,\alpha)\big):=\big(\#_{se}^{m_E}(\alpha),\#_{se}^{m_E-1}(\alpha),\ldots,\#_{se}^2(\alpha),\deg(\calY)\big)\]
Order the values of $f$ lexicographically. Then $f$ becomes a Morse function building up $\op{Corona}(\calK)$ from
$\op{Core}(\calK)$. Define
\begin{gather*}
	\mu_1(l):=\left\lfloor\frac{l-m_C}{2m_V+m_C+m_E}\right\rfloor-2\\
	\mu_2(l):=\left\lfloor\frac{l-m_C}{2m_V+m_E}\right\rfloor-1\\
	\mu_3(l):=\left\lfloor\frac{l-m_C}{2m_V+m_E}\right\rfloor-1\\
\end{gather*}

\begin{prop}\label{81197}
	For each object $(\calY,\alpha)$ in $\op{Corona}(\calK)$ which is not an object in $\op{Core}(\calK)$, the descending link 
	$lk_\downa(\calY,\alpha)$ with respect to the Morse function $f$ above is $\mu_d(\deg\calK)$-connected.
\end{prop}

From Theorem \ref{34332} we get that $\op{Core}(\calK)$ and $\op{Corona}(\calK)$ share the same homotopy groups up to
dimension $\mu_d(\deg\calK)$. We already know that $\op{Core}(\calK)$ is $\nu_d(\deg\calK)$-connected. Furthermore, we have
$\nu_d(l)\geq\mu_d(l)$. Consequently, we get the following

\begin{cor}\label{64879}
	$\op{Corona}(\calK)$ is $\mu_d(\deg\calK)$-connected. In particular, its connectivity tends to infinity as 
	$\deg(\calK)\rightarrow\infty$.
\end{cor}

In the rest of this subsubsection, we give a proof of Proposition \ref{81197}. We distinguish between two sorts of objects
$(\calY,\alpha)$ in $\op{Corona}(\calK)$ which are not objects in $\op{Core}(\calK)$: Such an object is called
\emph{mixed} if there is at least one very elementary operation in $\alpha$. It is called \emph{pure} if there is no 
very elementary operation in $\alpha$.

\begin{lem}\label{65633}
	Let $(\calY,\alpha)$ be mixed. Then $\overline{lk}_\downa(\calY,\alpha)$ and therefore $lk_\downa(\calY,\alpha)$ is
	contractible. In particular, Proposition \ref{81197} is true for mixed objects.
\end{lem}
\begin{proof}
	The data of an object in $\overline{lk}_\downa(\calY,\alpha)$ is $\Omega=\big((\calL,\beta_1),\beta_2\big)$ where $\calL$
	is an object in $\calU/\calG$, $\beta_1$ is an elementary arrow in $\calU/\calG$, $\beta_2$ is an arrow in
	$\calU/\calG$ such that $\beta_1\beta_2=\alpha$ and $(\calL,\beta_1)$ forms an object in $\op{Corona}(\calK)$
	of strictly smaller Morse height than $(\calY,\alpha)$. Let $\Omega'=\big((\calL',\beta'_1),\beta'_2\big)$ be another
	such object. An arrow $\Omega\rightarrow\Omega'$ is represented by an arrow $\delta\co\calL\rightarrow\calL'$ such that
	$\beta_1\delta=\beta'_1$ and $\delta\beta'_2=\beta_2$.
	
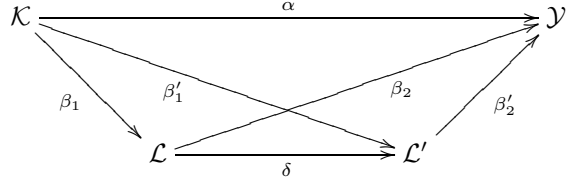
\begin{displaymath}\xymatrix@-8pt{
		\calK\ar[rrrrrrrr]^{\alpha}\ar[rrdd]_{\beta_1}\ar[rrrrrrdd]_/-15pt/{\beta'_1}&&&&&&&&\calY\\
		&&&&&&&&\\
		&&\calL\ar[uurrrrrr]_/15pt/{\beta_2}\ar[rrrr]_{\delta}&&&&\calL'\ar[uurr]_{\beta'_2}&&
	}\end{displaymath}
	It follows that $\overline{lk}_\downa(\calY,\alpha)$ is a poset since $\calU/\calG$ is a poset.
	
	Choose representatives $K$ and $Y$ of $\calK$ and $\calY$. Then $\alpha$ is represented by a unique arrow $a\co K\rightarrow Y$. 
	We can choose $K$ such that $a$ is a tensor product of higher degree operations and identities. Let $a^v\co K\rightarrow Y^v$ be the 
	arrow obtained from $a$ by replacing all strictly elementary operations $\theta$ with $\deg(\theta)$ identity operations. 
	Let $a^{se}\co Y^v\rightarrow Y$ be the arrow obtained from $a$ by replacing all very elementary operations by one identity operation each.
	We have $a^va^{se}=a$.
	An example of $a,a^v,a^{se}$ is pictured below. There, a white triangle is a placeholder for a strictly elementary operation.
	A black triangle indicates a very elementary operation. A straight horizontal line represents an identity operation.
	\ifthenelse{\boolean{drawtikz}}
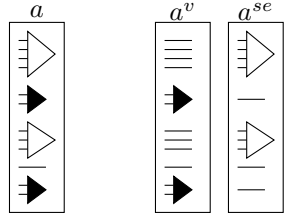
{
	\begin{center}
		\begin{tikzpicture}[scale=0.12]
		\node at (-20,5.2) {$a$};
		\draw (-18,0.5) -- (-21,3) -- (-21,-2) -- cycle;
		\draw (-21,2) -- (-22,2);
		\draw (-21,1) -- (-22,1);
		\draw (-21,0) -- (-22,0);
		\draw (-21,-1) -- (-22,-1);
		\filldraw[fill=black] (-21,-3) -- (-19,-4.5) -- (-21,-6) -- cycle;
		\draw (-21,-4) -- (-22,-4);
		\draw (-21,-5) -- (-22,-5);
		\draw (-21,-7) -- (-18,-9) -- (-21,-11) -- cycle;
		\draw (-21,-8) -- (-22,-8);
		\draw (-21,-9) -- (-22,-9);
		\draw (-21,-10) -- (-22,-10);
		\filldraw[fill=black] (-21,-13) -- (-19,-14.5) -- (-21,-16) -- cycle;
		\draw (-21,-14) -- (-22,-14);
		\draw (-21,-15) -- (-22,-15);
		\draw  (-23,4) rectangle (-17,-17);
		\node at (-4,5.2) {$a^v$};
		\draw (-3,2) -- (-6,2);
		\draw (-3,1) -- (-6,1);
		\draw (-3,0) -- (-6,0);
		\draw (-3,-1) -- (-6,-1);
		\filldraw[fill=black] (-5,-3) -- (-3,-4.5) -- (-5,-6) -- cycle;
		\draw (-5,-4) -- (-6,-4);
		\draw (-5,-5) -- (-6,-5);
		\draw (-3,-8) -- (-6,-8);
		\draw (-3,-9) -- (-6,-9);
		\draw (-3,-10) -- (-6,-10);
		\filldraw[fill=black] (-5,-13) -- (-3,-14.5) -- (-5,-16) -- cycle;
		\draw (-5,-14) -- (-6,-14);
		\draw (-5,-15) -- (-6,-15);
		\draw  (-7,4) rectangle (-1,-17);
		\node at (4,5.2) {$a^{se}$};
		\draw (6,0.5) -- (3,3) -- (3,-2) -- cycle;
		\draw (3,2) -- (2,2);
		\draw (3,1) -- (2,1);
		\draw (3,0) -- (2,0);
		\draw (3,-1) -- (2,-1);
		\draw (5,-4.5) -- (2,-4.5);
		\draw (3,-7) -- (6,-9) -- (3,-11) -- cycle;
		\draw (3,-8) -- (2,-8);
		\draw (3,-9) -- (2,-9);
		\draw (3,-10) -- (2,-10);
		\draw (5,-14.5) -- (2,-14.5);
		\draw  (1,4) rectangle (7,-17);
		\draw (-22,-12) -- (-19,-12);
		\draw (-6,-12) -- (-3,-12);
		\draw (2,-12) -- (5,-12);
		\end{tikzpicture}
	\end{center}	
	}{TIKZ}
	Set $\calY^v:=[Y^v]$ and $\alpha^v:=[a^v]$ as well as $\alpha^{se}:=[a^{se}]$. Then $(\calY^v,\alpha^v)$ is an object in 
	$\op{Core}(\calK)$ and $\alpha^{se}$ represents an arrow $(\calY^v,\alpha^v)\rightarrow(\calY,\alpha)$ in $\op{Corona}(\calK)$.
	Moreover, the pair $\Xi:=\big((\calY^v,\alpha^v),\alpha^{se}\big)$ is an object in $\overline{lk}_\downa(\calY,\alpha)$.
	
	Let $\Omega=\big((\calL,\beta_1),\beta_2\big)$ be an object in $\overline{lk}_\downa(\calY,\alpha)$.
	We define another object $F(\Omega)=\big((\calM,\gamma_1),\gamma_2\big)$ of $\overline{lk}_\downa(\calY,\alpha)$ 
	as follows: Choose a representative $L$ of $\calL$ such that $\beta_2$ is represented by $b_2\co L\rightarrow Y$ which is
	a tensor product of identities and higher degree operations. Then $\beta_1$ is represented by a unique $b_1\co K\rightarrow L$. Note that 
	$b_1b_2=a$. Think of $b_1$ as splitting higher degree operations of $a$ into operations of smaller degree and of $b_2$ as merging
	them back to their original form. Now define the arrows $g_1\co K\rightarrow M$ and $g_2\co M\rightarrow Y$ to be the same splitting 
	of $a$ with the only exception that no very elementary operation of $a$ is splitted. An example fitting to the example above is pictured 
	below. There, a gray triangle is a placeholder for an elementary operation or a degree $1$ operation, a blue triangle can be any operation 
	and a dot on a straight horizontal line indicates a possibly non-trivial degree $1$ operation.
	\ifthenelse{\boolean{drawtikz}}
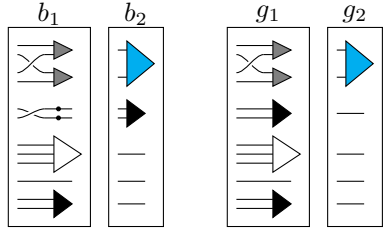
{
	\begin{center}
		\begin{tikzpicture}[scale=0.12]
		\node at (-35.5,6.5) {$b_1$};
		\filldraw[fill=gray] (-33,-0.5) -- (-35,0.5) -- (-35,-1.5) -- cycle;
		\draw (-35,3) -- (-39,3);
		\draw (-35,2) -- (-36,2);
		\draw (-35,0) -- (-36,0);
		\draw (-35,-1) -- (-39,-1);
		\draw (-33,-4) -- (-36,-4);
		\draw (-33,-5) -- (-36,-5);
		\draw (-35,-7) -- (-32,-9) -- (-35,-11) -- cycle;
		\draw (-35,-8) -- (-39,-8);
		\draw (-35,-9) -- (-39,-9);
		\draw (-35,-10) -- (-39,-10);
		\filldraw[fill=black] (-35,-13) -- (-33,-14.5) -- (-35,-16) -- cycle;
		\draw (-35,-14) -- (-39,-14);
		\draw (-35,-15) -- (-39,-15);
		\filldraw[fill=gray] (-33,2.5) -- (-35,3.5) -- (-35,1.5) -- cycle;
		\node at (-26,6.5) {$b_2$};
		\draw (-27,2.5) -- (-28,2.5);
		\draw (-27,-0.5) -- (-28,-0.5);
		\draw (-27,-4) -- (-28,-4);
		\draw (-27,-5) -- (-28,-5);
		\draw (-25,-9) -- (-28,-9);
		\draw (-25,-14.5) -- (-28,-14.5);
		\filldraw[fill=black] (-27,-3) -- (-27,-6) -- (-25,-4.5) -- cycle;
		\filldraw[fill=cyan] (-27,3.5) -- (-27,-1.5) -- (-24,1) -- cycle;
		\node at (-11.5,6.5) {$g_1$};
		\filldraw[fill=gray] (-9,-0.5) -- (-11,0.5) -- (-11,-1.5) -- cycle;
		\draw (-11,3) -- (-15,3);
		\draw (-11,2) -- (-12,2);
		\draw (-11,0) -- (-12,0);
		\draw (-11,-1) -- (-15,-1);
		\draw (-11,-4) -- (-15,-4);
		\draw (-11,-5) -- (-15,-5);
		\draw (-11,-7) -- (-8,-9) -- (-11,-11) -- cycle;
		\draw (-11,-8) -- (-15,-8);
		\draw (-11,-9) -- (-15,-9);
		\draw (-11,-10) -- (-15,-10);
		\filldraw[fill=black] (-11,-13) -- (-9,-14.5) -- (-11,-16) -- cycle;
		\draw (-11,-14) -- (-15,-14);
		\draw (-11,-15) -- (-15,-15);
		\filldraw[fill=gray] (-9,2.5) -- (-11,3.5) -- (-11,1.5) -- cycle;
		\filldraw[fill=black] (-11,-3) -- (-11,-6) -- (-9,-4.5) -- cycle;
		\node at (-2,6.5) {$g_2$};
		\draw (-3,2.5) -- (-4,2.5);
		\draw (-3,-0.5) -- (-4,-0.5);
		\draw (-1,-4.5) -- (-4,-4.5);
		\draw (-1,-9) -- (-4,-9);
		\draw (-1,-14.5) -- (-4,-14.5);	
		\filldraw[fill=cyan] (-3,3.5) -- (-3,-1.5) -- (0,1) -- cycle;
		\draw  (-40,5) rectangle (-31,-17);
		\draw  (-29,5) rectangle (-23,-17);
		\draw  (-16,5) rectangle (-7,-17);
		\draw  (-5,5) rectangle (1,-17);
		\draw (-39,-12) -- (-33,-12);
		\draw (-28,-12) -- (-25,-12);
		\draw (-15,-12) -- (-9,-12);
		\draw (-4,-12) -- (-1,-12);
		\filldraw (-34.5,-4) circle (6pt);
		\filldraw (-34.5,-5) circle (6pt);
		\draw (-12,0) .. controls (-13,0) and (-14,2) .. (-15,2);
		\draw[white, line width=2pt] (-12,2) .. controls (-13,2) and (-14,0) .. (-15,0);
		\draw (-12,2) .. controls (-13,2) and (-14,0) .. (-15,0);
		\draw (-36,-5) .. controls (-37,-5) and (-38,-4) .. (-39,-4);
		\draw[white, line width=2pt] (-36,-4) .. controls (-37,-4) and (-38,-5) .. (-39,-5);
		\draw (-36,-4) .. controls (-37,-4) and (-38,-5) .. (-39,-5);
		\draw (-36,0) .. controls (-37,0) and (-38,2) .. (-39,2);
		\draw[white, line width=2pt] (-36,2) .. controls (-37,2) and (-38,0) .. (-39,0);
		\draw (-36,2) .. controls (-37,2) and (-38,0) .. (-39,0);
		\end{tikzpicture}
	\end{center}
	}{TIKZ}
	Now set $\calM:=[M]$ and $\gamma_1:=[g_1]$ as well as $\gamma_2=[g_2]$.

	It is not hard to see that $\Omega\mapsto F(\Omega)$ extends to a functor 
	$\overline{lk}_\downa(\calY,\alpha)\rightarrow\overline{lk}_\downa(\calY,\alpha)$ which means that whenever we have 
	an arrow $\delta\co\Omega\rightarrow\Omega'$, then there is an arrow $\Delta\co F(\Omega)\rightarrow F(\Omega')$.
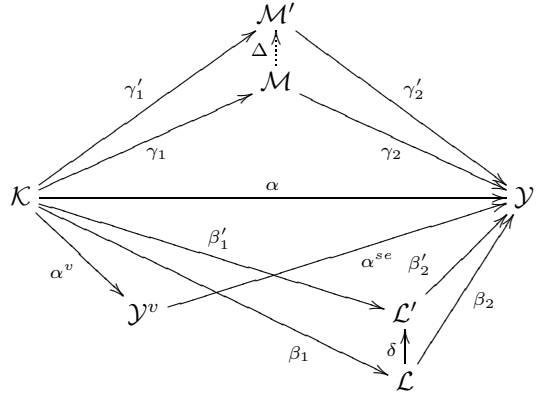
	\begin{displaymath}\xymatrix@-12pt{
		&&&& \calM'\ar[rrrrddd]^{\gamma'_2}&&&&\\
		&&&& \calM\ar[rrrrdd]_{\gamma_2}\ar@{.>}[u]^{\Delta}&&&&\\
		&&&&&&&&\\	
		\calK\ar[uurrrr]_{\gamma_1}\ar[uuurrrr]^{\gamma'_1}\ar[rrrrrrrr]^{\alpha}\ar[rrdd]_{\alpha^v}
		\ar[rrrrrrdd]^{\beta'_1}\ar[rrrrrrddd]_/40pt/{\beta_1}&&&&&&&&\calY\\
		&&&&&&&&\\
		&&\calY^v\ar[uurrrrrr]_/14pt/{\alpha^{se}}&&&& \calL'\ar[uurr]^/-14pt/{\beta'_2}&&\\
		&&&&&& \calL\ar[uuurr]_{\beta_2}\ar[u]^{\delta} &&
	}\end{displaymath}
	We also have arrows $\xi_\Omega\co \Xi\rightarrow F(\Omega)$ and $\iota_\Omega\co \Omega\rightarrow F(\Omega)$.
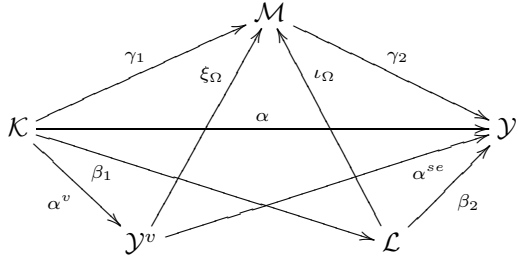
	\begin{displaymath}\xymatrix@-12pt{
		&&&& \calM\ar[rrrrdd]^{\gamma_2}&&&&\\
		&&&&&&&&\\
		\calK\ar[rrrrrrrr]^{\alpha}\ar[rrdd]_{\alpha^v}
		\ar[rrrrrrdd]_/-38pt/{\beta_1}\ar[uurrrr]^{\gamma_1}&&&&&&&&\calY\\
		&&&&&&&&\\
		&&\calY^v\ar[uurrrrrr]_/38pt/{\alpha^{se}}
		\ar[uuuurr]^/20pt/{\xi_\Omega}&&&&
		\calL\ar[uurr]_{\beta_2}\ar[uuuull]_/20pt/{\iota_\Omega}&&
	}\end{displaymath}
	The arrow $\xi_\Omega$ is represented by an arrow $x_\Omega\co Y^v\rightarrow M$ which satisfies $a^v x_\Omega=g_1$ and 
	$x_\Omega g_2=a^{se}$. The arrow $\iota_\Omega$ is represented by $i_\Omega\co L\rightarrow M$ which satisfies $b_1 i_\Omega=g_1$ 
	and $i_\Omega g_2=b_2$. In the example from above, these arrows look as follows:
	\ifthenelse{\boolean{drawtikz}}
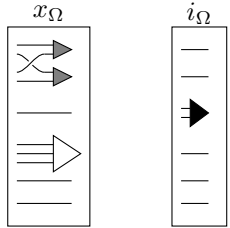
{
	\begin{center}
		\begin{tikzpicture}[scale=0.12]
		\node at (-16.5,6.5) {$x_\Omega$};
		\filldraw[fill=gray] (-14,-0.5) -- (-16,0.5) -- (-16,-1.5) -- cycle;
		\draw (-16,3) -- (-20,3);
		\draw (-16,2) -- (-17,2);
		\draw (-16,0) -- (-17,0);
		\draw (-16,-1) -- (-20,-1);
		\draw (-14,-4.5) -- (-20,-4.5);	
		\draw (-16,-7) -- (-13,-9) -- (-16,-11) -- cycle;
		\draw (-16,-8) -- (-20,-8);
		\draw (-16,-9) -- (-20,-9);
		\draw (-16,-10) -- (-20,-10);
		\draw (-14,-14.5) -- (-20,-14.5);
		\filldraw[fill=gray] (-14,2.5) -- (-16,3.5) -- (-16,1.5) -- cycle;
		\node at (0,6.5) {$i_\Omega$};	
		\draw (1,2.5) -- (-2,2.5);	
		\draw (1,-0.5) -- (-2,-0.5);	
		\draw (-1,-4) -- (-2,-4);
		\draw (-1,-5) -- (-2,-5);	
		\draw (1,-9) -- (-2,-9);	
		\draw (1,-14.5) -- (-2,-14.5);		
		\filldraw[fill=black] (-1,-3) -- (-1,-6) -- (1,-4.5) -- cycle;
		\draw  (-21,5) rectangle (-12,-17);
		\draw  (-3,5) rectangle (3,-17);
		\draw (-20,-12) -- (-14,-12);
		\draw (-2,-12) -- (1,-12);
		\draw (-17,0) .. controls (-18,0) and (-19,2) .. (-20,2);
		\draw[white, line width=2pt] (-20,0) .. controls (-19,0) and (-18,2) .. (-17,2);
		\draw (-20,0) .. controls (-19,0) and (-18,2) .. (-17,2);
		\end{tikzpicture}
	\end{center}
	}{TIKZ}
	The claim of the proposition now follows from item iii) in Subsection \ref{03887} applied to the functor $F$ and the 
	object $\Xi$.
\end{proof}

\begin{lem}\label{15097}
	Let $(\calY,\alpha)$ be pure. Then $lk_\downa(\calY,\alpha)$ is $\mu_d(\deg\calK)$-connected and Proposition \ref{81197} 
	is true for pure objects.
\end{lem}
\begin{proof}
	Choose representatives $K$ and $Y$ of $\calK$ and $\calY$ such that $a\co K\rightarrow Y$ representing $\alpha$ is a tensor 
	product of higher degree operations and identities.
	
	First observe the descending up link $\overline{lk}_\downa(\calY,\alpha)$. An object in $\overline{lk}_\downa(\calY,\alpha)$ 
	is a pair $\big((\calL,\beta_1),\beta_2\big)$ with $f(\calL,\beta_1)<f(\calY,\alpha)$ and $\beta_1\beta_2=\alpha$. When
	choosing a representative $L$ of $\calL$, we get unique representatives $b_1\co K\rightarrow L$ of $\beta_1$ and 
	$b_2\co L\rightarrow Y$ of $\beta_2$ such that $b_1b_2=a$. As in the proof of the previous lemma, $b_1$ can be interpreted 
	as splitting higher degree operations of $a$ into operations of smaller degree and $b_2$ as merging them back to their original
	form. Denote by $\calA_i$ the full subcategory of $\overline{lk}_\downa(\calY,\alpha)$ spanned by the objects which only split 
	the $i$'th higher degree operation in $a$. Denote by $\mathfrak{n}$ the number of higher degree operations in $a$. Observe now
	that when splitting operations in $a$ one by one, then we can also split all that operations at once. This observation reveals that
	\[\overline{lk}_\downa(\calY,\alpha)=\calA_1\circ\ldots\circ\calA_{\mathfrak{n}}\]
	is the Grothendieck join of the $\calA_i$ (see Subsection \ref{06663}). Note that the categories 
	$\calA_i$ are all non-empty since all the higher degree operations in $a$ are elementary but not very elementary and
	splitting such a strictly elementary operation decreases the Morse height. Thus, $\overline{lk}_\downa(\calY,\alpha)$
	is $(\mathfrak{n}-2)$-connected.
	
	Now look at the descending down link $\underline{lk}_\downa(\calY,\alpha)$. Objects are pairs
	$\big((\calL,\beta_1),\beta_2\big)$ with $f(\calL,\beta_1)<f(\calY,\alpha)$ and $\alpha\beta_2=\beta_1$. 
	When choosing a representative $L$ of $\calL$, we get unique representatives $b_1\co K\rightarrow L$ of $\beta_1$ and
	$b_2\co Y\rightarrow L$ of $\beta_2$ such that $a b_2=b_1$. Looking at the Morse function $f$ for the corona, 
	one sees that the higher degree operations of $b_2$ must be very elementary operations which only compose with 
	identity operations of $a$. At this point, we have to distinguish between the planar case on the one hand and the
	braided resp.~symmetric case on the other.
	
	We start with the braided resp.~symmetric case: The arguments in the proof of Proposition \ref{19549} reveal
	that $\underline{lk}_\downa(\calY,\alpha)$ is isomorphic to $\mathcal{AC}_d(C,A;T')$ where $T'$ is the color
	word obtained from the codomain of $a$ (viewed as an arrow in $\calS$) after deleting the higher degree operations. 
	Denote by $\mathfrak{l}$ the length of $T'$, i.e.~the number of identity operations in $a$. 
	Then we already know that $\mathcal{AC}_d(C,A;T')$
	is $\nu_d(\mathfrak{l})$-connected (compare with Corollary \ref{13058}). Consequently, the connectivity of the 
	descending link $lk_\downa(\calY,\alpha)=\overline{lk}_\downa(\calY,\alpha)*\underline{lk}_\downa(\calY,\alpha)$ is
	\begin{align*}
		\mathfrak{n}+\nu_d(\mathfrak{l}) &= \mathfrak{n}+\left\lfloor\frac{\mathfrak{l}-m_C}{2m_V-1}\right\rfloor-1\\
			&\geq \mathfrak{n}+\left\lfloor\frac{\deg\calK-\mathfrak{n}m_E-m_C}{2m_V-1}\right\rfloor-1\\
			&\geq \mathfrak{n}+\left\lfloor\frac{\deg\calK-\mathfrak{n}m_E-m_C}{2m_V+m_E}\right\rfloor-1\\
			&= \left\lfloor\frac{\deg\calK-m_C+2m_V\mathfrak{n}}{2m_V+m_E}\right\rfloor-1\\
			&\geq \left\lfloor\frac{\deg\calK-m_C}{2m_V+m_E}\right\rfloor-1\\
			&= \mu_d(\deg\calK)
	\end{align*}
	where we have used that $\mathfrak{n}m_E+\mathfrak{l}\geq \deg\calK$.	
	
	Now we turn to the planar case: An identity component in $a$ is a maximal subsequence of identity operations. 
	Let $\mathfrak{m}$ be the number of identity components and denote by $\mathfrak{l}_i$ for $i=1,\ldots,\mathfrak{m}$ 
	the length of the $i$'th identity component. Denote by $\mathfrak{l}$ the total number of identity operations in $a$,
	i.e.~the sum of the $\mathfrak{l}_i$. Define $\calB_i$ to be the full subcategory of $\underline{lk}_\downa(\calY,\alpha)$
	spanned by the objects which only add very elementary operations into the $i$'th identity component. Observe now that when 
	adding very elementary operations into different identity components one by one, then we can also add all that operations 
	at once. This reveals that $\underline{lk}_\downa(\calY,\alpha)$ is the Grothendieck join of the $\calB_i$. Note, though,
	when inspecting the direction of the arrows in $\underline{lk}_\downa(\calY,\alpha)$, one sees that it is in fact
	the dual Grothendieck join. So we have
	\[\underline{lk}_\downa(\calY,\alpha)=\calB_1\bullet\ldots\bullet\calB_{\mathfrak{m}}\]
	Similarly as in the braided resp.~symmetric case, $\calB_i$ is isomorphic to $\mathcal{AC}_1(C,A;T_i)$ where $T_i$ is the
	color word obtained from the codomain of $a$ after deleting all operations except the identity operations of the $i$'th
	identity component. The length of $T_i$ is $\mathfrak{l}_i$. Then we already know that $\mathcal{AC}_1(C,A;T_i)$ is
	$\nu_1(\mathfrak{l}_i)$-connected. Therefore, the connectivity of $\underline{lk}_\downa(\calY,\alpha)$ is at least
	\[2\mathfrak{m}-2+\sum_{j=1}^\mathfrak{m}\nu_1(\mathfrak{l}_j)\]
	Thus, the connectivity of $lk_\downa(\calY,\alpha)$ is at least
	\begin{align*}
		\mathfrak{n}+2\mathfrak{m}-2+\sum_{j=1}^\mathfrak{m}\nu_1(\mathfrak{l}_j) &\geq
		\mathfrak{n}+\mathfrak{m}-2+\sum_{j=1}^\mathfrak{m}\left\lfloor\frac{\mathfrak{l}_j-m_C}{2m_V+m_C}\right\rfloor\\
		&\geq \mathfrak{n}-2+\sum_{j=1}^\mathfrak{m}\frac{\mathfrak{l}_j-m_C}{2m_V+m_C}\\
		&= \mathfrak{n}-2+\frac{\mathfrak{l}-\mathfrak{m}m_C}{2m_V+m_C}\\
		&\geq \mathfrak{n}-2+\frac{\deg\calK-\mathfrak{n}m_E-(\mathfrak{n}+1)m_C}{2m_V+m_C}\\
		&\geq \mathfrak{n}+\frac{\deg\calK-\mathfrak{n}m_E-\mathfrak{n}m_C-m_C}{2m_V+m_C+m_E}-2\\
		&= \frac{\deg\calK-m_C+2m_V\mathfrak{n}}{2m_V+m_C+m_E}-2\\
		&\geq \frac{\deg\calK-m_C}{2m_V+m_C+m_E}-2\\
		&\geq \mu_1(\deg\calK)
	\end{align*}
	where we have used in the fourth step that $\mathfrak{m}\leq\mathfrak{n}+1$ and $\mathfrak{n}m_E+\mathfrak{l}\geq\deg\calK$.
\end{proof}

\subsubsection{The whole link}

In this last step, we show that the inclusion 
\[\op{Corona}(\calK)\subset\underline{lk}_\downa(\calK)\]
is a homotopy equivalence. It then follows from Corollary \ref{64879} that the connectivity of $\underline{lk}_\downa(\calK)$ tends to 
infinity as $\deg(\calK)\rightarrow\infty$ which is what we wanted to show in order to finish the proof of Theorem \ref{41762}.
This step is analogous to the reduction to the Stein space of elementary intervals 
in \cite{ste:gop}. We again apply the Morse method for categories to build $\underline{lk}_\downa(\calK)$ up from $\op{Corona}(\calK)$. 
The Morse function on objects of $\underline{lk}_\downa(\calK)$ which do not lie in $\op{Corona}(\calK)$ is given by
\[f\big((\calY,\alpha)\big):=-\deg(\calY)\]
We have $\underline{lk}_\downa(\calY,\alpha)=\emptyset$ and thus $lk_\downa(\calY,\alpha)=\overline{lk}_\downa(\calY,\alpha)$ with respect to 
this Morse function. Similarly as in the proofs of Lemmas \ref{65633} and \ref{15097}, we obtain
\[\overline{lk}_\downa(\calY,\alpha)=\calA_1\circ\ldots\circ\calA_{\mathfrak{n}}\]
where the $\calA_i$ are full subcategories of $\overline{lk}_\downa(\calY,\alpha)$ spanned by the objects which correspond to splitting
exactly one of the $\mathfrak{n}$ higher degree operations in a representative $a$ of $\alpha$. At least one of these
operations must be non-elementary since $(\calY,\alpha)$ is not an object in $\op{Corona}(\calK)$. Without loss of generality, 
assume that $\calA_1$ corresponds to such a non-elementary higher degree operation. If we show that $\calA_1$ is contractible, it 
follows that $\overline{lk}_\downa(\calY,\alpha)$ is contractible. Thus, we are building $\underline{lk}_\downa(\calK)$ up from 
$\op{Corona}(\calK)$ along contractible descending links and it follows from Theorem \ref{34332} that the inclusion 
$\op{Corona}(\calK)\subset\overline{lk}_\downa(\calK)$ is a homotopy equivalence. That $\calA_1$ is contractible follows from
Proposition \ref{29625} below.

First, we want to reinterprete the defining property of $E$ as the spine of $\mathcal{TC}^\ast(\calO)$ in terms of the category $\calU/\calG$.

\begin{lem}\label{90173}
	Let $\alpha\co\calK\rightarrow\calY$ be a non-elementary arrow in $\calU/\calG$ such that $\deg(\calK)=n>1$ and $\deg(\calY)=1$.
	Then there is a unique pair $(\alpha_1,\alpha_2)$ of arrows in $\calU/\calG$ (called the maximal elementary factorization of $\alpha$)
	such that $\alpha_2$ is elementary, $\alpha_1\alpha_2=\alpha$ and such that the following universal property is satisfied:
	Whenever $(\beta_1,\beta_2)$ is another pair with $\beta_2$ elementary and $\beta_1\beta_2=\alpha$ (called an elementary factorization
	of $\alpha$), then there is a unique arrow $\gamma$ with $\alpha_1\gamma=\beta_1$ and $\gamma\beta_2=\alpha_2$.
	
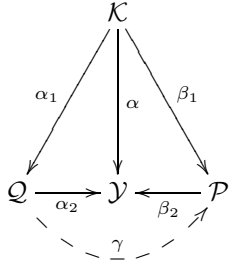
\begin{displaymath}\xymatrix{
		&\calK\ar[dd]^\alpha\ar[ldd]_{\alpha_1}\ar[ddr]^{\beta_1}&\\
		&&\\
		\calQ\ar[r]_{\alpha_2}\ar@/_25pt/@{-->}[rr]^\gamma &\calY &\calP\ar[l]^{\beta_2}
	}\end{displaymath}
\end{lem}
\begin{proof}
	Recall that $\calU/\calG$ is a poset. So there is at most one such $\gamma$. If also $(\beta_1,\beta_2)$ satisfies the universal property,
	then we must have $\calQ=\calP$ and consequently $\alpha_1=\beta_1$ as well as $\alpha_2=\beta_2$. This shows the uniqueness statements.
	
	Remains to prove the existence of such a pair: Choose representatives $K,Y$ of $\calK,\calY$. Then $\alpha$ is represented by a unique arrow 
	$a\co K\rightarrow Y$. Note that $a$ is just an operation since $\deg(Y)=1$. Denote the transformation class of $a$ by $\Omega$. Its degree is 
	$\deg(K)=n>1$ and it is non-elementary by assumption. Thus, by the definition of $E$ as the spine, there is a greatest elementary class $\Theta$ 
	with the property $\Theta<\Omega$. This implies that there is an operation $\theta\in\Theta$ and an arrow $q$ in $\calS$ such that $q*\theta=a$ in $\calS$.
	Define $Q:=K*q$ as an object in $\calU$ and further $\calQ:=[Q]$ as an object in $\calU/\calG$. The arrows $q\co K\rightarrow Q$ 
	resp.~$\theta\co Q\rightarrow Y$ in $\calU$ represent arrows $\alpha_1$ resp.~$\alpha_2$ in $\calU/\calG$ such that $\alpha_1\alpha_2=\alpha$ and 
	$\alpha_2$ is elementary.
	
	These two arrows satisfy the universal property: Let $b_1\co K\rightarrow P$ and $b_2\co P\rightarrow Y$ be representatives of 
	$\beta_1\co\calK\rightarrow\calP$ and $\beta_2\co\calP\rightarrow\calY$. Obviously, the transformation class $[b_2]$ of $b_2$ is elementary and satisfies 
	$[b_2]<[a]=\Omega$. Since $\Theta=[\theta]$ is the greatest such class, we obtain $[b_2]\leq[\theta]$. This means that there is an arrow $g$ in $\calS$ such 
	that $g*b_2=\theta$ in $\calS$. If $g$ is interpreted as an arrow $Q\rightarrow P$ in $\calU$, then it represents an arrow $\gamma\co\calQ\rightarrow\calY$
	in $\calU/\calG$ which satisfies $\gamma\beta_2=\alpha_2$. We then also have $\alpha_1\gamma=\beta_1$ since $\calU/\calG$ is a poset.
\end{proof}

We now turn to the announced proposition which concludes the proof of the main theorem.

\begin{prop}\label{29625}
	Let $\alpha\co\calK\rightarrow\calY$ be a non-elementary arrow in $\calU/\calG$ such that $\deg(\calK)=n>1$ and $\deg(\calY)=1$.
	Let $\calM$ be the full subcategory of $\calK\downa(\calU/\calG)_{n-1}$ spanned by the objects $(\calZ,\beta\co\calK\rightarrow\calZ)$
	with $\deg(\calZ)>1$ and
	\[\calL:=\calM\downa(\calY,\alpha)\]
	the descending up link of $(\calY,\alpha)$ with respect to the Morse function $f$ above. Then $\calL$ is contractible.
\end{prop}
\begin{proof}
	Note that the data of an object of $\calL$ is a non-trivial factorization of $\alpha$, i.e.~a pair $(\alpha_1,\alpha_2)$ of arrows
	in $\calU/\calG$ such that $\alpha_1\neq\id\neq \alpha_2$ and $\alpha_1\alpha_2=\alpha$. An arrow from $(\alpha_1,\alpha_2)$ to 
	$(\beta_1,\beta_2)$ is an arrow $\gamma$ such that $\alpha_1\gamma=\beta_1$ and $\gamma\beta_2=\alpha_2$. Clearly, $\calL$ is a poset.
	
	Apply Lemma \ref{90173} above to obtain a maximal elementary factorization $(\alpha_1,\alpha_2)$ of $\alpha$. Note that 
	$(\alpha_1,\alpha_2)$ is an object of $\calL$ and the universal property says that this object is initial among the objects 
	$(\beta_1,\beta_2)$ of $\calL$ with $\beta_2$ elementary.
	
	More generally, for an object $(\epsilon_1,\epsilon_2)$ of $\calL$ with $\epsilon_2$ non-elementary, we can apply the lemma to obtain
	a maximal elementary factorization $(\epsilon^*_1,\epsilon^*_2)$ of $\epsilon_2$. Then define $F(\epsilon_1,\epsilon_2):=
	(\epsilon_1\epsilon^*_1,\epsilon^*_2)$ which is again an object in $\calL$. If $\epsilon_2$ is already elementary, we set $\epsilon^*_1=\id$
	and $\epsilon^*_2=\epsilon_2$ so that $F(\epsilon_1,\epsilon_2)=(\epsilon_1,\epsilon_2)$.
	
	We claim that $F$ extends to a functor $\calL\rightarrow\calL$. So let $(\epsilon_1,\epsilon_2)$ and $(\beta_1,\beta_2)$ be two objects of 
	$\calL$ and $\gamma\co(\epsilon_1,\epsilon_2)\rightarrow(\beta_1,\beta_2)$ an arrow in $\calL$. We have to show that there is an arrow
	$\varphi\co F(\epsilon_1,\epsilon_2)\rightarrow F(\beta_1,\beta_2)$.
	\begin{displaymath}\xymatrix{
		&\bullet\ar[dd]^/-12pt/\alpha\ar[ld]_{\epsilon_1}\ar[dr]^{\beta_1} &\\
		\bullet\ar[rr]^/-12pt/\gamma\ar[dr]_{\epsilon_2}\ar[d]_{\epsilon^*_1} &&
		\bullet\ar[dl]^{\beta_2}\ar[d]^{\beta^*_1}\\
		\bullet\ar[r]_{\epsilon^*_2}\ar@/_25pt/@{-->}[rr]^\varphi &\bullet &\bullet\ar[l]^{\beta^*_2}
	}\end{displaymath}
	Observe first that if $\epsilon^*_1=\id$, then $\gamma\beta^*_1$ is an arrow $F(\epsilon_1,\epsilon_2)\rightarrow 
	F(\beta_1,\beta_2)$ as required. Else, observe that the pair $(\gamma\beta^*_1,\beta^*_2)$ is another elementary factorization of 
	$\epsilon_2$. Thus, by the universal property, we get a unique arrow $\varphi$ such that $\varphi\beta^*_2=\epsilon^*_2$ and $\epsilon^*_1\varphi=
	\gamma\beta^*_1$. This amounts to an arrow $F(\epsilon_1,\epsilon_2)\rightarrow F(\beta_1,\beta_2)$.
	
	Since $F(\epsilon_1,\epsilon_2)$ is an elementary factorization of $\alpha$, we get an arrow $(\alpha_1,\alpha_2)\rightarrow F(\epsilon_1,\epsilon_2)$ 
	for each object $(\epsilon_1,\epsilon_2)$ in $\calL$. Furthermore, $\epsilon^*_1$ clearly gives an arrow 
	$(\epsilon_1,\epsilon_2)\rightarrow F(\epsilon_1,\epsilon_2)$. The claim of the proposition now follows from item iii) in Subsection \ref{03887}
	applied to the functor $F$ and the object $(\alpha_1,\alpha_2)$.
\end{proof}

\subsection{Applications}

\subsubsection{Suboperads of endomorphism operads}\label{62712}

Consider the example with squares and triangles, the cube cutting operads (planar or symmetric) and the local similarity operads from Subsubsection \ref{42483}.
There, we have seen that they all satisfy the cancellative calculus of fractions. The squares and triangles operad and the cube cutting 
operads are of finite type. The local similarity operads are of finite type if and only if there are only finitely many $\mathrm{Sim}_X$-equivalence classes of 
balls, so we will assume this in the following. Then, in all three cases, the groupoid $\calI(\calO)$ is finite.

In order to apply Theorem \ref{41762}, it therefore remains to check color-tameness. The cube cutting operads are monochromatic, so color-tameness
is trivially satisfied here. The squares and triangles operad has two colors (the square and the triangle). It is easy to check that any sequence of at least
five squares and triangles is the domain of a very elementary arrow in $\calS(\calO)$. Consequently, it is color-tame as well. In general, a local similarity operad
is not color-tame.

As a special case, we obtain that the higher dimensional Thompson groups $nV$ are of type $F_\infty$. This has been shown before in \cite{f-m-w-z:tbg}.

The one dimensional cube cutting operads (planar or symmetric) with trivial groupoid of degree $1$ operations yield the groups of piecewise linear
homeomorphisms of the unit (Cantor) interval studied in \cite{ste:gop} and from the main theorem, it follows that they are of type $F_\infty$. This
has already been shown in \cite{ste:gop}.

The finiteness result for the local similarity groups has also been obtained in \cite{fa-hu:fpo}*{Theorem 6.5}. The hypothesis in this theorem consists 
of demanding that the finite similarity structure posseses only finitely many $\mathrm{Sim}_X$-equivalence classes of balls and of the property rich in
simple contractions which is implied by the easier to state property rich in ball contractions \cite{fa-hu:fpo}*{Definition 5.12}. It is not hard to see 
that the latter property exactly means that $\calO$, the local similarity operad associated to $\mathrm{Sim}_X$, is color-tame.

\subsubsection{Ribbon Thompson group}

The braided operad $\calO$ with transformations discussed in Subsubsection \ref{94422} satisfies the cancellative calculus of fractions.
It is monochromatic and therefore color-tame. There is only one very elementary transformation
class and thus, $\calO$ is of finite type. The groupoid $\calI(\calO)$ is the group $\ZZ$ which is of type $F^+_\infty$. The main theorem
yields that the Ribbon Thompson group $RV$ is of type $F_\infty$.


\begin{bibdiv}
\begin{biblist}

\bib{ben:sro}{article}{
   author={B{\'e}nabou, Jean},
   title={Some remarks on $2$-categorical algebra. I},
   journal={Bull. Soc. Math. Belg. S\'er. A},
   volume={41},
   date={1989},
   number={2},
   pages={127--194}
}

\bib{be-br:mta}{article}{
   author={Bestvina, Mladen},
   author={Brady, Noel},
   title={Morse theory and finiteness properties of groups},
   journal={Invent. Math.},
   volume={129},
   date={1997},
   number={3},
   pages={445--470}
}

\bib{bor:hoc}{book}{
   author={Borceux, Francis},
   title={Handbook of categorical algebra. 1},
   series={Encyclopedia of Mathematics and its Applications},
   volume={50},
   publisher={Cambridge University Press},
   place={Cambridge},
   date={1994}
}

\bib{bri:hdt}{article}{
   author={Brin, Matthew G.},
   title={Higher dimensional Thompson groups},
   journal={Geom. Dedicata},
   volume={108},
   date={2004},
   pages={163--192}
}

\bib{bri:tao}{article}{
   author={Brin, Matthew G.},
   title={The algebra of strand splitting. I. A braided version of
   Thompson's group $V$},
   journal={J. Group Theory},
   volume={10},
   date={2007},
   number={6},
   pages={757--788}
}

\bib{br-sq:gop}{article}{
   author={Brin, Matthew G.},
   author={Squier, Craig C.},
   title={Groups of piecewise linear homeomorphisms of the real line},
   journal={Invent. Math.},
   volume={79},
   date={1985},
   number={3},
   pages={485--498}
}

\bib{bro:fpo}{article}{
   author={Brown, Kenneth S.},
   title={Finiteness properties of groups},
   booktitle={Proceedings of the Northwestern conference on cohomology of
   groups (Evanston, Ill., 1985)},
   journal={J. Pure Appl. Algebra},
   volume={44},
   date={1987},
   number={1-3},
   pages={45--75}
}

\bib{br-ge:ait}{article}{
   author={Brown, Kenneth S.},
   author={Geoghegan, Ross},
   title={An infinite-dimensional torsion-free ${\rm FP}_{\infty }$ group},
   journal={Invent. Math.},
   volume={77},
   date={1984},
   number={2},
   pages={367--381}
}

\bib{b-f-m-w-z:tbt}{article}{
   author={Bux, Kai-Uwe},
   author={Fluch, Martin G.},
   author={Marschler, Marco},
   author={Witzel, Stefan},
   author={Zaremsky, Matthew C. B.},
   title={The braided Thompson's groups are of type $F_\infty$},
   journal={{\tt arXiv:1210.2931v2}},
   date={2014}
}

\bib{cis:lcd}{article}{
   author={Cisinski, Denis-Charles},
   title={La classe des morphismes de Dwyer n'est pas stable par retractes},
   language={French, with English summary},
   journal={Cahiers Topologie G\'eom. Diff\'erentielle Cat\'eg.},
   volume={40},
   date={1999},
   number={3},
   pages={227--231}
}

\bib{dw-ka:csl}{article}{
   author={Dwyer, W. G.},
   author={Kan, D. M.},
   title={Calculating simplicial localizations},
   journal={J. Pure Appl. Algebra},
   volume={18},
   date={1980},
   number={1},
   pages={17--35}
}

\bib{dyd:asp}{article}{
   author={Dydak, Jerzy},
   title={A simple proof that pointed FANR-spaces are regular fundamental
   retracts of ANR's},
   language={English, with Russian summary},
   journal={Bull. Acad. Polon. Sci. S\'er. Sci. Math. Astronom. Phys.},
   volume={25},
   date={1977},
   number={1},
   pages={55--62}
}

\bib{far:aop}{article}{
   author={Farley, Daniel S.},
   title={Actions of picture groups on CAT(0) cubical complexes},
   journal={Geom. Dedicata},
   volume={110},
   date={2005},
   pages={221--242}
}

\bib{far:fac}{article}{
   author={Farley, Daniel S.},
   title={Finiteness and $\rm CAT(0)$ properties of diagram groups},
   journal={Topology},
   volume={42},
   date={2003},
   number={5},
   pages={1065--1082}
}

\bib{far:haf}{article}{
   author={Farley, Daniel S.},
   title={Homological and finiteness properties of picture groups},
   journal={Trans. Amer. Math. Soc.},
   volume={357},
   date={2005},
   number={9},
   pages={3567--3584 (electronic)}
}

\bib{fa-hu:fpo}{article}{
   author={Farley, Daniel S.},
   author={Hughes, Bruce},
   title={Finiteness properties of some groups of local similarities},
   journal={Proc. Edinb. Math. Soc. (2)},
   volume={58},
   date={2015},
   number={2},
   pages={379--402}
}

\bib{fie:act}{article}{
   author={Fiedorowicz, Zbigniew},
   title={A counterexample to a group completion conjecture of J. C. Moore},
   journal={Algebr. Geom. Topol.},
   volume={2},
   date={2002},
   pages={33--35 (electronic)}
}

\bib{fi-le:aac}{article}{
   author={Fiore, Marcelo},
   author={Leinster, Tom},
   title={An abstract characterization of Thompson's group $F$},
   journal={Semigroup Forum},
   volume={80},
   date={2010},
   number={2},
   pages={325--340}
}

\bib{f-l-s:eco}{article}{
   author={Fiore, Thomas M.},
   author={L{\"u}ck, Wolfgang},
   author={Sauer, Roman},
   title={Euler characteristics of categories and homotopy colimits},
   journal={Doc. Math.},
   volume={16},
   date={2011},
   pages={301--354}
}

\bib{f-m-w-z:tbg}{article}{
   author={Fluch, Martin G.},
   author={Marschler, Marco},
   author={Witzel, Stefan},
   author={Zaremsky, Matthew C. B.},
   title={The Brin-Thompson groups $sV$ are of type $\text{F}_\infty$},
   journal={Pacific J. Math.},
   volume={266},
   date={2013},
   number={2},
   pages={283--295}
}

\bib{fr-he:shi}{article}{
   author={Freyd, Peter},
   author={Heller, Alex},
   title={Splitting homotopy idempotents. II},
   journal={J. Pure Appl. Algebra},
   volume={89},
   date={1993},
   number={1-2},
   pages={93--106}
}
	
\bib{ga-zi:cof}{book}{
   author={Gabriel, P.},
   author={Zisman, M.},
   title={Calculus of fractions and homotopy theory},
   series={Ergebnisse der Mathematik und ihrer Grenzgebiete, Band 35},
   publisher={Springer-Verlag New York, Inc., New York},
   date={1967}
}

\bib{geo:tmi}{book}{
   author={Geoghegan, Ross},
   title={Topological methods in group theory},
   series={Graduate Texts in Mathematics},
   volume={243},
   publisher={Springer, New York},
   date={2008}
}

\bib{go-ja:sht}{book}{
   author={Goerss, Paul G.},
   author={Jardine, John F.},
   title={Simplicial homotopy theory},
   series={Modern Birkh\"auser Classics},
   publisher={Birkh\"auser Verlag},
   place={Basel},
   date={2009}
}

\bib{gu-sa:dg}{article}{
   author={Guba, Victor},
   author={Sapir, Mark},
   title={Diagram groups},
   journal={Mem. Amer. Math. Soc.},
   volume={130},
   date={1997},
   number={620}
}

\bib{hug:lsa}{article}{
   author={Hughes, Bruce},
   title={Local similarities and the Haagerup property},
   note={With an appendix by Daniel S. Farley},
   journal={Groups Geom. Dyn.},
   volume={3},
   date={2009},
   number={2},
   pages={299--315}
}

\bib{jo-st:btc}{article}{
   author={Joyal, Andr{\'e}},
   author={Street, Ross},
   title={Braided tensor categories},
   journal={Adv. Math.},
   volume={102},
   date={1993},
   number={1},
   pages={20--78}
}

\bib{jo-st:tgo}{article}{
   author={Joyal, Andr{\'e}},
   author={Street, Ross},
   title={The geometry of tensor calculus. I},
   journal={Adv. Math.},
   volume={88},
   date={1991},
   number={1},
   pages={55--112}
}

\bib{lue:tto}{article}{
   author={L{\"u}ck, Wolfgang},
   title={The type of the classifying space for a family of subgroups},
   journal={J. Pure Appl. Algebra},
   volume={149},
   date={2000},
   number={2},
   pages={177--203}
}

\bib{mac:cft}{book}{
   author={Mac Lane, Saunders},
   title={Categories for the working mathematician},
   series={Graduate Texts in Mathematics},
   volume={5},
   edition={2},
   publisher={Springer-Verlag, New York},
   date={1998}
}

\bib{mcd:otc}{article}{
   author={McDuff, Dusa},
   title={On the classifying spaces of discrete monoids},
   journal={Topology},
   volume={18},
   date={1979},
   number={4},
   pages={313--320}
}

\bib{mon:gop}{article}{
   author={Monod, Nicolas},
   title={Groups of piecewise projective homeomorphisms},
   journal={Proc. Natl. Acad. Sci. USA},
   volume={110},
   date={2013},
   number={12},
   pages={4524--4527}
}

\bib{par:ucc}{article}{
   author={Par{\'e}, Robert},
   title={Universal covering categories},
   journal={Rend. Istit. Mat. Univ. Trieste},
   volume={25},
   date={1993},
   number={1-2},
   pages={391--411 (1994)}
}

\bib{qui:hak}{article}{
   author={Quillen, Daniel},
   title={Higher algebraic $K$-theory. I},
   book={
      publisher={Springer},
      place={Berlin},
   },
   date={1973},
   pages={85--147. Lecture Notes in Math., Vol. 341}
}

\bib{qui:hpo}{article}{
   author={Quillen, Daniel},
   title={Homotopy properties of the poset of nontrivial $p$-subgroups of a
   group},
   journal={Adv. in Math.},
   volume={28},
   date={1978},
   number={2},
   pages={101--128}
}

\bib{sel:aso}{article}{
   author={Selinger, P.},
   title={A survey of graphical languages for monoidal categories},
   conference={
      title={New structures for physics},
   },
   book={
      series={Lecture Notes in Phys.},
      volume={813},
      publisher={Springer, Heidelberg},
   },
   date={2011},
   pages={289--355}
}

\bib{spa:at}{book}{
   author={Spanier, Edwin H.},
   title={Algebraic topology},
   publisher={McGraw-Hill Book Co., New York-Toronto, Ont.-London},
   date={1966}
}

\bib{squ:tha}{article}{
   author={Squier, Craig C.},
   title={The homological algebra of Artin groups},
   journal={Math. Scand.},
   volume={75},
   date={1994},
   number={1},
   pages={5--43}
}

\bib{ste:gop}{article}{
   author={Stein, Melanie},
   title={Groups of piecewise linear homeomorphisms},
   journal={Trans. Amer. Math. Soc.},
   volume={332},
   date={1992},
   number={2},
   pages={477--514}
}

\bib{tho:caa}{article}{
   author={Thomason, R. W.},
   title={Cat as a closed model category},
   journal={Cahiers Topologie G\'eom. Diff\'erentielle},
   volume={21},
   date={1980},
   number={3},
   pages={305--324}
}

\bib{tho:hci}{article}{
   author={Thomason, R. W.},
   title={Homotopy colimits in the category of small categories},
   journal={Math. Proc. Cambridge Philos. Soc.},
   volume={85},
   date={1979},
   number={1},
   pages={91--109}
}

\bib{thesis}{thesis}{
   author={Thumann, Werner},
   title={Operad groups},
   type={PhD thesis, KIT Karlsruhe},
   date={2015},
   note={{\tt urn:nbn:de:swb:90-454145}}
}

\bib{wei:wdt}{article}{
   author={Weiss, Michael},
   title={What does the classifying space of a category classify?},
   journal={Homology Homotopy Appl.},
   volume={7},
   date={2005},
   number={1},
   pages={185--195}
}

\end{biblist}
\end{bibdiv}
\end{document}